\documentclass[12pt]{article}
\usepackage{fullpage, amsmath, amsthm,amsfonts,url,amssymb,stmaryrd, mathrsfs}

\usepackage[pdftex]{color}
\usepackage[all]{xy}

\newtheorem{theorem}{Theorem}[subsection]
\newtheorem*{theoremstar}{Theorem}
\newtheorem{lemma}[theorem]{Lemma}
\newtheorem{cor}[theorem]{Corollary}
\newtheorem{conj}[theorem]{Conjecture}
\newtheorem{prop}[theorem]{Proposition}

\theoremstyle{definition}
\newtheorem{defn}[theorem]{Definition}
\newtheorem{hypothesis}[theorem]{Hypothesis}
\newtheorem{example}[theorem]{Example}
\newtheorem{remark}[theorem]{Remark}
\newtheorem{convention}[theorem]{Convention}
\newtheorem{construction}[theorem]{Construction}
\newtheorem{notation}[theorem]{Notation}

\numberwithin{equation}{theorem}

\newcommand{\calE}{\mathcal{E}}

\newcommand{\calL}{\mathcal{L}}

\newcommand{\calO}{\mathcal{O}}
\newcommand{\calR}{\mathcal{R}}

\newcommand{\gothm}{\mathfrak{m}}

\newcommand{\gothp}{\mathfrak{p}}

\newcommand{\BB}{\mathbb{B}}
\newcommand{\QQ}{\mathbb{Q}}
\newcommand{\NN}{\mathbb{N}}
\newcommand{\RR}{\mathbb{R}}
\newcommand{\ZZ}{\mathbb{Z}}
\newcommand{\dual}{\vee}
\newcommand{\bbe}{\mathbf{e}}
\newcommand{\bbf}{\mathbf{f}}
\newcommand{\bbG}{\mathbb{G}}

\newcommand{\bbm}{\mathbf{m}}
\newcommand{\bbv}{\mathbf{v}}
\newcommand{\bbw}{\mathbf{w}}

\newcommand{\bbD}{\mathbf{D}}
\newcommand{\bbL}{\mathbf{L}}

\newcommand{\bbR}{\mathbf{R}}
\newcommand{\rmA}{\mathrm{A}}
\newcommand{\rmC}{\mathrm{C}}
\newcommand{\rmD}{\mathrm{D}}
\newcommand{\scrC}{\mathscr{C}}

\newcommand{\an}{\mathrm{an}}
\newcommand{\alg}{\mathrm{alg}}

\newcommand{\cont}{\mathrm{cont}}
\newcommand{\cris}{\mathrm{cris}}

\newcommand{\Dfm}{\mathbf{Dfm}}
\newcommand{\dR}{\mathrm{dR}}
\newcommand{\even}{\mathrm{even}}
\newcommand{\fil}{\mathrm{fil}}
\newcommand{\id}{\mathrm{id}}
\newcommand{\Ind}{\mathrm{Ind}}

\newcommand{\Iw}{\mathrm{Iw}}
\newcommand{\LT}{\mathrm{LT}}
\newcommand{\odd}{\mathrm{odd}}
\newcommand{\perf}{\mathrm{perf}}
\newcommand{\pst}{\mathrm{pst}}
\def\Qp{\QQ_p}
\newcommand{\red}{\mathrm{red}}
\newcommand{\RHom}{\bbR\mathrm{Hom}}
\newcommand{\rig}{\mathrm{rig}}

\newcommand{\gr}{\mathrm{gr}}
\newcommand{\Sen}{\mathrm{Sen}}
\newcommand{\Ta}{\mathrm{Ta}}
\newcommand{\tors}{\mathrm{tors}}
\newcommand{\ptors}{{p\text{-tors}}}
\newcommand{\Tr}{\mathrm{Tr}}
\newcommand{\whotimes}{\mathop{\widehat{\otimes}}}
\newcommand{\Zp}{\ZZ_p}
\newcommand{\Lotimes}{\mathop{\stackrel{\bbL}{\otimes}}}

\DeclareMathOperator{\Adj}{Adj}
\DeclareMathOperator{\Cone}{Cone}
\DeclareMathOperator{\Coker}{Coker}
\DeclareMathOperator{\Fib}{Fib}
\DeclareMathOperator{\Fitt}{Fitt}
\DeclareMathOperator{\Frac}{Frac}
\DeclareMathOperator{\Gal}{Gal}
\DeclareMathOperator{\Ker}{Ker}
\DeclareMathOperator{\Hom}{Hom}
\DeclareMathOperator{\Image}{Image}
\DeclareMathOperator{\Max}{Max}
\DeclareMathOperator{\ord}{ord}
\DeclareMathOperator{\rank}{rank}
\DeclareMathOperator{\Res}{Res}
\DeclareMathOperator{\Spec}{Spec}
\DeclareMathOperator{\Supp}{Supp}
\DeclareMathOperator{\Tor}{Tor}

\begin{document}

\title{Cohomology of arithmetic families of $(\varphi, \Gamma)$-modules}
\author{Kiran S. Kedlaya, Jonathan Pottharst, and Liang Xiao}

\maketitle
\begin{abstract}
We prove the finiteness and compatibility with base change of the
$(\varphi, \Gamma)$-cohomology and the Iwasawa cohomology of
arithmetic families of $(\varphi, \Gamma)$-modules.  Using this
finiteness theorem, we show that a family of Galois representations
that is densely pointwise refined in the sense of Mazur is actually
trianguline as a family over a large subspace.  In the case of the
Coleman-Mazur eigencurve, we determine the behavior at all points.
\end{abstract}

\tableofcontents

\section{Introduction}

One of the most powerful tools available for the study of $p$-adic
Galois representations is the theory of
$(\varphi, \Gamma)$-modules, which provides an equivalence of
categories between Galois representations and
modules over a certain somewhat simpler group algebra.
To name just one example, $(\varphi, \Gamma)$-modules
form a key intermediate step in the $p$-adic local Langlands correspondence for
the group
$\mathrm{GL}_2(\QQ_p)$, as discovered by Colmez \cite{colmez-kirillov}.
One key feature of the theory, which plays a crucial role in the
previous example, is that
the full category of $(\varphi, \Gamma)$-modules is strictly larger
than the subcategory
on which the equivalence to Galois representations take place (namely
the subcategory of \emph{\'etale}
$(\varphi, \Gamma)$-modules); this makes it possible (and indeed quite
frequent) for an \emph{irreducible}
$p$-adic Galois representation to become \emph{reducible} when carried
to the full category of
$(\varphi, \Gamma)$-modules. This is explained largely using the
theory of slope filtrations for $\varphi$-modules
(as developed for instance in \cite{kedlaya-relative}) and has
important applications in Iwasawa theory
especially in cases of nonordinary reduction \cite{pottharst1, pottharst2}.

One important feature of the theory of $(\varphi, \Gamma)$-modules is
its compatibility with variation in
analytic families. On the side of Galois representations, this means
working with continuous actions of
the Galois group of a finite extension of $\QQ_p$ not on a
finite-dimensional $\QQ_p$-vector space, but
on a finite projective module over a $\QQ_p$-affinoid algebra
(or more globally, a locally free coherent
sheaf on a rigid analytic space over $\QQ_p$).
The work of Berger and Colmez \cite{berger-colmez} provides a functor
from such Galois representations to
a certain category of \emph{relative $(\varphi, \Gamma)$-modules} (or
more globally, to an \emph{arithmetic
family of $(\varphi, \Gamma)$-modules} in the language of
\cite{kiran-liu}). In contrast to the usual theory,
however, this functor is fully faithful but not essentially
surjective, even onto the \'etale objects; this somewhat complicates
the relative theory. Another complication is that the theory of slope
filtrations does not extend in a completely satisfactory way to
families; see \cite{kiran-liu} for discussion of some of the
difficulties.

This paper primarily concerns itself with the relative analogue of the theory of
\emph{Galois cohomology} for
$(\varphi, \Gamma)$-modules. The mechanism for computing Galois
cohomology of $p$-adic Galois representations
on the side of $(\varphi, \Gamma)$-modules was introduced by Herr, and
it was later shown by Liu \cite{liu}
(using slope filtrations) that this extends in a satisfactory way to
all $(\varphi, \Gamma)$-modules when the family is reduced to a point.
This includes analogues of some basic results of Tate (finite
dimensionality of cohomology,
an Euler characteristic formula, and Tate local duality) and is
important for such applications as \cite{pottharst1}.
This makes clear the need for an analogue of these results in
families, but a straightforward imitation of
\cite{liu} is infeasible for the sorts of reasons described in the
previous paragraph.
The arguments used form a combination of several different strategies,
which are described in the discussion of the structure of the paper below.

In addition to results on relative Galois cohomology, we also obtain
relative versions of some results
on \emph{Iwasawa cohomology}. Iwasawa cohomology, computed using the
$\psi$ operator
(a left inverse of the Frobenius operator $\varphi$), is used in
\cite{pottharst2} to provide a structural
link between $(\varphi, \Gamma)$-modules and Iwasawa theory, giving a
way to clarify some previously
mysterious phenomena on the latter side.

Arguments taken from \cite{pottharst2} also show that our finiteness
results imply compatibility
with base change for both Galois cohomology and Iwasawa cohomology. As
one might expect, these results must be stated at the
level of derived categories except in the case of a \emph{flat} base
change.

We summarize the preceding results in the following theorem, referring to the body of the paper for the relevant terminology.

\begin{theoremstar}
Let $K$ be a finite extension of $\Qp$, let $A$ be a $\Qp$-affinoid
algebra, and let $M$ be (the global sections of) a locally free
coherent sheaf over the relative Robba ring $\calR_A(\pi_K)$ with
continuous semilinear $(\varphi,\Gamma_K)$-action, such that the
induced linear map $\varphi^*M \to M$ is an isomorphism.  Then $M$ is
finitely generated projective over $\calR_A(\pi_K)$, its Galois
cohomology $\rmC_{\varphi,\gamma_K}^\bullet(M)$ is quasi-isomorphic to
$\rmC_{\psi,\gamma_K}^\bullet(M)$ and lies in $\bbD_\perf^\flat(A)$,
and its Iwasawa cohomology $\rmC_\psi^\bullet(M)$ lies in
$\bbD_\perf^\flat(\calR_A^\infty(\Gamma_K))$.  Variants of Tate local
duality and the Euler-Poincar\'e formula hold for
$\rmC_{\varphi,\gamma_K}^\bullet(M)$ and $\rmC_\psi^\bullet(M)$.

If $A \to B$ is a homomorphism of $\Qp$-affinoid algebras, then the natural maps
\[
\rmC_{\varphi,\gamma_K}^\bullet(M) \Lotimes_A B \to \rmC_{\varphi,\gamma_K}^\bullet(M \whotimes_A B) \quad \text{and} \quad \rmC_\psi^\bullet(M) \Lotimes_{\calR_A^\infty(\Gamma_K)} \calR_B^\infty(\Gamma_K) \to \rmC_\psi^\bullet(M \whotimes_A B)
\]
are isomorphisms in the derived category.
\end{theoremstar}

While many applications of these results are expected, we limit
ourselves in this paper to just one application: the problem of
\emph{global triangulations} of families of $(\varphi,
\Gamma)$-modules, a triangulation being a useful notion due to Colmez
\cite{colmez}. As noted earlier, it occurs quite frequently that
irreducible $p$-adic Galois representations become more reducible when
pushed into the full category of $(\varphi, \Gamma)$-modules; our
results show that this phenomenon occurs frequently in families as
well.

One primary area of application is to
the study of \emph{eigenvarieties}, which parametrize $p$-adic
analytic families of automorphic forms.  
In this case, we prove the existence of global triangulations after modifying the eigenvarieties along some proper birational morphisms.  We also prove that the representation is trianguline at \emph{every} point on the eigenvariety.
In the special case of the Coleman-Mazur eigencurve, we prove the existence of a global triangulation after resolving the singularities; we prove that if a point corresponds to an overconvergent modular form which is in the image of the $\theta^{k-1}$-operator of Coleman, then the global triangulation does \emph{not} give a saturated triangulation at that point.

As explained in \cite{pottharst1}, the reducibility of trianguline
families can be exploited to form analogues of Greenberg's families of
Selmer groups for \emph{nonordinary} families of Galois
representations; the finiteness of cohomology also intervenes in a
crucial way to guarantee their well-behavedness.  Although we do not
explain the following point in this paper, it should be noted that
given our work on the eigencurve, Nekov\'a\v{r}'s methods then
generalize immediately to show that the validity of the Parity
Conjecture is constant along the irreducible components of the
eigencurve; this reduces the conjecture immediately to the (presently
unknown) claim that each irreducible component contains a classical
weight two point of noncritical slope.

Some similar results for representations of $G_{\Qp}$ have recently
been announced by other authors.  Using a significant technical
improvement of Kisin's original method \cite{R:Kisin} of interpolating
crystalline periods, Liu has developed techniques for showing that
refined families admit triangulations over Zariski-dense open sets,
similar to our Example~\ref{E:refined family}.  As an application, he
has recovered essentially our Proposition~\ref{P:eigencurve} in
\cite[Proposition~5.4.2]{R:Liu2} (in which the cases (1) and (2) treat
nonordinary and ordinary points, respectively); see also
\cite[Theorem~5.4.3]{R:Liu2} for some information at singularities of
the eigencurve.  Also, Hellmann \cite[Corollary~1.5]{hellmann} has
obtained a result on the triangulation of eigenvarieties for definite
unitary groups over imaginary quadratic fields, following a strategy
suggested by the second author: to construct a universal family of
(rigidified) trianguline $(\varphi,\Gamma)$-modules as in
\cite{chenevier}, then use automorphic data to construct a map from
the eigenvariety to this moduli space.

\subsubsection*{Structure of the paper}

In \S\ref{sec:family-phi-Gamma-modules}, we carry out some preliminary
arguments
relating sheaves on relative annuli to \emph{coadmissible} modules in
the sense of Schneider and Teitelbaum
\cite{schneider-teitelbaum}. A key issue is to establish finite
generation for certain such modules;
for instance, we obtain (Proposition~\ref{P:phi module v.s. phi bundle})
an equivalence between the two flavors of relative $(\varphi,
\Gamma)$-modules
treated in \cite{kiran-liu} (one defined using modules over a relative
Robba ring, the other defined using vector bundles over a relative
annulus).

In \S\ref{S:psi operator}, we make a careful study of the
$\psi$ operator, a one-sided inverse of the Frobenius operator $\varphi$.
Some of these arguments (such as the finiteness of the cokernel of
$\psi-1$)
are straightforward generalizations of arguments appearing in the
usual $(\varphi, \Gamma)$-module theory.

In \S\ref{S:finiteness of cohomology}, we establish the formal
framework for our results, and obtain
relative versions of the results on Galois cohomology from \cite{liu}
and the results of Iwasawa
cohomology from \cite{pottharst2}, plus base change in both settings,
modulo
a key finiteness theorem  for Iwasawa cohomology
(Theorem~\ref{T:finite Iwasawa cohomology}).
Given this result, most of the proofs proceed by reducing to the case
of a point, for which
we appeal to the previously known results; in particular, we do not
give a new method for
proving any results from \cite{liu} or \cite{pottharst2}.

In \S\ref{S:proof}, we prove Theorem~\ref{T:finite Iwasawa
cohomology}. After some initial simplifications,
the argument consists of two main steps. The first step is to
establish finiteness of Iwasawa cohomology in degree $1$
assuming vanishing of outer $(\varphi,\Gamma)$-cohomology, using the existence of the
duality pairing in general and its
perfectness at points. The second step is to reduce to this case using
a d\'evissage argument akin to those
in \cite{liu}, in which one carefully constructs an extension of the
original $(\varphi, \Gamma)$-module
(using slope filtrations) in order to kill off undesired cohomology in degree $2$.

In \S\ref{S:triangulation}, we give some applications of our results
to triangulations in families
of $(\varphi, \Gamma)$-modules for a finite extension $K$ over $\Qp$.  We also derive some explicit
consequences for eigenvarieties and
in particular for the Coleman-Mazur eigencurve.  In addition, we prove that all rank one arithmetic families of $(\varphi,\Gamma)$-modules are of character type (up to twisting by a line bundle on the base), answering in the affirmative a question of Bella\"iche \cite[\S3,~Question~1]{R:Bellaiche}.

\subsubsection*{Notation}
Throughout this paper, we fix a prime number $p$.  Set $\omega =
p^{-1/(p-1)}$.

For $H \subseteq G$ a subgroup of finite index and $M$ a
$\ZZ[H]$-module, we write $\Ind_H^G M$ for the induced
$\ZZ[G]$-module; it is canonically isomorphic to
$\Hom_{\ZZ[H]}(\ZZ[G], M)$.

We will exclusively work with affinoid and rigid analytic spaces in
the sense of Tate, rather than Berkovich.
The letter $A$ will always denote a $\QQ_p$-affinoid algebra;
we use $\Max(A)$ to denote the associated rigid analytic space.  For
$z \in \Max(A)$, we use $\gothm_z$ to denote the maximal ideal of $A$
at $z$ and put $\kappa_z = A /\gothm_z$.  For $M$ an $A$-module, we
write $M_z$ to mean $M/\gothm_zM$.

All Hom spaces consist of \emph{continuous} homomorphisms for the
relevant topologies (although we have no need to topologize the Hom
spaces themselves, except for when specified). We will sometimes use
the subscript ``cont" to emphasize this point.

We normalize the theory of slope filtrations so that a nonzero pure
submodule of an \'etale object has \emph{negative} slope.

\subsubsection*{Acknowledgements}
We thank Rebecca Bellovin, John Bergdall, Matthew Emerton, David
Geraghty, Ruochuan Liu, Zhiwei Yun, and Weizhe Zheng for helpful
discussions.  We thank Jo\"el Bella\"iche for suggesting that his
conjecture be attacked using the methods of this paper.  We thank the
anonymous referees for their careful reading of this paper and
numerous helpful suggestions.  Comments by Peter Schneider, Filippo
Nuccio, and Tadashi Ochiai led to improvements in the exposition of
this paper.  The first two authors would like to thank the University
of Chicago for its hospitality during their visits.  The third author
would like to thank Boston University and MIT for their hospitality
during his visits.  Kedlaya was supported by NSF (CAREER grant
DMS-0545904, grant DMS-1101343), DARPA (grant HR0011-09-1-0048), MIT
(NEC Fund, Cecil and Ida Green professorship), and UCSD (Stefan
E. Warschawski professorship).  Pottharst was supported by NSF
(MSPRF).

\section{Families of $(\varphi, \Gamma)$-modules}
\label{sec:family-phi-Gamma-modules}

In this section, we introduce the definition of a $(\varphi,
\Gamma)$-module over the Robba ring with coefficients in the affinoid
algebra $A$ over $\Qp$.  Our definition is algebraic in nature,
involving a finite projective module equipped with extra structures;
however, we show that it agrees with the definition of an
\emph{arithmetic family of $(\varphi,\Gamma)$-modules} over the Robba
ring in the sense of \cite{kiran-liu}, which involves a vector bundle
over a certain rigid analytic space.  We also recall the relationship
between $(\varphi, \Gamma)$-modules and families of Galois
representations following Berger-Colmez \cite{berger-colmez} and
Kedlaya-Liu \cite{kiran-liu}, as well as the formalism of Galois
cohomology for $(\varphi, \Gamma)$-modules following Herr and
especially Liu \cite{liu}.

\begin{convention}
Throughout this paper, all radii $r$ and $s$ are assumed to be
rational numbers, except possibly $r=\infty$.
\end{convention}

\subsection{Modules over relative discs and annuli}
\label{S:half-open}
In general, the module of global sections of a vector bundle over an open relative disc or a half-open relative annulus may not be finitely generated over the ring of analytic functions on the corresponding space.  We establish a criterion for the module of sections to be finitely generated or finite projective.

\begin{notation}
For $r>0$, define the \emph{$r$-Gauss norm} on $\Qp[T^{\pm1}]$ by the formula $\left|\sum_i a_iT^i\right|_r = \max_{i\in\ZZ} \{|a_i| \omega^{ir}\}$, where $a_i \in \Qp$.  This is a multiplicative nonarchimedean norm.

For $0 < s \leq r$, we write $\rmA^1[s,r]$ for the rigid analytic
annulus in the variable $T$ with radii $|T| \in [\omega^r, \omega^s]$;
its ring of analytic functions, denoted by $\calR^{[s,r]}$, is the
completion of $\Qp[T^{\pm1}]$ with respect to the norm
$|\cdot|_{[s,r]} = \max\{|\cdot|_r, |\cdot|_s\}$.  We also allow $r$
(but not $s$) to be $\infty$, in which case $\rmA^1[s,r]$ is
interpreted as the rigid analytic disc in the variable $T$ with radii
$|T| \leq \omega^s$; then $\calR^{[s,r]} = \calR^{[s, \infty]}$ is the
completion of $\Qp[T]$ with respect to $|\cdot|_s$.  We treat
$[s,\infty]$ as a closed interval when referring to it.

Let $\calR_A^{[s,r]}$ denote the ring of analytic functions on the relative annulus (or disc if $r=\infty$) $\Max(A) \times \rmA^1[s,r]$; its ring of analytic functions is $\calR_A^{[s,r]} = \calR^{[s,r]} \widehat \otimes_{\Qp} A$.
Put $\calR^r_A = \bigcap_{0 < s \leq r} \calR^{[s,r]}_A$ and 
$\calR_A = \bigcup_{0<r} \calR^r_A$.
\end{notation}

\begin{hypothesis}
In this subsection, we fix $r_0$ to be a positive rational number or $\infty$.
\end{hypothesis}

For any decreasing sequence of positive (rational) numbers $r_1, r_2, \dots$ tending to zero with $r_0 > r_1$, we have
\begin{itemize}
\item[(i)] $\calR_A^{[r_n,r_0]}$ and $\calR_A^{[r_n,\infty]}$ are noetherian Banach $A$-algebras, and
\item[(ii)] $\calR_A^{[r_{n+1},r_0]} \to \calR_A^{[r_n,r_0]}$ and $\calR_A^{[r_{n+1},\infty]} \to \calR_A^{[r_n,\infty]}$ are flat and have topologically dense images,
\end{itemize}
for any $n \geq 0$.  This implies that $\calR_A^{r_0}$ is a Fr\'echet-Stein algebra in the sense of \cite[Section~3]{schneider-teitelbaum}.  We recall the terminology therein as follows.

\begin{defn}
A \emph{coherent sheaf} over $\calR_A^{r_0}$ consists of one finite  module $M^{[s,r]}$ over each ring $\calR_A^{[s,r]}$ with $0<s\leq r \leq r_0$, together with isomorphisms $M^{[s',r']}\cong M^{[s,r]} \otimes_{\calR_A^{[s,r]}} \calR_A^{[s',r']}$ for all $0< s \leq s' \leq r'\leq r \leq r_0$ satisfying the obvious cocycle conditions.  Its \emph{module of global sections} is $M = \varprojlim_{s \to 0^+} M^{[s,r_0]}$.  An $\calR_A^{r_0}$-module is \emph{coadmissible} if it occurs as the module of global sections of some coherent sheaf.

A \emph{vector bundle} over $\calR_A^{r_0}$ is a coherent sheaf $(M^{[s,r]})$ where each $M^{[s,r]}$ is flat over $\calR_A^{[s,r]}$.
\end{defn}

We list a few basic facts from \cite[Section~3]{schneider-teitelbaum}. (Strictly speaking, \cite[Section~3]{schneider-teitelbaum} only treats the case $r=r_0$, but one can deduce the results for general $r$ using standard techniques from \cite{BGR}.)

\begin{lemma}
\label{L:schneider-teitelbaum}
Let $(M^{[s,r]})$ be a coherent sheaf over $\calR_A^{r_0}$ with module of global sections $M$.
\begin{itemize}
\item[(1)] For any $0 < s\leq r \leq r_0$, $M$ (resp. $M[\frac 1T]$ if $r_0 =\infty$ and $r\neq \infty$) is dense in $M^{[s,r]}$.
\item[(2)] We have $\bbR^i\varprojlim_{s\to 0^+} M^{[s,r_0]} = 0$ for $i\geq 1$.
\item[(3)] We have a natural isomorphism $M \otimes_{\calR_A^{r_0}} \calR_A^{[s,r]} \cong M^{[s,r]}$ for any $0<s \leq r\leq r_0$.
\item[(4)] The ring $\calR_A^{[s,r]}$ is flat over $\calR_A^{r_0}$ for any $0<s \leq r\leq r_0$.
\item[(5)] The kernel and cokernel of an arbitrary $\calR_A^{r_0}$-linear map between coadmissible $\calR_A^{r_0}$-modules are coadmissible.
\item[(6)] Any finitely generated ideal of $\calR_A^{r_0}$ is coadmissible, or more generally, any finitely generated $\calR_A^{r_0}$-submodule of a coadmissible module is coadmissible.
\item[(7)] Any finitely presented $\calR_A^{r_0}$-module is coadmissible.
\end{itemize}

\end{lemma}

\begin{cor}
\label{C:R_A flat over A}
For any element $t \in \calR_{\Qp}^{r_0}$, the ring $\calR_A^{r_0}/t$
is flat over $A$.  In particular, $\calR^{r_0}_A$, $\calR_A/t$, and
$\calR_A$ are flat over $A$.
\end{cor}
\begin{proof}
Granted the claim for $\calR_A^{r_0}/t$, the claim for $\calR_A^{r_0}$
follows by taking $t=0$, and the claims with $\calR_A$ in place of
$\calR_A^{r_0}$ follow by taking the direct limit on $r_0$.  So we
prove the first statement.

Let $N \to N'$ be an injective morphism of finite $A$-modules.  Since $A$ is noetherian, $N \otimes_A \calR_A^{r_0}/t$ and $N'\otimes_A \calR_A^{r_0}/t$ are finitely presented and hence coadmissible by Lemma~\ref{L:schneider-teitelbaum}(7).  To check the injectivity of $N \otimes_A \calR_A^{r_0}/t \to N'\otimes_A \calR_A^{r_0}/t$, by Lemma~\ref{L:schneider-teitelbaum}(2), it suffices to check the injectivity of $N \otimes_A \calR_A^{[s,r]}/t \to N'\otimes_A \calR_A^{[s,r]}/t$ for any $0<s<r\leq r_0$.  When $t \neq 0$, 
$\calR_A^{[s,r]}/t$ is in fact finite and free over $A$, so the injectivity is obvious.  When $t=0$, we take a Schauder basis of $\calR_{\Qp}^{[s,r]}$ over $\Qp$ to identify $\calR_{\Qp}^{[s,r]}$ with the completed direct sum $\widehat \oplus_{i\in I}\Qp e_i$, where $(e_i)_{i \in I}$ form a potentially orthonormal basis.  Under this identification, $N\otimes_A \calR_A^{[s,r]} \cong \widehat \oplus_{i\in I}N$ and $N' \otimes_A \calR_A^{[s,r]} \cong \widehat \oplus_{i\in I}N'$.  Since $N \to N'$ is injective, so is $\widehat \oplus_{i\in I}N \to \widehat \oplus_{i\in I}N'$.  The corollary follows.
\end{proof}

\begin{lemma}
\label{L:fp+bundle=>finite-proj}
For a vector bundle $(M^{[s,r]})$ over $\calR_A^{r_0}$, its module of global sections $M$ is a finite projective $\calR_A^{r_0}$-module if and only if $M$ is finitely presented.
\end{lemma}
\begin{proof}
A module over a commutative ring is finite projective if and only if it is finitely presented and flat \cite[Corollary~of~Theorem~7.12]{matsumura}.  It thus suffices to assume that $M$ is finitely presented and prove that it is flat.  We need to prove that, for any finitely generated ideal $I$ of $\calR_A^{r_0}$, the natural map $I \otimes_{\calR_A^{r_0}} M \to M$ is injective.

We write $M$ as the cokernel of an $\calR_A^{r_0}$-linear homomorphism $f: (\calR_A^{r_0})^{\oplus m} \to (\calR_A^{r_0})^{\oplus n}$.  We can then realize $I \otimes_{\calR_A^{r_0}} M$ as  the cokernel of the $\calR_A^{r_0}$-linear map $I \otimes f: I^{\oplus m} \to I^{\oplus n}$.  By Lemma~\ref{L:schneider-teitelbaum}(6), $I$ is coadmissible, and hence so is $I \otimes_{\calR_A^{r_0}} M$ by Lemma~\ref{L:schneider-teitelbaum}(5).   By Lemma~\ref{L:schneider-teitelbaum}(2), to check the injectivity of the natural map $I \otimes_{\calR_A^{r_0}} M \to M$, it suffices to check the injectivity of
\[
(I \otimes_{\calR_A^{r_0}} M) \otimes_{\calR_A^{r_0}} \calR_A^{[s,r_0]} \to M\otimes_{\calR_A^{r_0}} \calR_A^{[s,r_0]}
\]
for all $0<s \leq r_0$.
By the flatness in Lemma~\ref{L:schneider-teitelbaum}(3)(4), the map above becomes
\[
(I\cdot \calR_A^{[s,r_0]}) \otimes_{\calR_A^{[s,r_0]}} M^{[s,r_0]} \to M^{[s,r_0]}.
\]
It is injective because $M^{[s,r_0]}$ is a flat $\calR_A^{[s,r_0]}$-module.  This finishes the proof.
\end{proof}

\begin{cor}
\label{C:pointwise criterion for finite projectivity}
Assume that $A$ is reduced.
Let $M$ be a finitely presented $\calR_A^{r_0}$-module such that for any $z \in \Max(A) \times \rmA^1(0, r_0]$, $M/\gothm_zM$ has the same dimension.  Then $M$ is a finite projective $\calR_A^{r_0}$-module.
\end{cor}
\begin{proof}
For any $0<s\leq r\leq r_0$, $\calR_A^{[s,r]}$ is noetherian, and hence $M \otimes_{\calR_A^{r_0}}\calR_A^{[s,r]}$ is finite and flat by simple commutative algebra (see Lemma~\ref{L:constant dimension function implies flat}(1) below).  By Lemma~\ref{L:fp+bundle=>finite-proj}, $M$ is then a finite projective $\calR_A^{r_0}$-module.
\end{proof}

\begin{lemma}
\label{L:constant dimension function implies flat}
\begin{enumerate}
\item[(1)]
Let $B$ be a noetherian ring whose Jacobson radical is zero, or
equivalently, a reduced noetherian ring whose maximal ideals form a
dense subset of $\Spec(B)$. Let $M$ be a finitely generated
$B$-module.  If $z \mapsto \dim_{\kappa_z} M/\gothm_zM$ is a locally
constant function on the set of maximal ideals $z$ of $B$, then $M$ is
flat.
\item[(2)]
Let $A$ and $B$ be $\Qp$-affinoid algebras with $B$ \emph{reduced},
and $N$ a coherent sheaf on $\Max(A) \times \Max(B)$.  Assume that for
every $y \in \Max(B)$, $N/\gothm_yN$ is a finite projective $A
\otimes_{\Qp} \kappa_y$-module of rank depending only on the connected
component of $y$ in $\Spec(B)$.  Then $N$ is a finite flat $A \widehat
\otimes_{\Qp} B$-module.
\end{enumerate}
\end{lemma}
\begin{proof}
(1) Suppose $m_1,\ldots,m_r \in M$ reduce mod $\gothm_z$ to a basis,
and let $f: B^r \to M$ denote the map associated to the $m_i$.  It
suffices to show that $f$ is an isomorphism in a neighborhood of $z$
in $\Spec(B)$.  Consider the kernel $K$ and cokernel $K'$ of $f$.  By
Nakayama's lemma, the $m_i$ generate $M_{\gothm_z}$, so $K'_{\gothm_z}
= 0$.  But the support of $K'$ in $\Spec(B)$ is a closed subset, so in
a neighborhood of $z$ we have that $f$ is surjective.  Replacing
$\Spec(B)$ by a suitable affine neighborhood of $z$, we assume $f$ to
be surjective.  On the other hand, for all closed points $z' \in
\Spec(B)$ the surjectivity of $f_{z'}: (B/\gothm_{z'})^r \to
M/\gothm_{z'}M$ implies, by comparing $\kappa_{z'}$-dimensions (and
using the local constancy hypothesis), that $f_{z'}$ is also
injective.  Thus elements of $K$ have all entries in $B^r$ belonging
to every $\gothm_{z'}$, and hence to the Jacobson radical of $B$,
which was assumed to be zero.

(2) Since $N$ is finite over $A \widehat\otimes_{\Qp} B$, it is flat
if and only if for all $z \in \Max(A \widehat\otimes_{\Qp} B)$ one has
$N_{\gothm_z}$ flat over $(A \widehat\otimes_{\Qp} B)_{\gothm_z}$.
Given such $z$, let $y \in \Max(B)$ be such that $\gothm_y = \gothm_z
\cap B$.  Then since $B \to A \widehat\otimes_{\Qp} B$ is flat,
$B_{\gothm_y} \to (A \widehat\otimes_{\Qp} B)_{\gothm_y} \to (A
\widehat\otimes_{\Qp} B)_{\gothm_z}$ is flat, so \cite[0.10.2.5]{ega3}
shows that $N_{\gothm_z}$ is flat over $(A \widehat\otimes_{\Qp}
B)_{\gothm_z}$ if and only if both $N_{\gothm_z}/\gothm_yN_{\gothm_z}$
is flat over $(A \widehat\otimes_{\Qp} B)_{\gothm_z}/\gothm_y(A
\widehat\otimes_{\Qp} B)_{\gothm_z}$ and $N_{\gothm_z}$ is flat over
$B_{\gothm_y}$.  The first condition is immediate, because by
hypothesis $N/\gothm_yN$ is flat over $(A \widehat\otimes_{\Qp}
B)/\gothm_y(A \widehat\otimes_{\Qp} B)$, and then we may localize at
$\gothm_z$.  We next show that $N$ is flat over $B$, so that
$N_{\gothm_y}$ is flat over $B_{\gothm_y}$, from which it follows
easily that $N_{\gothm_z}$ is also flat over $B_{\gothm_y}$.

One can find a decreasing sequence of ideals $J_i$ of $A$ such that
$J_i/J_{i+1} \approx A/I_i$ for some radical ideal $I_i \subset A$,
where, say, $J_0 = A$ and $J_r = (0)$.  We will show by induction on
$i$ that each $J_i N / J_{i+1} N$ is flat over $B$ and, for all $y \in
\Max(B)$, $(J_{i+1} N) \otimes_B \kappa_y \cong J_{i+1}(N \otimes_B
\kappa_y)$; it follows from this that $N$ is flat over $B$.  There is
nothing to prove when $i=0$. Suppose that we have proved the claim for
$i-1$.  Applying $\otimes_B \kappa_y$ to the exact sequence $0 \to
J_{i+1} N \to J_i N \to (J_i N / J_{i+1} N) \to 0$ and then using the
inductive hypothesis gives
\begin{align*}
(J_i N / J_{i+1} N) \otimes_B \kappa_y
&\cong
(J_i N) \otimes_B \kappa_y / \Image((J_{i+1} N) \otimes_B \kappa_y) \\
&\cong
J_i(N \otimes_B \kappa_y) / J_{i+1} (N \otimes_B \kappa_y)
\cong (A/I_i) \otimes_A (N \otimes_B \kappa_y).
\end{align*}
The right hand side is a finite projective module over $A/I_i$ of constant
rank; because $A/I_i$ is reduced, the claim (1) then gives that $(J_i
N / J_{i+1} N)$ is flat over $(A/I_i) \widehat\otimes_{\Qp} B$ and
hence flat over $B$.  This moreover implies that the application of
$\otimes_B k_y$ above was exact and hence $(J_{i+1} N) \otimes_B
\kappa_y \cong J_{i+1}(N \otimes_B \kappa_y)$.
\end{proof}

\begin{defn}
\label{D:uniform-finite-presentation}
For $R$ a ring and $M$ an $R$-module, we say that $M$ is $n$-finitely generated (resp. \emph{$(m,n)$-finitely presented}) if there exists an $R$-linear exact sequence $R^{\oplus n} \to M \to 0$ (resp. $R^{\oplus m} \to R^{\oplus n} \to M \to 0$).

By an \emph{admissible cover} of $(0, r_0]$, we mean a cover of $(0, r_0]$ by closed intervals $\{[s_i, r_i]\}_{i \in \ZZ_{\geq 0}}$ with nonempty interiors, admitting a locally finite refinement.  Given such a cover of $(0,r_0]$, the system of rigid annuli $\rmA^1[s_i,r_i]$ of respective valuations $[s_i,r_i]$ gives an admissible affinoid cover of the rigid annulus $\rmA^1(0,r_0]$.  Thus, to define a coherent sheaf over $\calR_A^{r_0}$ is equivalent to specify $M^{[s_i,r_i]}$ for each $i$ together with the obvious compatibility data and conditions.

Let $(M^{[s,r]})$ be a coherent sheaf over $\calR_A^{r_0}$.  We say that the sheaf is \emph{uniformly finitely generated} if there exist $n \in \NN$ and an admissible cover $\{[s_i, r_i]\}_{i \in \ZZ_{\geq 0}}$ of $(0,r_0]$ such that each $M^{[s_i,r_i]}$ is $n$-finitely generated.
We say that the sheaf is \emph{uniformly finitely presented} if there exist $m,n \in \NN$ and an admissible cover $\{[s_i, r_i]\}_{i \in \ZZ_{\geq 0}}$ of $(0,r_0]$ such that each $M^{[s_i,r_i]}$ is $(m,n)$-finitely presented.
\end{defn}

Keeping track of the number of generators, a standard argument in homological algebra proves the following.

\begin{lemma}
\label{L:always-finitely presented}
Let $0 \to (M'^{[s,r]}) \to (M^{[s,r]}) \to (M''^{[s,r]}) \to 0$ be a short exact sequence  of coherent sheaves over $\calR_A^{r_0}$.
\begin{itemize}
\item If $(M'^{[s,r]})$ and one of the other coherent sheaves are uniformly finitely presented, so is the third one.
\item If both $(M^{[s,r]})$ and $(M''^{[s,r]})$ are uniformly finitely presented, then $(M'^{[s,r]})$ is uniformly finitely generated.
\end{itemize}
\end{lemma}

\begin{lemma}
\label{L:finite-generation-local}
Let $(M^{[s,r]})$ be a coherent sheaf over $\calR_A^{r_0}$ with module of global sections $M$.  Assume that there exist elements $\bbf_1, \dots, \bbf_n \in M$ and an admissible cover $\{[s_i, r_i]\}_{i \in \ZZ_{\geq 0}}$ of $(0, r_0]$ such that $\bbf_1, \dots, \bbf_n$ generate each $M^{[s_i, r_i]}$ as an $\calR_A^{[s_i,r_i]}$-module.  Then $\bbf_1, \dots, \bbf_n$ generate $M$ as an $\calR_A^{r_0}$-module.
\end{lemma}
\begin{proof}
Consider the $\calR_A^{r_0}$-linear map $f: (\calR_A^{r_0})^{\oplus m} \to M$ given by $\bbf_1, \dots, \bbf_n$.  The hypothesis of the lemma implies that
$
f \otimes \calR_A^{[s_i,r_i]}: (\calR_A^{[s_i,r_i]})^{\oplus n} \to M^{[s_i,r_i]}
$
is surjective for any $i$.  Hence $f$ is surjective by Lemma~\ref{L:schneider-teitelbaum}(2), yielding the lemma.
\end{proof}

\begin{lemma}
\label{L:perturbation-of-generators}
Let $N$ be a finite Banach module over a  Banach algebra $R$ over $\Qp$ with generators $\bbe_1, \dots, \bbe_n$.  Then there exists $\epsilon >0$ such that, for any elements $\bbe'_1, \dots, \bbe'_n \in N$ with $|\bbe_i - \bbe'_i|\leq \epsilon$ for any $i$, $N$ is generated by $\bbe'_1, \dots, \bbe'_n$.
\end{lemma}
\begin{proof}
We have a continuous surjective morphism $f: R^{\oplus n} \to M$ of Banach $R$-modules given by $\bbe_1, \dots, \bbe_n$.  By the Banach open mapping theorem (\cite[\S I.3.3,~Th\'eor\`eme 1]{bourbaki} or \cite[Proposition~8.6]{schneider}), there exists $\epsilon>0$ such that for any $\bbv \in M$ with $|\bbv|\leq \epsilon$, we can write $\bbv = x_1 \bbe_1+\cdots + x_n \bbe_n$ for some $x_i \in R$ with $|x_i|\leq \frac 12$.  Given the data in the lemma, we write $\bbe'_i-\bbe_i = \sum_{j=1}^n x_{ji}\bbe_j$ for $x_{ji} \in R$ with $|x_{ji}|\leq \frac 12$.  Set $X = (x_{ji})$.  The transition matrix for the two sets of elements is $1+X$, which is invertible.  Hence $N$ is generated by $\bbe'_1, \dots, \bbe'_n$.
\end{proof}

We are grateful to R.~Bellovin for showing us a more elementary proof
of the following result (see \cite[\S2.2]{bellovin}); we provide the
following argument anyway, because its method is used to prove
Lemma~\ref{L:variant finite generation} below, to which her argument
does not seem to apply.

\begin{prop}
\label{P:finite-generation}
Let $(M^{[s,r]})$ be a coherent sheaf over $\calR_A^{r_0}$ with module of global sections $M$.  

\begin{itemize}
\item[(1)]
The $\calR_A^{r_0}$-module $M$ is finitely generated if and only if $(M^{[s,r]})$ is uniformly finitely generated.
\item[(2)]
The $\calR_A^{r_0}$-module $M$ is finitely presented if and only if $(M^{[s,r]})$ is uniformly finitely  presented.
\item[(3)]
The $\calR_A^{r_0}$-module $M$ is finite projective if and only if $(M^{[s,r]})$ is a uniformly finitely presented vector bundle.
\end{itemize}
\end{prop}
\begin{proof}
(3) follows from combining (2) with Lemma~\ref{L:fp+bundle=>finite-proj}.  (2) follows from applying (1) twice and using Lemma~\ref{L:always-finitely presented} to pass on the conditions.
The sufficiency part of (1) is obvious.

We now prove the necessity part of (1).  Assume that each $(M^{[s,r]})$ is $n$-finitely generated.  We may refine and reorder the admissible cover $\{[s_i, r_i]\}_{i \in \ZZ_{\geq0}}$ so that the subcollections $I_\mathrm{odd} = \{[s_{2i+1}, r_{2i+1}]\}_{i \in \ZZ_{\geq0}}$ and $I_\mathrm{even} = \{[s_{2i}, r_{2i}]\}_{i \in \ZZ_{\geq0}}$ consist of pairwise \emph{disjoint} intervals in decreasing order.  We claim that there exist global sections $\bbf^\odd_1, \dots, \bbf^\odd_n \in M$ (resp. $\bbf^\even_1, \dots, \bbf^\even_n \in M$) generating each $M^{[s_i, r_i]}$ where $i$ is odd (resp. even).  Then the proposition follows from Lemma~\ref{L:finite-generation-local} above. We will only prove the claim for the even subcollection, as the proof for the odd one proceeds the same way.

For each $i$, we use $|\cdot|_{M, [s_i, r_i]}$ to denote a fixed Banach norm on $M^{[s_i,r_i]}$.
As in Definition~\ref{D:uniform-finite-presentation}, $M^{[s_{2i}, r_{2i}]}$ is generated by $n$ elements $\bbf_{2i, 1}, \dots, \bbf_{2i,n} \in M^{[s_{2i}, r_{2i}]}$.  
By the density of 
$M$ (resp. $M[\frac 1T]$ if $r_0=\infty$ and $r_{2i}\neq \infty$) in $M^{[s_{2i}, r_{2i}]}$ (Lemma~\ref{L:schneider-teitelbaum}(1))
 and Lemma~\ref{L:perturbation-of-generators},
we may assume that $\bbf_{2i, 1}, \dots, \bbf_{2i,n} \in M$ are global sections.  (When $r_0=\infty$ and $r_{2i}\neq \infty$, we first get the generators in $M[\frac 1T]$ and then we multiply them by some powers of $T$; they are still generators.)

 By Lemma~\ref{L:perturbation-of-generators} again, there exists $\epsilon_{2i}>0$ such that any elements $\bbf'_{2i, 1}, \dots, \bbf'_{2i,n} \in M$ with $|\bbf'_{2i,j} - \bbf_{2i,j}|_{M, [s_{2i}, r_{2i}]}\leq \epsilon_{2i}$ for all $j$ also generate  $M^{[s_{2i}, r_{2i}]}$.  For each $j$, we hope to define
\[
\bbf_j^\even = \alpha_{0,j} T^{\lambda_{0,j}} \bbf_{0,j}+ \alpha_{2,j}T^{\lambda_{2,j}}\bbf_{2, j} + \alpha_{4,j}T^{\lambda_{4,j}} \bbf_{4, j} 
 + \cdots,
\]
where $\alpha_{2i,j} \in \Qp$ and $\lambda_{2i,j} \in \ZZ_{\geq 0}$ are chosen inductively on $i$ with $\lambda_{0,j} = 0, \alpha_{0,j} = 1$ so that $\lambda_{2i,j} \geq \lambda_{2i-2,j}$, and
\begin{align}
\label{E:conditions on alpha and lambda1}
\textrm{for any }i'<i, \ &
\Big|\frac{\alpha_{2i,j}T^{\lambda_{2i,j}}}{\alpha_{2i',j}T^{\lambda_{2i',j}}}\bbf_{2i, j}\Big|_{M,[s_{2i'}, r_{2i'}]}\leq \epsilon_{2i'},\\
\label{E:conditions on alpha and lambda2}
\textrm{for any }i'<i, \ &
\Big|\frac{\alpha_{2i',j}T^{\lambda_{2i',j}}}{\alpha_{2i,j}T^{\lambda_{2i,j}}}\bbf_{2i', j}\Big|_{M,[s_{2i}, r_{2i}]}\leq \epsilon_{2i}, \textrm{ and }\\
\label{E:conditions on alpha and lambda3}
\textrm{for any }i''<2i-1 \textrm{ (even or odd)}, \ &
\big|\alpha_{2i,j}T^{\lambda_{2i,j}}\bbf_{2i,j}\big|_{M, [s_{i''}, r_{i''}]} \leq p^{-i}.
\end{align}
This would imply that the infinite sum defining $\bbf_j^\even$ converges and we have for any $i$,
\[
\Big|\bbf_{2i, j} - \frac 1{\alpha_{2i,j}T^{\lambda_{2i,j}}} \bbf_j^\even\Big|_{M, [s_{2i}, r_{2i}]} \leq \epsilon_{2i}.
\]
(This equation is also valid for $i=0$ because $\alpha_{0,j}=1$ and $\lambda_{0,j}=0$.)
These imply that $\bbf_1^\even, \dots, \bbf_n^\even$ generate $M^{[s_{2i}, r_{2i}]}$ over $\calR_A^{[s_{2i}, r_{2i}]}$.

We now prove that we can choose $\alpha_{2i,j}$ and $\lambda_{2i,j}$ satisfying \eqref{E:conditions on alpha and lambda1}--\eqref{E:conditions on alpha and lambda3}.  It suffices to meet the following conditions:
\begin{align*}
\textrm{for any }i'<i, \ &
|\alpha_{2i,j}|\cdot \omega^{s_{2i'}\lambda_{2i,j}} \leq
\omega^{s_{2i'}\lambda_{2i',j}} |\bbf_{2i, j}/\alpha_{2i',j}|^{-1}_{M,[s_{2i'}, r_{2i'}]} \cdot \epsilon_{2i'},\\
\textrm{for any }i'<i, \ &
|\alpha_{2i,j}|\cdot \omega^{r_{2i}\lambda_{2i,j}} \geq 
\omega^{r_{2i}\lambda_{2i',j}}|\alpha_{2i',j}\bbf_{2i', j}|_{M,[s_{2i}, r_{2i}]} / \epsilon_{2i}, \textrm{ and }\\
\textrm{for any }i''<2i-1, \ &
|\alpha_{2i,j}|\cdot \omega^{s_{i''}\lambda_{2i,j}} \leq
|\bbf_{2i, j}|^{-1}_{M,[s_{i''}, r_{2i''}]} \cdot p^{-i}.
\end{align*}
Note that for any $i'<i$ and $i''<2i-1$, we have $r_{2i}< s_{2i'}$ and $r_{2i}< s_{i''}$ because the intervals in $I_\even$ are pairwise disjoint.  Hence by making $\lambda_{2i,j} - \lambda_{2i',j}$ sufficiently large, we may find a choice of $\alpha_{2i,j} \in \Qp$ satisfying all conditions above.  This finishes the proof.
\end{proof}

\begin{remark}
The condition in (3) may be weakened to only require $(M^{[s,r]})$ to be a uniformly finitely generated vector bundle.  Indeed, $M^{[s,r]}$ is finite flat and hence finite projective, so if it is generated by $n$ elements, it is automatically $(n,n)$-finitely presented. Hence $(M^{[s,r]})$ is  uniformly finitely presented.
\end{remark}

We will later apply Proposition~\ref{P:finite-generation} in several
cases: $\varphi$-bundles and $\varphi$-modules over $\calR_A^{r_0}$
(Proposition~\ref{P:phi module v.s. phi bundle} and Lemma
\ref{L:extending isom and surj}, respectively), for a $(\varphi,
\Gamma)$-module $M$ over $\calR_A^{r_0}$, the structure theorems for
$M^{\psi=0}$ (Theorem~\ref{T:structure of D^psi=0}) and $M^{\psi=1}$
(Proposition~\ref{P:N^s uniformly fp}).

The following variant of Proposition~\ref{P:finite-generation} is tailored to the situation of Theorem~\ref{T:structure of M^psi=1}.

\begin{lemma}
\label{L:variant finite generation}
Let $M$ be a finite projective $\calR_A^\infty$-module and $N$ an $\calR_A^\infty$-submodule of $M$ such that $N \otimes_{\calR_A^\infty} \calR_A^{[s,\infty]} \stackrel \sim \to M\otimes_{\calR_A^\infty} \calR_A^{[s,\infty]}$ for any $ s>0$.
Assume moreover that $N$ is complete for a Fr\'echet topology defined by a sequence of norms $|\cdot|_{N, q_1}\leq |\cdot|_{N, q_2} \leq \cdots$ indexed by $\NN$, such that for each $m$, there exists $g_m \in \NN$ such that the operator norm $|T^{g_m}|_{N,q_m} \leq \omega$.  Then $N=M$.
\end{lemma}
\begin{proof}
For $0<s\leq r$, we write $M^{[s,r]}$ for $M  \otimes_{\calR_A^\infty} \calR_A^{[s,r]}$.
It suffices to exhibit a finite set of elements of $N$ generating all $M^{[s_i,r_i]}$ over $\calR_A^{[s_i,r_i]}$ for some admissible cover $\{[s_i,r_i]\}_{i\in \ZZ_{\geq0}}$ of $(0,\infty]$, by Lemma~\ref{L:finite-generation-local}.

We fix generators $\bbe_1, \dots, \bbe_n$ of $M$ as an $\calR_A^\infty$-module.
We proceed as in the proof of Proposition~\ref{P:finite-generation} by separating the admissible cover into subcollections $I_\odd$ and $I_\even$ of pairwise disjoint decreasing intervals.  We will prove that there exist $\bbf_1^\even, \dots, \bbf_n^\even \in N$ generating $M^{[s_{2i}, r_{2i}]}$ for any $i$, and the corresponding statement for odd subcollection will follow by the same argument.  For each $i$, we use $|\cdot|_{M, [s_{2i}, r_{2i}]}$ to denote a fixed Banach norm on $M^{[s_{2i}, r_{2i}]}$.  
Using the isomorphism $N \otimes_{\calR_A^\infty} \calR_A^{[s,\infty]} \cong M^{[s,\infty]}$, we may write each $\bbe_j$ in terms of a finite sum $\sum_\alpha m_\alpha \otimes f_\alpha(T)$ for $m_\alpha \in N$ and $f_\alpha(T) \in \calR_A^{[s,\infty]}$.  We may approximate each $f_\alpha(T)$ by elements in $\calR_A^\infty[\frac 1T]$ (resp. $\calR_A^\infty$ if $i=0$) and hence obtain $\bbm_{2i,j} \in N[\frac1T]$ (resp. $\bbm_{2i,j} \in N$ if $i=0$) such that $|\bbm_{2i,j} - \bbe_j|_{M, [s_{2i},r_{2i}]}$ is sufficiently small so that $\bbm_{2i,1}, \dots, \bbm_{2i,n}$ form a basis of $M^{[s_{2i}, r_{2i}]}$ by invoking Lemma~\ref{L:perturbation-of-generators}.  By multiplying some powers of $T$ (when $i>0$), we may take $\bbf_{2i,1}, \dots, \bbf_{2i,m} \in N$ generating $M^{[s_{2i}, r_{2i}]}$. 

By Lemma~\ref{L:perturbation-of-generators} again, there exists $\epsilon_{2i}>0$ such that any elements $\bbf'_{2i, 1}, \dots, \bbf'_{2i,n} \in N$ with $|\bbf'_{2i,j} - \bbf_{2i,j}|_{M, [s_{2i}, r_{2i}]}\leq \epsilon_{2i}$ for any $j$ also generate  $M^{[s_{2i}, r_{2i}]}$.  We define $\bbf_j^\even$ as in the proof of Proposition~\ref{P:finite-generation}, but instead of \eqref{E:conditions on alpha and lambda3}, we require (inductively on $i$)
\[
\textrm{ for any }m \leq i\textrm{ with }r_{2i}<1/g_m,\textrm{ we have }\big|\alpha_{2i,j}T^{\lambda_{2i,j}}\bbf_{2i,j}
\big|_{N, q_m} \leq p^{-i},
\]
which is implied by
\begin{eqnarray*}
&\textrm{ for any }m \leq i\textrm{ with }r_{2i}<1/g_m,\textrm{ we have }\\
&|\alpha_{2i,j}|\omega^{\lfloor\lambda_{2i,j}/g_m\rfloor} \leq |\bbf_{2i,j}|_{N, q_m}^{-1}\cdot p^{-i} \cdot \min\{1, |T|_{q_m}, \dots, |T^{g_m-1}|_{q_m}\}^{-1}.
\end{eqnarray*}
By making $\lambda_{2i,j}$ sufficiently large, we may find a choice of $\alpha_{2i,j} \in \Qp$ satisfying all required conditions, finishing the proof.
\end{proof}

\begin{lemma}
\label{L:isom on Robba descent}
Let $f: M \to N$ be a morphism of two $\calR_A^{r_0}$-modules with
finitely generated kernel (resp.\ cokernel).  If $f \otimes \calR_A: M
\otimes_{\calR_A^{r_0}} \calR_A \to N \otimes_{\calR_A^{r_0}} \calR_A$
is injective (resp.\ surjective), then there exists some $r \in (0,
r_0]$ such that $f \otimes \calR_A^r: M \otimes_{\calR_A^{r_0}}
\calR_A^r \to N \otimes_{\calR_A^{r_0}} \calR_A^r$ is injective
(resp.\ surjective).
\end{lemma}
\begin{proof}
Let $Q$ denote the kernel (resp.\ cokernel) of $f$.  We have $Q
\otimes_{\calR_A^{r_0}}\calR_A = 0$.  Since $Q$ is finitely generated,
there exists $r \in (0, r_0]$ such that $Q
\otimes_{\calR_A^{r_0}}\calR_A^{r} = 0$ and hence $f\otimes \calR_A^r$
is injective (resp.\ surjective).
\end{proof}

\begin{remark}\label{R:isom on Robba descent}
In the preceding proposition, we only need the finite generation of $N$ to test for surjectivity.  To test for bijectivity, it is sufficient that $M$ and $N$ both be finite projective: after restricting to $r \in (0,r_0]$ we may assume $f \otimes \calR_A^r$ is surjective, and the kernel is then a summand of $M$ hence also finitely generated.  (The issue we are working around here is that $\calR_A^{r_0}$ is not noetherian.)
\end{remark}

The following lemma will be needed later in Theorem~\ref{T:strong Iwasawa cohomology}.

\begin{lemma}
\label{L:finite A-mod as R_A-mod}
If $M$ is an $\calR_A^{r_0}$-module which is finite over $A$, then there exists $s \in (0, r_0]$ such that $M \otimes_{\calR_A^{r_0}} \calR_A^s = 0$.
\end{lemma}
\begin{proof}
If $A$ is a finite $\Qp$-algebra then $M$ is artinian, and so $M \to
M^{[s_0,r_0]}$ must be an isomorphism for sufficiently small $s_0$.
Any $s \in (0,s_0)$ then has the desired property.

In the general case, view $M$ as a module over the subring $A\langle
T\rangle \subseteq \calR_A^{r_0}$; it is finitely generated over this
ring, because it is finitely generated over $A$.  By specializing
modulo maximal ideals of $A$ and invoking the preliminary case above,
we see that the reduced closed affinoid subspace $\Supp(M)$ of
$\Max(A\langle T\rangle)$ does not meet $\Max(A\langle
T^{\pm1}\rangle)$.  By the maximum modulus principle, the spectral
norm $\omega^{s_0}$ of $T$ on $\Supp(M)$ therefore satisfies
$\omega^{s_0} < 1$.  Any $s \in (0,s_0)$ then has the desired
property.
\end{proof}

We conclude this subsection with some constructions involving residues
and duality of Robba rings.  Denote by $\Omega_{\calR_A} =
\widehat{\Omega}^1_{\calR_A/A}$ the module of continuous $A$-linear
differentials of $\calR_A$.  The \emph{residue} map $\Res_{\calR_A}:
\Omega_{\calR_A} \to A$ sends $\sum_{n \in \ZZ} a_nT^ndT$ to $a_{-1}$.

\begin{lemma}
\label{L:calR duality}
The module $\Omega_{\calR_A}$ is free of rank one over $\calR_A$ with
basis $dT$, and in particular $\Res_{\calR_A}$ is well-defined.
Moreover, $\Res_{\calR_A}$ depends only on $\calR_A$ up to $A$-linear
continuous isomorphism, and not on the choice of indeterminate $T$.

The natural pairing $\calR_A \times \calR_A \to A$ sending $(f(T),
g(T))$ to $\Res_{\calR_A}(f(T)g(T)dT)$ induces two topological
isomorphisms depending on $T$:
\[
\Hom_{A, \cont}(\calR_A, A) \cong \calR_A \quad\textrm{and}\quad \Hom_{A, \cont} (\calR_A/\calR^\infty_A, A) \cong \calR^\infty_A.
\]
(The topology on the first $\Hom$-set is dual-to-LF and hence LF, and
the topology on the second $\Hom$-set is Fr\'echet.)
\end{lemma}
\begin{proof}
The computation of $\Omega_{\calR_A}$ reduces to the structure of
$\widehat{\Omega}^1_{\calR_A^{[r,s]}/A}$, and that the latter is free
over $A$ with basis $dT$ is obvious.  To see that $\Res_{\calR_A}$ is
independent of choices, it suffices to take $A=\QQ_p$.  The
indeterminate $T$ belongs to the the bounded subring $\calE^\dag$ of
$\calR$, which is dense in $\calR$ and is a Henselian discretely
valued field with uniformizer $p$; call the completion of $\calE^\dag$
for the $p$-adic topology $\calE$.  Continuous automorphisms of
$\calR$ preserve $\calE^\dag$ and induce $p$-adically continuous
automorphisms of $\calE$.  Thus we may show the independence of the
residue map for $\calE$ instead of $\calR$, and this is done in
\cite[A2 2.2.3]{fontaine}.

The dualities are well-known; they appear implicitly in
\cite[Section~8.5]{kedlaya-finite}, and when $A$ is a finite extension
of $\Qp$, they are proved in \cite[Theorem~5.4]{crew}.  For the
convenience of the reader, we sketch a short proof of the bijection.
The pairing obviously defines a map $\calR_A \to
\Hom_{A,\cont}(\calR_A, A)$.  Conversely, to any
$\mu\in\Hom_{A,\cont}(\calR_A, A)$, we associate a power series
$\sum_{n \in\ZZ}\mu(T^{-1-n})T^n$.  One checks easily that the
continuity of $\mu$ translates directly to the condition that this
formal Laurent series lies in $\calR_A$.  This defines an inverse to
the natural morphism $\calR_A \to \Hom_{A,\cont}(\calR_A, A)$, proving
the first isomorphism.  Having proved this, the second isomorphism
follows immediately because the associated formal Laurent series lies
in $\calR_A^\infty$ if and only if $\mu(T^{-1-n})=0$ for any $n<0$, or
equivalently, $\mu(\calR_A^\infty)=0$.  See
\cite[Proposition~5.5]{crew} for a discussion of the topologies in the
case where $A$ is a field, the general case being similar.
\end{proof}

We will also need a group-theoretic variant of the residue pairing.

\begin{defn}\label{D:omega-Gamma}
Assume given the data $(C,\ell)$, where $C$ is a multiplicative
abelian topological group whose torsion subgroup $C_\tors$ is finite
and $C/C_\tors$ is noncanonically isomorphic to $\Zp$, and $\ell$ is a
nonzero continuous homomorphism $C \to \Qp$.

Let $0 < s \leq r$ be given.  In the case where $C$ is torsion-free,
we choose a topological generator $c$ of $C$ and we define
$\calR_A^{[s,r]}(C)$ to be the formal substitution of $T$ by $c-1$ in
$\calR_A^{[s,r]}$.  For any other choice $c'$, the element
$\frac{c'-1}{c-1}$ is a unit with all Gauss norms equal to $1$, so
this definition does not depend on $c$.  Moreover, since one has
$|f((1+T)^p-1)|_r = |f(T)|_{pr}$ for all $f(T) \in \Qp[T^{\pm1}]$ if
$r < 1$ (recall that $|T|_r = \omega^r$), it follows that if either $r
< 1$ or both $r=\infty$ and $s < 1$ then one has $\calR_A^{[s,r]}(C) =
\calR_A^{[ps,pr]}(C^p) \otimes_{\ZZ[C^p]} \ZZ[C]$.  Therefore, in the
general case, if we choose a torsion-free open subgroup $D \subseteq
C$ and set $\calR_A^{[s,r]}(C) = \calR_A^{[[C:D]s,[C:D]r]}(D)
\otimes_{\ZZ[D]} \ZZ[C]$, then $\calR_A^{[s,r]}(C)$ is independent of
all choices if either $r < 1/\#C_\tors$ or both $r=\infty$ and $s <
1/\#C_\tors$.  One defines as well $\calR_A^r(C) = \bigcap_{0 < s \leq
  r} \calR_A^{[s,r]}(C)$ and $\calR_A(C) = \bigcup_{0<r}
\calR_A^r(C)$.  For general open subgroups $D' \subseteq D$ of $C$, if
either $r < 1/\#D_\tors$ or both $r=\infty$ and $s < 1/\#D_\tors$ then
one has $\calR_A^?(D) = \calR_A^{[D:D']?}(D') \otimes_{\ZZ[D']}
\ZZ[D]$, where for $? = [s,r],r,\emptyset$ we write $[D:D']?  =
   [[D:D']s,[D:D']r], [D:D']r, \emptyset$, respectively.  In
   particular, in this case $\calR_A^?(D)$ is a free
   $\calR_A^{[D:D']?}(D')$-module of rank $[D:D']$.

If $a \in \ZZ$ and $c \in C$ is nontorsion, then $d\log(c^a) =
\frac{d(c^a)}{c^a} = a\frac{dc}c = ad\log(c)$.  It follows that the
differential form $\omega_\ell = \ell(c)^{-1}d\log(c) =
\ell(c)^{-1}\frac{dc}{c} \in \Omega_{\calR_A(C)}$ is independent of
the choice of nontorsion $c \in C$, and even independent of $C$, in
the sense that for every open subgroup $D \subseteq C$ it belongs to
$\Omega_{\calR_A(D)} \subseteq \Omega_{\calR_A(C)}$ and agrees with
the differential form $\omega_{\ell|_D}$.  One has that $\omega_\ell$
is an $\calR_A(C)$-basis of $\Omega_{\calR_A(C)}$, and therefore it
makes sense to define
\[
\Res_C: \Omega_{\calR_A(C)} \to A,
\qquad
f\omega_\ell \mapsto
   \Res_{\calR_A(D)}(\Tr_{\calR_A(C)/\calR_A(D)}(f)\omega_\ell),
\]
for $D \subseteq C$ a torsion-free open subgroup, where
$\Res_{\calR_A(D)}$ denotes the intrinsic residue map on the Robba
ring $\calR_A(D)$ as above.  It is easy to check that $\Res_C$ for
$\calR_A(C)$ is independent of the auxilliary choices $D$ and $\ell$.
As in the case of a Robba ring, the pairing
\[
\{-,-\}_\ell: \calR_A(C) \times \calR_A(C) \to A,
\qquad
(f,g) \mapsto \Res_C(fg\omega_\ell),
\]
depends only on $\ell$, and is perfect and induces topological
isomorphisms
\begin{equation}\label{E:gamma residue pairing}
\calR_A(C) \stackrel\sim\to \Hom_{A,\cont}(\calR_A(C),A),
\qquad
\calR_A^\infty(C) \stackrel\sim\to \Hom_{A,\cont}(\calR_A(C)/\calR_A^\infty(C),A).
\end{equation}
\end{defn}

\begin{remark}\label{R:normalizations}
Continue with the data $(C, \ell)$ of Definition~\ref{D:omega-Gamma}.
Let $c$ be an element of $C$ whose image in $C/C_\tors$ is a
topological generator.  For $X$ a $\Qp$-vector space equipped with a
continuous action $C$, the continuous cohomology $H^*(C, X)$ can be
computed by the complex $[X^{C_\tors} \xrightarrow{c-1} X^{C_\tors}]$.
In particular, $H^1(C, X)$ is isomorphic to $X^{C_\tors} / (c-1) \cong
X_C$, the coinvariants of $X$.  We wish to warn the reader that this
isomorphism depends on the choice of the element $c$ above, in the
following sense.  Given another choice $c'$, the natural isomorphism
$[X^{C_\tors} \xrightarrow{c-1} X^{C_\tors}]$ with $[X^{C_\tors}
  \xrightarrow{c'-1} X^{C_\tors}]$ is given by $[1 \quad
  \frac{c'-1}{c-1}]$.  The two isomorphisms $H^1(C, X) \approx X_C$
differ thus by a factor of $\ell(c')/\ell(c)$; multiplying by
$\ell(c)^{-1}$ the na\"ive identification $H^1(C,X) \stackrel\sim\to
X_C$ arising from the choice $c$ yields an isomorphism that is
compatible with changing the choice of $c$, and depends only on
$\ell$.

The preceding discussion applies equally well when the $p$-torsion
subgroup $C_\ptors$ is such that $C_\tors/C_\ptors$ is cyclic (the
Chinese remainder theorem then implies $C/C_\ptors$ is procyclic), one
replaces $X^{C_\tors}$ by $X^{C_\ptors}$, and one instead chooses $c
\in C$ mapping to a topological generator $\overline c$ of
$C/C_\ptors$.  These conditions are satisfied, in the notations to
appear shortly, by $C = \Gamma_K$, $\ell = \log \circ \chi$, and $c =
\gamma_K$.  Then every occurrence of the term $\log\chi(\gamma_K)$ or
its reciprocal in this paper can be explained as normalizing a passage
between $H^1(\Gamma_K,X)$ and $X_{\Gamma_K}$ to render it compatible
with changing the choice of $\gamma_K$.
\end{remark}

\begin{remark}
Except for Corollary \ref{C:pointwise criterion for finite projectivity} and Lemma~\ref{L:constant dimension function implies flat}, the results in this subsection remain valid even if $A$ is only assumed to be a strongly noetherian Banach algebra over $\Qp$, meaning that the Tate algebra over $A$ in any number of variables is Noetherian.
\end{remark}

\subsection{$(\varphi, \Gamma)$-modules and Galois representations}

In this subsection, we define $(\varphi, \Gamma)$-modules over the
relative Robba ring and state their relationship with Galois
representations.  In particular, using the results from the previous
subsection, we establish the equivalence of the two viewpoints of
relative $(\varphi,\Gamma)$-modules treated in \cite{kiran-liu} (one
defined using modules over the relative Robba ring, the other defined
using vector bundles over a relative annulus).

\begin{notation}
Throughout the rest of the paper, fix an algebraic closure $\overline
{\QQ}_p$ of $\Qp$ and a finite extension $K$ of $\QQ_p$ inside
$\overline {\QQ}_p$.  Choose a system $\varepsilon = (\zeta_{p^n})_{n
  \geq 0}$, where each $\zeta_{p^n} \in \overline {\QQ}_p$ is a
primitive $p^n$th root of unity and $(\zeta_{p^{n+1}})^p =
\zeta_{p^n}$.

Let $\mu_{p^\infty}$ denote the set of $p$-power roots of unity.  Set
$G_K = \Gal(\overline \QQ_p / K)$, $H_K = \Gal(\overline \QQ_p /
K(\mu_{p^\infty}))$, and $\Gamma_K = \Gal(K(\mu_{p^\infty}) / K)$.
The group $\Gamma_K$ is naturally identified to a subgroup of
$\Gamma_{\Qp}$.  The cyclotomic character $\chi: \Gamma_{\Qp} =
\Gal(\QQ_p(\mu_{p^\infty})/ \Qp) \to \ZZ_p^\times$ is given by
$\gamma(\zeta) = \zeta^{\chi(\gamma)}$ for any $\zeta \in
\mu_{p^\infty}$ and $\gamma \in \Gamma_{\Qp}$.

We let $k$ be the residue field of $K$, with $F = W(k)[1/p]$, and we
let $k'$ be the residue field of $K(\mu_{p^\infty})$, with $F'
=W(k')[1/p]$.  We write $\tilde e_K =
[K(\mu_{p^\infty}):\Qp(\mu_{p^\infty})]/[k':k]$ for the ramification
degree of $K(\mu_{p^\infty})/\Qp(\mu_{p^\infty})$.
\end{notation}

\begin{defn}\label{D:field of norms}
The theory of the field of norms allows a certain choice of
indeterminate $\pi_K$, and gives rise to constants $C(\pi_K)>0$.  (See
\cite[pp.~154--156]{fontaine-ouyang} for details.  We assume, in
addition, that $C(\pi_K) < 1/\#(\Gamma_{\Qp})_\tors$.)
For $0 < s \leq r\leq C(\pi_K)$, we set $\calR^{[s,r]}(\pi_K)$ to be the formal substitution of $T$ by $\pi_K$ in the ring $\calR_{F'}^{[s,r]}$; we set $\calR_A^{[s,r]}(\pi_K)  = \calR^{[s,r]}(\pi_K) \widehat \otimes_{\Qp} A$.  We similarly define $\calR_A^r(\pi_K)$ and $\calR_A(\pi_K)$; the latter is referred to as the \emph{relative Robba ring} over $A$ for $K$.
If $\pi_K$ and $\pi'_K$ are two choices of indeterminates as above, then for $0<s \leq r \leq \min\{C(\pi_K),C(\pi'_K)\}$ the respective rings $\calR_A^{[s,r]}(\pi_K)$ and $\calR_A^{[s,r]}(\pi'_K)$ are canonically isomorphic.
There are commuting $A$-linear actions on $\calR_A^{[s,r]}(\pi_K)$ of $\Gamma_K$ and of an operator $\varphi:\calR_A^{[s,r]}(\pi_K) \to \calR_A^{[s/p,r/p]}(\pi_K)$.  
The actions on the coefficients $F'$ are the natural ones, $\Gamma_K$ through its
quotient $\Gal(F'/F)$ and $\varphi$ by Witt functoriality, but the actions on $\pi_K$ depend on its choice; see \cite{berger-intro} for more discussion of these actions.  Moreover, $\varphi$ makes $\calR_A^{[s,r]}(\pi_K)$ into a free $\calR_A^{[s/p,r/p]}(\pi_K)$-module of rank $p$, and we obtain a $\Gamma_K$-equivariant left inverse $\psi : \calR_A^{[s/p,r/p]}(\pi_K) \to \calR_A^{[s,r]}(\pi_K)$ by the formula $p^{-1}\varphi^{-1} \circ
\Tr_{\calR_A^{[s/p,r/p]}(\pi_K)/\varphi( \calR_A^{[s,r]}(\pi_K))}$.  The map $\psi$ extends to functions $\calR_A^{r_0/p}(\pi_K) \to \calR_A^{r_0}(\pi_K)$ for $r_0 \leq C(\pi_K)$ and $\calR_A(\pi_K) \to \calR_A(\pi_K)$.
\end{defn}

\begin{notation}\label{N:simplify Qp}
Since we will often be in the case $K = \Qp$, we simplify our
notations in this case, writing $\Gamma = \Gamma_{\Qp}$ and $\calR_A^?
= \calR_A^?(\pi_K)$ for $?=[s,r],r,\emptyset$ (the latter overloading
our previous meaning for $\calR_A^?$).

Moreover, in this case one has $F'=\Qp$, and the choice of
$\varepsilon$ above gives rise to a preferred indeterminate $\pi =
\pi_{\Qp}$ with constant any number $C(\pi) < 1/(p-1)$.  We can
explicitly describe the $(\varphi, \Gamma)$-actions on this
indeterminate by $\gamma(\pi) = (1+\pi)^{\chi(\gamma)}-1$ for $\gamma
\in \Gamma = \Gamma_{\Qp}$ and $\varphi(\pi) = (1+\pi)^p-1$.  (Any
other choice of $\varepsilon$ is of the form $\gamma(\varepsilon)$,
$\gamma \in \Gamma$, and leads to the indeterminate $\gamma(\pi)$.)
The $\psi$-action is computed by knowing that, in this case,
$\calR^{[s,r]}_A = \bigoplus_{i=0}^{p-1}
(1+\pi)^i\varphi\calR^{[ps,pr]}_A$, so if $f = \sum_{i=0}^{p-1}
(1+\pi)^i \varphi(f_i)$ then $\psi(f) = f_0$.

We use the following notation for subgroups of $\Gamma =
\Gamma_{\Qp}$.  For $n \geq 1$ (resp. $n \geq 2$ if $p=2$), write
$\gamma_n \in \Gamma$ for the unique element with $\chi(\gamma_n) =
1+p^n$, as well as $\gamma'_n = \gamma_1^{p^{n-1}}$ so that
$\chi(\gamma'_n) = (1+p)^{p^{n-1}}$.  Each of these elements
topologically generates the subgroup $\Gamma_n$ of $\Gamma$ mapped
isomorphically by $\chi$ onto $1+p^n\ZZ_p$, so that $\Gamma = \Gamma_1
\times \mu_{p-1}$ (resp. $\Gamma = \Gamma_2 \times \mu_2$ if $p=2$).
\end{notation}

\begin{remark}
\label{R:Robba-ring-base-change}  
For $L$ a finite extension of $K$ and for $0< r \leq
\min\{C(\pi_K),C(\pi_L)\tilde e_L/\tilde e_K\}$, 
$\calR^{r\tilde e_K/\tilde e_L}_A(\pi_L)$ is naturally a finite \'etale
$\calR^r_A(\pi_K)$-algebra, free of rank
$[L(\mu_{p^\infty}):K(\mu_{p^\infty})]$; when $L/K$ is
Galois, it has Galois group
$H_K/H_L$.  In general, the inclusion is equivariant for the
actions of $\Gamma_L$ (as a subgroup of $\Gamma_K$) and $\varphi$.
\end{remark}

\begin{notation}
For $M^{r_0}$ a module over $\calR_A^{r_0}(\pi_K)$ and for $0< s \leq
r \leq r_0$, we set $M^? = M^{r_0} \otimes_{\calR_A^{r_0}(\pi_K)}
\calR_A^?(\pi_K)$ where $? = [s,r],s,\emptyset$.  For example, $M =
M^{r_0} \otimes_{\calR_A^{r_0}(\pi_K)} \calR_A(\pi_K)$.
\end{notation}

\begin{defn}
Let $r_0 \in (0, C(\pi_K)]$.  A \emph{$\varphi$-module} over
  $\calR^{r_0}_A(\pi_K)$ is a finite projective
  $\calR^{r_0}_A(\pi_K)$-module $M^{r_0}$ equipped with an isomorphism
  $\varphi^* M^{r_0} \cong M^{r_0/p}$.  A \emph{$\varphi$-module} over
  $\calR_A(\pi_K)$ is the base change to $\calR_A(\pi_K)$ of a
  $\varphi$-module over some $\calR^{r_0}_A(\pi_K)$.  For $M^{r_0}$ a
  $\varphi$-module over $\calR_A^{r_0}(\pi_K)$, the isomorphism
  $\varphi^* M^{r_0} \cong M^{r_0/p}$ induces a continuous $A$-linear
  homomorphism $\psi: M^{r_0/p} \cong \varphi (M^{r_0})
  \otimes_{\varphi(\calR_A^{r_0}(\pi_K))} \calR_A^{r_0/p}(\pi_K) \to
  M^{r_0}$ given by $\psi(\varphi(m) \otimes f) = m \otimes \psi(f)$
  for $m \in M^{r_0}$ and $f \in \calR_A^{r_0/p}(\pi_K)$.

A \emph{$\varphi$-bundle} over $\calR_A^{r_0}(\pi_K)$ is a vector bundle $(M^{[s,r]})$ over $\calR_A^{r_0}(\pi_K)$ equipped with  isomorphisms $\varphi^*M^{[s,r]} \cong M^{[s/p, r/p]}$ for all $0 < s \leq r \leq r_0$ satisfying the obvious compatibility conditions.
\end{defn}

\begin{prop}
\label{P:phi module v.s. phi bundle}
The natural functor from  $\varphi$-modules over $\calR^{r_0}_A(\pi_K)$ to  $\varphi$-bundles over $\calR_A^{r_0}(\pi_K)$ is an equivalence of categories.
\end{prop}
\begin{proof}
This natural functor is obviously fully faithful.  To check essential surjectivity, by Proposition~\ref{P:finite-generation}, it suffices to check that  any $\varphi$-bundle $(M^{[s,r]})$ is uniformly finitely presented.  Since $\calR_A^{[r_0/p, r_0]}(\pi_K)$ is noetherian, $M^{[r_0/p, r_0]}$ is $(m,n)$-finitely presented for some $m,n \in \NN$.  Pulling back along $\varphi^a$ for any $a\in\NN$, we see that $M^{[r_0/p^{a+1}, r_0/p^a]} \cong (\varphi^a)^* M^{[r_0/p, r_0]}$ is also $(m,n)$-finitely presented.  This checks that the $\varphi$-bundle is uniformly finitely presented, and finishes the proof.
\end{proof}

\begin{remark}
Proposition~\ref{P:phi module v.s. phi bundle}
 shows that the two types of objects considered in \cite{kiran-liu}, $(\varphi, \Gamma)$-modules over $\calR_A$ and families of $(\varphi, \Gamma)$-modules, are actually the same thing.  We will no longer distinguish them from now on.  But we remind the reader that these categories do depend on the choice of $r_0$.
\end{remark}

The following two technical lemmas will be used in the proof of Theorem~\ref{T:global triangulation step 1}, the first of which is a strengthening of Lemma~\ref{L:isom on Robba descent} making use of the $\varphi$-structure.

\begin{lemma}
\label{L:extending isom and surj}
Let $M^{r_0}$ and $N^{r_0}$ be two  $\varphi$-modules over $\calR_A^{r_0}(\pi_K)$ and let $g: M \to N$ be a morphism of  $\varphi$-modules over $\calR_A(\pi_K)$.  Then $g$ is induced by some morphism $g^{r_0}:M^{r_0} \to N^{r_0}$ of  $\varphi$-modules over $\calR_A^{r_0}(\pi_K)$. Moreover, if $g$ is injective (resp. surjective), so is $g^{r_0}$.
\end{lemma}
\begin{proof}
The morphism $g$ is the base change of some $g^r: M^r \to N^r$ for some $r \in (0, r_0]$. For $r' = \min\{pr, r_0\}$, $g^r$ induces $g^{r'/p}: M^{r'/p} \to N^{r'/p}$. We define
\[
g^{r'}: M^{r'} \xrightarrow{\varphi} M^{r'/p} \xrightarrow{g^{r'/p}} N^{r'/p} \xrightarrow{\psi} N^{r'}.
\]
It is easy to check that $g^{r'}$ is a homomorphism of $\varphi$-modules over $\calR_A^{r'}(\pi_K)$, which agrees with $g^r$ upon base change to $\calR^r_A(\pi_K)$.
Iterating this construction gives the morphism $g^{r_0}:M^{r_0} \to N^{r_0}$.
Now, suppose $g$ is injective (resp. surjective).  Let $Q^{r_0}$ be the kernel (resp. cokernel) of $g^{r_0}$; then $\varphi^*(Q^r) \cong Q^{r/p}$ for $r \in (0, r_0]$.  Using the similar argument as in the proof of Proposition~\ref{P:phi module v.s. phi bundle} (by invoking Proposition~\ref{P:finite-generation}(1)), we know that $Q^{r_0}$ is finitely generated over $\calR^{r_0}_A(\pi_K)$.  By the injectivity (resp. surjectivity) of $g$ and Lemma~\ref{L:isom on Robba descent}, we know that $Q^r = 0$ for some $r$.  But then for $r' = \min\{pr, r_0\}$, $\varphi^*(Q^{r'}) = Q^{r'/p} = 0$ implies that $Q^{r'} =0$.  Iterating this proves the lemma.
\end{proof}

\begin{notation}
For any $f \in A$, we let $Z(f) \subseteq \Max(A)$ denote the closed subspace defined by $f$.
\end{notation}

\begin{lemma}
\label{L:local saturation=>global saturation}
Let $f \in A$ be a nonzero element.
Let $g: M \to N$ be a morphism of  $\varphi$-modules over $\calR_A$ (or $\calR_A^{r_0}$ for some $r_0>0$).  Assume that $g_z: M_z \to N_z$ is surjective for every $z \in \Max(A) \backslash Z(f)$.  Then $g:M[\frac 1f] \to N[\frac 1f]$ is surjective.
\end{lemma}
\begin{proof}
If both  $\varphi$-modules are defined over $\calR_A$, we may choose $r_0>0$ such that both modules and the morphism are the base change from $\calR_A^{r_0}$; in this case, the surjectivity of the model $g^{r_0}_z$ of $g_z$ over $\calR^{r_0}_{\kappa_z}$ for $z \in \Max(A)\backslash Z(f)$ is still guaranteed by Lemma~\ref{L:extending isom and surj}.  Therefore, we need only to prove the lemma when the conditions are stated for  $\varphi$-modules over $\calR_A^{r_0}$.
Let $Q$ denote the cokernel of $g$; it is finitely presented.  The assumption of the lemma implies that $Q_z=0$ for any $z \in \Max(A)\backslash Z(f)$. Therefore, $Q^{[r_0/p,r_0]}$ is supported on $Z(f) \times \rmA^1[r_0/p,r_0]$, and  is then killed by $f^m$ for some $m\in \NN$.  Pulling back along powers of $\varphi$, we have $Q^{[r_0/p^{n+1}, r_0/p^n]} \cong (\varphi^n)^* Q^{[r_0/p,r_0]}$ is also killed by $f^m$ for any $n \in \NN$.  Therefore, $Q$ is killed by $f^m$, concluding the lemma.
\end{proof}

\begin{defn}\label{D:phigammamodule}
Let $r_0 \in (0, C(\pi_K)]$.  A \emph{$(\varphi, \Gamma_K)$-module}
  over $\calR_A^{r_0}(\pi_K)$ is a $\varphi$-module $M^{r_0}$ over
  $\calR_A^{r_0}(\pi_K)$ equipped with a commuting semilinear
  continuous action of $\Gamma_K$.  A \emph{$(\varphi,
    \Gamma_K)$-module} over $\calR_A(\pi_K)$ is the base change to
  $\calR_A(\pi_K)$ of a $(\varphi, \Gamma_K)$-module over
  $\calR_A^{r_0}(\pi_K)$ for some $r_0 \in (0, C(\pi_K)]$.  We note
    that the map $\psi$ commutes with the action of $\Gamma_K$.

If $L$ is a finite extension of $K$ and $M$ is a $(\varphi,
\Gamma_L)$-module over $\calR_A(\pi_L)$, we define the \emph{induced}
$(\varphi, \Gamma_K)$-module to be $\Ind_L^K M =
\Ind_{\Gamma_K}^{\Gamma_L}M$, viewed as an $\calR_A(\pi_K)$-module,
where the action of $\varphi$ is inherited from $M$.
\end{defn}

\begin{remark}
The continuity condition in the preceding definition means that, for
any $m \in M^{[s,r]}$ with $0 < s \leq r \leq r_0$, the map $\Gamma_K
\to M^{[s,r]}$, $\gamma \mapsto \gamma(m)$ is continuous for the
profinite topology on $\Gamma_K$ and the Banach topology on
$M^{[s,r]}$.  In the next proposition, we reformulate this condition
more explicitly, and show that one in fact obtains an action of
$\calR_A^\infty(\Gamma_K)$ on $M$, the latter ring being defined in
Definition~\ref{D:omega-Gamma}.
\end{remark}

The following claims are well-known, but not well documented in the
literature, so we provide a proof.

\begin{prop}
\label{P:calR_A(Gamma_K) acts on M}
Let $r_0 \in (0,C(\pi_K)]$, and let $M^{r_0}$ be a finitely generated
  $\calR_A^{r_0}(\pi_K)$-module, say generated by $m_1,\ldots,m_N$.
  For any $s \in (0,r_0]$ let $|\cdot|_{M^{[s,r_0]}}$ denote any fixed
    $\calR_A^{[s,r_0]}(\pi_K)$-Banach module norm on $M^{[s,r_0]}$,
    and $\|\cdot\|_{M^{[s,r_0]}}$ the associated operator norm.
    Assume $M^{r_0}$ is equipped with a semilinear action of
    $\Gamma_K$.
\begin{itemize}
\item[(1)] In order that the action of $\Gamma_K$ on $M^{r_0}$ be
  continuous, it is necessary and sufficient that for each
  $i=1,\ldots,N$ one has $\lim_{\gamma \to 1} \gamma(m_i) = m_i$.
\item[(2)] If the action of $\Gamma_K$ on $M^{r_0}$ is continuous,
  then for each $s$ one has $\lim_{\gamma \to 1}
  \|\gamma-1\|_{M^{[s,r_0]}} = 0$.
\item[(3)] If the action of $\Gamma_K$ on $M^{r_0}$ is continuous,
  then the action of $A[\Gamma_K]$ extends by continuity to an action
  of $\calR^\infty_A(\Gamma_K)$.  More precisely, taking any $n \geq
  1$ (resp.\ $n \geq 2$ if $p=2$) such that $\Gamma_n \subseteq
  \Gamma_K$, for any $s \in (0,r_0]$ there exists $g_s \in \NN$ such
    that $\|(\gamma'_n-1)^{g_s}\|_{M^{[s,r_0]}} \leq \omega$.
\end{itemize}
\end{prop}
\begin{proof}
For (1) and (2), if the action is continuous then clearly the
condition on the $m_i$ holds.  Conversely, assuming the condition on
the $m_i$ holds, we now show that the condition on the operator norms
holds, which implies that the action of $\Gamma_K$ is continuous,
because we only need to check continuity at the origin.  A general $m
\in M^{[s,r_0]}$ is of the form $\sum_{i=1}^N f_im_i$ with $f_i \in
\calR_A^{[s,r_0]}(\pi_K)$, and since there are only finitely many
terms, it suffices to treat each $f_im_i$ separately.  Writing
\[
\gamma(f_im_i)-f_im_i = (\gamma-1)(f_i)m_i+\gamma(f_i)(\gamma-1)(m_i),
\]
it suffices to make $(\gamma-1)(f_i)$ arbitrarily small, to keep
$\gamma(f_i)$ bounded, and to make $(\gamma-1)(m_i)$ arbitrarily
small, with respect to $\|\cdot\|_{M^{[s,r_0]}}$ as $\gamma \to 1$.
The third term can be made small by hypothesis.  The second term is
bounded by $\|f_i\|_{M^{[s,r_0]}}$ as soon as we can make the first
term small.  So we have reduced the general case to treating the free
rank one module $M = \calR_A^{r_0}(\pi_K)$.  Since the module
$\Ind_K^{\Qp} \calR_A^{r_0}(\pi_K)$ is noncanonically a direct sum of
finitely many copies of $\calR_A^{r_0}(\pi_K)$, the desired claim for
$\calR_A^{r_0}(\pi_K)$ follows from the case of the $\calR_A^{\tilde
  e_K r_0}(\pi_{\Qp})$-module $\Ind_K^{\Qp} \calR_A^{r_0}(\pi_K)$ with
$\Gamma_{\Qp}$-action.  But, as above, a general $\calR_A^{\tilde e_K
  r_0}(\pi_{\Qp})$-module with $\Gamma_{\Qp}$-action reduces to the
case of the free rank one module.  Henceforth, we replace $\tilde e_K
r_0$ by $r_0$ for simplicity.  We now take $\pi_{\Qp} = \pi$ as in
Notation~\ref{N:simplify Qp}.  Because $|(1+\pi)-1|_{[s,r_0]} < 1$ we
have $|(1+\pi)^a-1|_{[s,r_0]} \leq C(s,r_0) < 1$ for some constant
$C(s,r_0)$ independent of $a \in \Zp$, and because taking $p$th powers
improves congruences, this implies that for any $\epsilon>0$ we can
make $|(1+\pi)^{ap^n}-1|_{[s,r_0]}<\epsilon$ independently of $a \in
\Zp$ by taking $n$ sufficiently large.  For $n \geq 1$ (resp.\ $n \geq
2$ if $p=2$) and $\gamma \in \Gamma_n$, write $\chi(\gamma) = 1+ap^n$
with $a \in \Zp$.  Given $f \in \calR_A^{[s,r_0]}$, in order to bound
$|\gamma(f)-f|_{[s,r_0]}$ it suffices to consider the positive and
negative powers of $\pi$ separately.  Elementary algebra gives
\begin{gather}\label{E:pi estimates}
\gamma(\pi)-\pi
= (1+\pi)((1+\pi)^{ap^n}-1), \\
\nonumber
\frac1{\gamma(\pi)} - \frac1\pi
= -\frac{1+\pi}{\pi}
   \frac{(1+\pi)^{ap^n}-1}{(1+\pi)((1+\pi)^{ap^n}-1)+\pi}.
\end{gather}
For any $\epsilon>0$, by taking $n$ large enough the first
(resp.\ second) right hand side can be made less than $\epsilon
|\pi|_{[s,r_0]}$ (resp.\ $\epsilon |\pi^{-1}|_{[s,r_0]}$)
independently of $a \in \Zp$.  This proves (1) and (2).

For (3), the second claim implies the first, because it provides $M$
with an action of $\calR_A^\infty(\Gamma_n)$ and one has
$\calR_A^\infty(\Gamma_K) = \calR_A^\infty(\Gamma_n)
\otimes_{\ZZ[\Gamma_n]} \ZZ[\Gamma_K]$.  Thus we concern ourselves
with the second claim.  Because $\Ind_K^{\Qp} M$ is noncanonically a
direct sum of finitely many $\Gamma_K$-stable copies of $M$, and
$\Gamma_K \subseteq \Gamma_{\Qp}$, the desired bound follows from the
case of the $\calR_A^{\tilde e_K r_0}(\pi_{\Qp})$-module $\Ind_K^{\Qp}
M$ with $\Gamma_{\Qp}$-action.  So we assume from now on that $K=\Qp$
and we replace $\tilde e_K r_0$ by $r_0$.  Since for any $a \in \Zp$
and $n \geq 1$ (resp.\ $n \geq 2$ if $p=2$) one has
$|(1+\pi)^{ap^n}-1|_{[s,r_0]} < |\pi|_{[s,r_0]}$, it follows from the
identities \eqref{E:pi estimates} that the norm $|\cdot|_{[s,r_0]}$ on
$\calR_A^{[s,r_0]}(\pi_{\Qp})$ is invariant under precomposition with
$\gamma \in \Gamma_n$.  It follows from (2) that for some $n' \geq n$,
the norm $|\cdot|_{M^{[s,r_0]}}$ is invariant under $\Gamma_{n'}$.
Then $C = \max_{\gamma \in \Gamma_n/\Gamma_{n'}}
\|\tilde\gamma\|_{M^{[s,r_0]}}$, where $\tilde\gamma$ denotes any
choice of lift of $\gamma \to \Gamma_n$, is equal to $\max_{\gamma \in
  \Gamma_n} \|\gamma\|_{M^{[s,r_0]}}$.  Keeping in mind that
$|\cdot|_{M^{[s,r_0]}}$ is invariant under $\Gamma_n$, it follows that
the average
\[
|\cdot|'_{M^{[s,r_0]}}
=
\frac1{[\Gamma_n:\Gamma_{n'}]}
\sum_{\gamma \in \Gamma_n/\Gamma_{n'}} |\cdot|_{M^{[s,r_0]}} \circ \tilde\gamma
\]
does not depend on the choice of $n'$, and is a $\Gamma_n$-invariant
norm on $M^{[s,r_0]}$ such that
\[
C^{-1}|\cdot|'_{M^{[s,r_0]}}
\leq
|\cdot|_{M^{[s,r_0]}}
\leq
C|\cdot|'_{M^{[s,r_0]}}.
\]
Assuming we know the result for the norm $|\cdot|'_{M^{[s,r_0]}}$,
there exists $g'_s \in \NN$ such that
$\|\gamma_n'-1)^{g'_s}\|'_{M^{[s,r_0]}} \leq \omega$.  Choosing $h \in
\NN$ large enough so that $\omega^hC^2 \leq 1$, we have
\[
\|(\gamma'_n-1)^{(h+1)g'_s}\|_{M^{[s,r_0]}}
\leq
C\|(\gamma'_n-1)^{(h+1)g'_s}\|'_{M^{[s,r_0]}}
\leq
C^2\omega^{h+1} \leq \omega,
\]
whence the claim for the norm $|\cdot|_{M^{[s,r_0]}}$.  By (2), there
exists $n' \geq n$ such that $||\gamma'_{n'}-1||_{M^{[s,r_0]}} \leq
C^{-2}\omega$ so that $||\gamma'_{n'}-1||'_{M^{[s,r_0]}} \leq \omega$.
Now write $(\gamma'_n-1)^{p^{n'-n}} = (\gamma'_{n'}-1) +
P(\gamma'_{n'})$, where $P(X) \in \ZZ[X]$ is divisible by $p$; the
first term has operator norm under $|\cdot|'_{M^{[s,r_0]}}$ at most
$\omega$ by the choice of $n'$, and the latter term has operator norm less
than or equal to $|p|\leq \omega$, so the operator norm
$||(\gamma'_{n}-1)^{p^{n'-n}}||'_{M^{[s,r_0]}} \leq \omega$, proving
part (3) of the proposition.
\end{proof}

It would be useful for technical purposes to know that the following
conjecture holds (see also \cite[\S3,~Question~1]{R:Bellaiche} and
\cite[Remark~6.2]{kiran-liu}).

\begin{conj}
\label{C:strong locally free}
For $M$ a $(\varphi, \Gamma_K)$-module over $\calR_A(\pi_K)$, there exists a finite admissible cover $\{\Max(A_i)\}_{i=1, \dots, n}$ of $\Max(A)$ such that each $M \widehat\otimes_A A_i$ is a finite \emph{free} $\calR_{A_i}(\pi_K)$-module.
\end{conj}

\begin{remark}
We will prove the preceding conjecture for rank one
$(\varphi,\Gamma_K)$-modules in Theorem~\ref{T:full rank 1
  classification}, and the conjecture follows for trianguline
$(\varphi,\Gamma_K)$-modules.  The conjecture is also known for the
essential image of the functor $\bbD_\rig$ (defined in
Theorem~\ref{T:Drig} below), though not for a general \'etale
$(\varphi,\Gamma_K)$-module.
\end{remark}

The importance of $(\varphi, \Gamma)$-modules rests upon the following result.

\begin{theorem}
\label{T:Drig}
Let $V$ be a finite projective $A$-module equipped with a continuous $A$-linear action of $G_K$.
Then there is functorially associated to $V$ a $(\varphi,\Gamma_K)$-module
$\bbD_\rig(V)$ over $\calR_A(\pi_K)$.  The rule $V \mapsto \bbD_\rig(V)$ is fully faithful and exact, and it commutes with base change in $A$.
\end{theorem}

\begin{proof}
See \cite[Theorem~3.11]{kiran-liu}, which generalizes
\cite[Th\'eor\`eme~4.2.9]{berger-colmez}\footnote{R.~Liu has pointed out to us a small gap in this reference; see \cite[Theorem~1.1.4]{R:Liu2} for a fix.}. (It was also pointed out by Ga\"etan Chenevier that some modification of \cite[Th\'eor\`eme~4.2.9]{berger-colmez} can recover the theorem; see \cite[Remark~3.13]{kiran-liu}.)  For further details, see
\cite[\S2.1]{pottharst1}.
\end{proof}

\begin{lemma}
\label{L:compatibility of induction}
Let $V$ be a finite projective $A$-module equipped with a continuous
$A$-linear action of $G_L$.  Denote the induced representation of
$G_K$ by $\Ind_{G_L}^{G_K} V$, and let $\Ind_L^K \bbD_\rig(V)$ be as
in Definition~\ref{D:phigammamodule}.  Then we have
$\bbD_\rig(\Ind_{G_L}^{G_K} V) \cong \Ind_L^K \bbD_\rig(V)$.
\end{lemma}
\begin{proof}
This can be proved in the same way as in \cite[Proposition~3.1]{liu}.
\end{proof}

\subsection{Cohomology of $(\varphi, \Gamma)$-modules and duality}

Fontaine and Herr \cite{herr} defined a cohomology theory for $(\varphi, \Gamma)$-modules, compatible with the theory of Galois cohomology.  When $A$ is a finite extension of $\Qp$, Liu \cite{liu} showed that this cohomology theory of $(\varphi,\Gamma)$-modules admits analogs of local Tate duality and the Euler characteristic formula for $p$-adic representations of $G_K$.  In this subsection, we set up the basic formalism in the relative setting, including Tate duality pairing.  In the special case $K=\Qp$, we give an explicit computation of the pairing in terms of residues.

\begin{notation}
Let $\Delta_K$ denote the $p$-torsion subgroup of $\Gamma_K$, which is
trivial if $p \neq 2$ and at largest cyclic of order two.  When
$K=\Qp$, we write $\Delta = \Delta_{\Qp}$.  Choose $\gamma_K \in
\Gamma_K$ whose image in $\Gamma_K / \Delta_K$ is a topological
generator. (This choice is useful for explicit formulas, but if
desired one can reformulate everything to eliminate this choice.)
\end{notation}

\begin{notation}
We follow \cite{nekovar} for sign conventions throughout.  Thus, the
tensor product of complexes $X^\bullet \otimes Y^\bullet$ has
differential defined by $d_{X \otimes Y}^{i+j}(x \otimes y) = d_X^ix
\otimes y + (-1)^ix \otimes d_Y^jy$ if $x \in X^i$ and $y \in Y^j$.
For a morphism of complexes $f^\bullet: X^\bullet \to Y^\bullet$,
define its \emph{mapping fiber} $\Fib(f) = \Cone(f)[-1]$, i.e.\ the
complex with $\Fib(f)^i = X^i \oplus Y^{i-1}$ and $d_{\Fib(f)}^i =
d_X^i - d_Y^{i-1} - f^i$.  If $f: X \to Y$ is a morphism of objects,
then $\Fib(f)$ is the complex $[X \xrightarrow{-f} Y]$ concentrated in
degrees $0,1$.
\end{notation}

\begin{defn}
\label{D:varphi Gamma cohomology}
For $M$ a  $(\varphi, \Gamma_K)$-module over $\calR_A(\pi_K)$, we define the complexes $\rmC^\bullet_{\varphi, \gamma_K}(M)$ and $\rmC^\bullet_{\psi, \gamma_K}(M)$ concentrated in degree $[0,2]$, and a morphism $\Psi_M$ between them, as follows:
\begin{equation}
\label{E:Herr-complex}
\xymatrix@C=10pt{
\rmC^\bullet_{\varphi, \gamma_K}(M) \ar[d]^{\Psi_M} & = &
[ M^{\Delta_K} \ar^(.4){(\varphi-1,\gamma_K-1)}[rrrr]\ar^{\id}[d] &&&& M^{\Delta_K} \oplus M^{\Delta_K} \ar^(.6){(\gamma_K-1) \oplus (1-\varphi)}[rrrr] \ar^{-\psi \oplus \id}[d] &&&& M^{\Delta_K}] \ar^{-\psi}[d] \\
\rmC^\bullet_{\psi, \gamma_K}(M) & = & [M^{\Delta_K} \ar^(.4){(\psi-1,\gamma_K-1)}[rrrr] &&&& M^{\Delta_K} \oplus M^{\Delta_K} \ar^(.6){(\gamma_K-1) \oplus (1-\psi)}[rrrr] &&&& M^{\Delta_K}].
}
\end{equation}
These complexes and the morphism are independent of the choice of
$\gamma_K$ up to canonical $A$-linear isomorphism: if $f$ is one of
$\varphi$ or $\psi$ then $\Gamma_{\gamma_K,\gamma_K',M} :
\rmC_{f,\gamma_K}^\bullet(M) \stackrel\sim\to
\rmC_{f,\gamma_K'}^\bullet(M)$ is given by $[1 \quad 1 \oplus
  \frac{\gamma_K'-1}{\gamma_K-1} \quad
  \frac{\gamma_K'-1}{\gamma_K-1}]$.

The complex $\rmC_{\varphi,\gamma_K}^\bullet(M)$ is called the \emph{Herr complex} of $M$.  Its cohomology group is denoted
$H^*_{\varphi, \gamma_K}(M)$ and is called the \emph{$(\varphi, \Gamma)$-cohomology}.
 We similarly define $H^*_{\psi, \gamma_K}(M)$ to be the cohomology of $\rmC^\bullet_{\psi, \gamma_K}(M)$, called the \emph{$(\psi, \Gamma)$-cohomology} of $M$.  
We remark that the same definition applies to any module $N$ with a commuting action of $\varphi$ (resp. $\psi$) and $\Gamma_K$; we denote the resulting cohomology by $H^*_{\varphi, \gamma_K}(N)$ (resp. $H^*_{\psi, \gamma_K}(N)$).
\end{defn}

\begin{remark}\label{R:useful-sequence}
For $f$ one of $\varphi$ or $\psi$, we have
$\rmC_{f,\gamma_K}^\bullet(M) =
\Fib(1-\gamma_K|\Fib(1-f|M^{\Delta_K}))$.  Moreover, from the
definition we easily deduce a useful short exact sequence
\[
0 \to M^{\Delta_K,\psi=1}/(\gamma_K-1) \to H^1_{\psi, \gamma_K}(M) \to
(M/(\psi-1))^{\Gamma_K} \to 0,
\]
where the nontrivial maps are induced by inclusion of the second
(resp.\ projection onto the first) coordinate in
$\rmC^1_{\psi,\gamma_K}(M) = M^{\Delta_K} \oplus M^{\Delta_K}$.
\end{remark}

\begin{lemma}
Let $L$ be a finite extension of $K$. Let $M$ be a $(\varphi, \Gamma_L)$-module over $\calR_A(\pi_L)$.  We have a natural quasi-isomorphism $\rmC^\bullet_{\varphi, \gamma_L}(M) \to \rmC^\bullet_{\varphi, \gamma_K}(\Ind_L^K(M))$ (inducing isomorphisms $H^i_{\varphi, \gamma_L}(M) \cong H^i_{\varphi, \gamma_K}(\Ind_L^KM)$ for any $i$).
\end{lemma}
\begin{proof}
As $\rmC^\bullet_{\varphi,\gamma_K}(M)$ is equal, up to signs within
the differentials, to the mapping fiber of $1-\varphi$ on the complex
\[
\rmC^\bullet_{\gamma_K}(M) = [M^{\Delta_K} \xrightarrow{\gamma_K-1} M^{\Delta_K}],
\]
it suffices to treat the complex $\rmC^\bullet_{\gamma_K}(M)$ in place
of $\rmC^\bullet_{\varphi,\gamma_K}(M)$.  But the latter complex is
functorially quasi-isomorphic to the continuous cochain group
$\rmC_\cont^\bullet(\Gamma_K,M)$.  Thus the lemma follows from the
usual formulation of Shapiro's lemma for $\Gamma_L \subseteq
\Gamma_K$.
\end{proof}

\begin{prop}
\label{P:phi cohomology = psi cohomology}
For $M$ a $(\varphi, \Gamma_K)$-module over $\calR_A(\pi_K)$, the morphism $\Psi_M :
\rmC_{\varphi,\gamma_K}^\bullet(M) \to \rmC_{\psi,\gamma_K}^\bullet(M)$ is a
quasi-isomorphism.  Hence we have $H^*_{\varphi, \gamma_K}(M) \cong H^*_{\psi, \gamma_K}(M)$.
\end{prop}
\begin{proof}
The morphism $\Psi_M$ is surjective at each degree.
The kernel of $\Psi_M$ is the complex $[M^{\Delta_K, \psi =0} \xrightarrow{\gamma_K - 1} M^{\Delta_K,\psi=0}]$ concentrated in degrees 1 and 2.  The bijectivity of $\gamma_K-1$ on $M^{\Delta_K,\psi = 0}$ will be proved in Theorem~\ref{T:structure of D^psi=0} in the next subsection (the reader can check that there is no circular reasoning), and the proposition follows from this fact.
\end{proof}

From now on, we will use the $(\varphi, \Gamma)$-cohomology and the $(\psi, \Gamma)$-cohomology interchangeably without notification.

\begin{prop}
\label{P:comparison with Galois cohomology}
Let $V$ be a finite projective $A$-module equipped with a continuous
$A$-linear action of $G_K$.  If we use $\bbR\Gamma_\cont(G_K,V)$ to
denote the cohomology of continuous $G_K$-cochains with values in $V$,
we have a functorial isomorphism $\bbR\Gamma_\cont(G_K,V) \cong
\bbR\Gamma_{\varphi,\gamma_K}(\bbD_\rig(V))$ compatible with base
change, in the derived category of perfect complexes over $A$.
\end{prop}

\begin{proof}
This is \cite[Theorem~2.8]{pottharst1}.
\end{proof}

\begin{defn}
As is customary, write $\Zp(1) = \varprojlim_n \mu_{p^n}$, and for any
$\Zp$-module $V$ with $G_K$-action (possibly via $\Gamma_K$), write
$V(1)$ for $V \otimes_{\Zp} \Zp(1)$ with the diagonal action.

We write $\calR_A(\pi_K)(1) = \bbD_\rig(A(1))$.  For $M$ a $(\varphi,
\Gamma_K)$-module over $\calR_A(\pi_K)$, we write $M(1)$ for the twist
$M \otimes_{\calR(\pi_K)} \calR(\pi_K)(1)$, and $M^\dual$ for the
module dual $\Hom_{\calR_A(\pi_K)}(M, \calR_A(\pi_K))$, as well as
$M^* = M^\dual(1)$ for the \emph{Cartier dual}.  For $L$ a finite
extension of $K$ and $M$ a $(\varphi, \Gamma_L)$-module over
$\calR_A(\pi_L)$, we have $(\Ind_L^K(M))^* \cong \Ind_L^K(M^*)$.

Since $A$ is flat over $\Qp$, one has $H^2(G_K,\Qp(1)) \otimes_{\Qp} A
\cong H^2(G_K,A(1))$.  From local class field theory one has the trace
isomorphism $H^2(G_K,\Qp(1)) \cong \Qp$, hence also $H^2(G_K,A(1))
\cong A$.  Combining this with the case $V = A(1)$ in the preceding
proposition, one gets the \emph{Tate isomorphism} $\Ta_K:
H^2_{\varphi,\gamma_K}(\calR_A(\pi_K)(1)) \cong H^2(G_K,A(1)) \cong
A$.
\end{defn}

We record a general formalism of cup products, which is well-known in the literature.

\begin{lemma}
\label{L:cup product}
Let $R$ be a ring.  Let $f_i, g_i: \rmC_i^\bullet \to \rmD_i^\bullet$ for $i = 1,2$ be morphisms of $R$-complexes.  Then we have a natural morphism 
\[
\Fib(\rmC_1^\bullet \xrightarrow{g_1 - f_1} \rmD_1^\bullet) \otimes \Fib(\rmC_2^\bullet \xrightarrow{g_2 - f_2} \rmD_2^\bullet) \quad \longrightarrow \quad \Fib(\rmC_1^\bullet \otimes \rmC_2^\bullet \xrightarrow{g_1 \otimes g_2 - f_1 \otimes f_2} \rmD_1^\bullet \otimes \rmD_2^\bullet)
\]
given by the downward arrows of the diagram
\[
\xymatrix@C=3.5cm{
\rmC_1^\bullet \otimes \rmC_2^\bullet \ar[r]^-{\big((f_1-g_1) \otimes 1 , 1 \otimes (f_2 -g_2)\big)} \ar[d]^{\id} & \rmD_1^\bullet \otimes \rmC_2^\bullet \oplus \rmC_1^\bullet \otimes \rmD_2^\bullet \ar[r]^-{\big(1 \otimes (g_2-f_2)\big) \oplus \big((f_1 - g_1) \otimes 1\big)} \ar[d]^{1 \otimes g_2 \oplus f_1 \otimes 1  } & \rmD_1^\bullet \otimes \rmD_2^\bullet\\
\rmC_1^\bullet \otimes \rmC_2^\bullet
\ar[r]^{f_1\otimes f_2 - g_1 \otimes g_2}
& \rmD_1^\bullet \otimes \rmD_2^\bullet
}
\]
where the maps $1 \otimes (f_2-g_2)$ in the first horizontal arrow and
$f_1 \otimes 1$ in the second vertical arrow get multiplied by $-1$ on
summands with a tensor factor that is an odd-degree piece of
$\Fib(\rmC_1^\bullet \xrightarrow{g_1-f_1} \rmD_1^\bullet)$.
\end{lemma}
\begin{proof}
This is a straightforward computation.
\end{proof}

\begin{defn}
For $M_1,M_2$ two $(\varphi, \Gamma_K)$-modules over $\calR_A(\pi_K)$,
we construct cup products as follows.  Write $\rmC_\varphi^\bullet(M)
= \Fib(1-\varphi|M) = [M \xrightarrow{\varphi-1} M]$ for brevity.
First apply Lemma~\ref{L:cup product} (with $R=A$) to the case
$\rmC_i^\bullet = \rmD_i^\bullet = M_i$, $f_i = \varphi$, and $g_i =
\id$, to get the first arrow in the composite
\begin{multline*}
\cup_\varphi :
\rmC_\varphi^\bullet(M_1) \otimes_A \rmC_\varphi^\bullet(M_2)
= [M_1 \xrightarrow{\varphi-1} M_1]
  \otimes_A [M_2 \xrightarrow{\varphi-1} M_2] \\
\to
[M_1 \otimes_A M_2 \xrightarrow{\varphi \otimes \varphi - 1 \otimes 1}
  M_1 \otimes_A M_2] \\
\twoheadrightarrow
[M_1 \otimes_{\calR_A(\pi_K)} M_2 \xrightarrow{\varphi-1}
  M_1 \otimes_{\calR_A(\pi_K)} M_2]
= \rmC_\varphi^\bullet(M_1 \otimes_{\calR_A(\pi_K)} M_2).
\end{multline*}
Then, one applies
Lemma~\ref{L:cup product} (with $R=A$) to the case $\rmC_i^\bullet = \rmD_i^\bullet =
\rmC_\varphi^\bullet(M_i)$, $f_i = \gamma_K$, and $g_i = \id$, to get
the arrow $\cup_{\gamma_K}$ in the composite
\begin{multline*}
\cup_{\varphi,\gamma_K} :
\rmC_{\varphi,\gamma_K}^\bullet(M_1)
  \otimes_A
  \rmC_{\varphi,\gamma_K}^\bullet(M_2) \\
=
\Fib(\rmC_\varphi^\bullet(M_1) \xrightarrow{1-\gamma_K}
    \rmC_\varphi^\bullet(M_1))
  \otimes_A
  \Fib(\rmC_\varphi^\bullet(M_2) \xrightarrow{1-\gamma_K}
    \rmC_\varphi^\bullet(M_2)) \\
\xrightarrow{\cup_{\gamma_K}}
\Fib(\rmC_\varphi^\bullet(M_1) \otimes_A \rmC_\varphi^\bullet(M_2)
  \xrightarrow{1 \otimes 1 - \gamma_K \otimes \gamma_K}
  \rmC_\varphi^\bullet(M_1) \otimes_A \rmC_\varphi^\bullet(M_2)) \\
\xrightarrow{\cup_\varphi}
\Fib(\rmC_\varphi^\bullet(M_1 \otimes_{\calR_A(\pi_K)} M_2)
  \xrightarrow{1-\gamma_K}
  \rmC_\varphi^\bullet(M_1 \otimes_{\calR_A(\pi_K)} M_2)) \\
=
\rmC_{\varphi,\gamma_K}^\bullet(M_1 \otimes_{\calR_A(\pi_K)} M_2).
\end{multline*}
This morphism induces an $A$-bilinear pairing $H^i_{\varphi,
  \gamma_K}(M) \times H^j_{\varphi, \gamma_K}(N) \to H^{i+j}_{\varphi,
  \gamma_K}(M \otimes_{\calR_A(\pi_K)} N)$.  This gives rise to the
explicit formulas of \cite[\S4.2]{herr2}.  Especially, when $i=j=1$,
one has
\[
\overline{(x_1, x_2)} \otimes \overline{(y_1, y_2)} \mapsto \overline{x_2 \otimes \gamma(y_1) - x_1 \otimes \varphi(y_2)}.
\]

Using the cup product, functoriality for the morphism $M
\otimes_{\calR_A(\pi_K)} M^* \to \calR_A(\pi_K)(1)$, and the Tate map
$\Ta_K: H^2_{\varphi, \gamma_K}(\calR_A(\pi_K)(1)) \cong A$, one gets
\emph{Tate duality pairings}
\begin{gather}
\nonumber
\cup_\Ta: \rmC^\bullet_{\varphi, \gamma_K}(M) \times \rmC^\bullet_{\varphi, \gamma_K}(M^*) \xrightarrow{\cup_{\varphi,\gamma_K}} \rmC^\bullet_{\varphi, \gamma_K}(\calR_A(\pi_K)(1)) \\
\label{E:Tate duality complex}
\to H^2_{\varphi, \gamma_K}(\calR_A(\pi_K)(1))[-2] \xrightarrow{\Ta_K} A[-2],
\\
\label{E:Tate duality}
\cup_\Ta: H^i_{\varphi, \gamma_K}(M) \times H^{2-i}_{\varphi, \gamma_K}(M^*) \to H^2_{\varphi, \gamma_K}(\calR_A(\pi_K)(1)) \xrightarrow{\Ta_K} A.
\end{gather}
\end{defn}

The following extension of Tate local duality is due to Liu \cite{liu}.  We will later generalize this result to the case of a general affinoid algebra $A$ in Theorem~\ref{T:finite cohomology}.
\begin{theorem}[Liu]  
\label{T:Liu}
Suppose that $A$ is a finite extension of $\QQ_p$. Let $M$ be a $(\varphi, \Gamma_K)$-module over $\calR_A(\pi_K)$.
\begin{enumerate}
\item[(1)] The $A$-vector spaces $H^i_{\varphi,\gamma_K}(M)$ are finite-dimensional for $i=0,1,2$.
\item[(2)]
The Euler characteristic $\chi(M) = \sum_{i=0}^2 (-1)^i \dim_A H^i_{\varphi,\gamma_K}(M)$ equals
$-[K:\QQ_p] \rank M$.
\item[(3)]
The duality cup product \eqref{E:Tate duality} is a perfect pairing.
\end{enumerate}
\end{theorem}

We focus for the remainder of this subsection on the case $K=\Qp$, and
describe an explicit construction of the Tate isomorphism $\Ta_{\Qp}$
(up to a fixed $\QQ_p^\times$-multiple) and duality pairings
$\cup_\Ta$.

Confusing additive and multiplicative notations, we may write
$\QQ_p(1) = \QQ_p \otimes \varepsilon$ and $\calR(1) = \calR \otimes
\varepsilon$ (with $\varphi$ acting trivially on $\varepsilon$).
Recall that to our choice of $\varepsilon$ is associated an
indeterminate $\pi$ for $\calR$.  We define the \emph{normalized
  residue maps} $\Res_{\Qp}, \Res_{\gamma_{\Qp}}: \calR_A(1) = \calR_A
\otimes \varepsilon \to A$ by $\Res_{\Qp}(f \otimes \varepsilon) =
\Res_{\calR_A}(f\frac{d\pi}{1+\pi})$ and $\Res_{\gamma_{\Qp}} =
(\log\chi(\gamma_{\Qp}))^{-1} \Res_{\Qp}$.  These maps are formally
checked to be independent of the choice of $\varepsilon$, and
invariant under precomposition with $\varphi$, $\psi$, and all $\gamma
\in \Gamma$ (cf. \cite[\S I.2]{colmez-mirabolique}).  In particular,
they factor through maps $H^2_{\varphi,\gamma_{\Qp}}(\calR_A(1)) \cong
\calR_A(1)_\Gamma/(\varphi-1) \to A$ which we also call $\Res_{\Qp}$
and $\Res_{\gamma_{\Qp}}$, respectively.  By the reasoning of
Remark~\ref{R:normalizations}, the maps $\Res_{\gamma_{\Qp}}$ are
compatible for varying choices of $\gamma_{\Qp}$.  Since
$H^2_{\varphi,\gamma_{\Qp}}(\calR_A(1))$ is isomorphic to $A$ as
above, and the residue map is obviously surjective, it follows that
$\Res_{\Qp}$ and $\Res_{\gamma_{\Qp}}$ are isomorphisms.  By base
changing from the case $A=\Qp$, we see that $\Res_{\gamma_{\Qp}} = C_p
\cdot \Ta_{\Qp}$ for some $C_p \in \QQ_p^\times$ independent of $A$
and $\gamma_{\Qp}$.

\begin{remark}
Although the precise value of $C_p$ is irrelevant in this paper, we
remark that it is claimed in \cite[Theorem~2.2.6]{benois} that $C_p =
-(p-1)/p$.
\end{remark}

\begin{notation}
\label{N:residue pairing}
Let $\langle -, -\rangle$ denote the tautological $\calR_A$-bilinear
pairing $M \times M^* \to \calR_A(1)$.  We define the \emph{residue
  pairings} $\{-,-\}_{\Qp},\{-,-\}_{\gamma_{\Qp}}: M \times M^* \to A$
by $\Res_{\Qp} \circ \langle-,-\rangle$ and $\Res_{\gamma_{\Qp}} \circ
\langle -, -\rangle$, respectively.  (Our notation differs from the
twisted pairing defined by Colmez \cite{colmez-kirillov}.)
\end{notation}

From the properties of $\langle-,-\rangle$ and $\Res_{\gamma_{\Qp}}$
it follows that $\{\gamma(x), \gamma(y)\} = \{\varphi(x), \varphi(y)\}
= \{x,y\}$, that $\{\varphi(x), y\} = \{x, \psi(y)\}$, and that
$\{\psi(x), y\} = \{x, \varphi(y)\}$ for any $\gamma \in \Gamma$, $x
\in M$ and $y \in M^*$.  Then the scaled Tate duality pairing $C_p
\cdot \cup_\Ta$ is computed, under the identification
\[
\Psi_{M^*} \circ \Gamma_{\gamma_{\Qp},\gamma_{\Qp}^{-1},M^*} :
\rmC_{\varphi,\gamma_{\Qp}}^\bullet(M^*)
\stackrel\sim\to
\rmC_{\psi,\gamma_{\Qp}^{-1}}^\bullet(M^*),
\]
by the diagram
\[
\xymatrix@R=2pt{
\rmC^\bullet_{\varphi, \gamma_{\Qp}}(M)\ = &
[\ M^{\Delta_K} \ar[r] & M^{\Delta_K} \oplus M^{\Delta_K} \ar[r] & M^{\Delta_K}\ ] \\
&\times & \times & \times\\
& [\ (M^*)^{\Delta_K} \ar[ddd]^{\{-,-\}_{\gamma_{\Qp}}} & (M^*)^{\Delta_K} \oplus (M^*)^{\Delta_K} \ar[l]\ar[ddd]^{\{-,-\}_{\gamma_{\Qp},1}} &(M^*)^{\Delta_K}\ ]\ar[ddd]^{\{-,-\}_{\gamma_{\Qp}}} \ar[l]&=\ \rmC^\bullet_{\psi, \gamma_{\Qp}^{-1}}(M^*),\\
\\ \\
& A&A&A
}
\]
where $\{(m,n),(k,l)\}_{\gamma_{\Qp},1} =
\{m,l\}_{\gamma_{\Qp}}-\{n,k\}_{\gamma_{\Qp}}$.  (The diagram only
commutes up to sign, but it is compatible with the sign convention for
the tensor product of complexes.)

\section{The $\psi$ operator}
\label{S:psi operator}

\setcounter{theorem}{0}

We now focus attention more closely on the action of the operator
$\psi$ on $(\varphi, \Gamma)$-modules, and particularly the kernel of
$\psi$ and the kernel and cokernel of $\psi-1$. The technical
importance of the $\psi$-action is apparent in much of the prior work
on the cohomology of $(\varphi, \Gamma)$-modules. Moreover, Fontaine
observed that, in addition to being a useful intermediate step in the
computation of $(\varphi, \Gamma)$-cohomology, $\psi$-cohomology is
itself important because of its identification with Iwasawa
cohomology; see \cite{cherbonnier-colmez}. We thus take a bit of extra
care in our analysis of $\psi$-cohomology here, in order to later
obtain conclusions about Iwasawa cohomology in arithmetic families.

\begin{notation}
\label{N:Gamma_n}
Throughout this section, the notations $\calR_A^?(C)$ of
Definition~\ref{D:omega-Gamma} are in force with $C = \Gamma_K$ and
$c=\gamma_K$.
\end{notation}

\begin{remark}
Although we have assumed throughout $A$ to be a $\Qp$-affinoid
algebra, the statements and proofs in this section apply to any
strongly noetherian Banach algebra $A$ over $\Qp$.
\end{remark}

\subsection{$\Gamma$-action on $M^{\psi=0}$}

The aim of this subsection is to prove the following theorem regarding the $\Gamma_K$-action on $M^{\psi=0}$, by removing an auxiliary boundedness condition from \cite[Th\'eor\`eme~2.4]{chenevier}.

\begin{theorem}
\label{T:structure of D^psi=0}
For $M$ any $(\varphi,\Gamma_K)$-module over $\calR_A(\pi_K)$, there
exists $r_1 \in (0,C(\pi_K))$ such that for any $0< r \leq r_1$,
$\gamma_K - 1$ is invertible on $(M^r)^{\psi=0}$, and the $A[\Gamma_K,
  (\gamma_K-1)^{-1}]$-module structure on $(M^r)^{\psi=0}$ extends
uniquely by continuity to an $\calR_A^{\tilde e_K r}(\Gamma_K)$-module
structure, for which $(M^r)^{\psi=0}$ is finite projective of rank
$[K:\Qp]\cdot\rank M$.
\end{theorem}
\begin{proof}
We first observe that
\[
(\Ind_K^{\Qp}M)^{\psi = 0} = \big(\Hom_{\ZZ[\Gamma_K]}(\ZZ[\Gamma], M)\big)^{\psi = 0} = \Hom_{\ZZ[\Gamma_K]}(\ZZ[\Gamma], M^{\psi = 0}).
\]
Since $\ZZ[\Gamma]$ is a finite free $\ZZ[\Gamma_K]$-module, the
statement of the theorem for $M$ is equivalent to that for
$\Ind_K^{\Qp}M$.  (Note our restriction on $C(\pi_K)$ in
Definition~\ref{D:field of norms} allows us to compare
$\calR_A^{\tilde e_Kr}(\Gamma_K)$ and $\calR_A^r(\Gamma)$.)
Hence we can and will assume that $K =\Qp$ from now on.

We may assume that $M$ is a $(\varphi,\Gamma)$-module over
$\calR_A^{r_0}$, where $r_0 < C(\pi_K)$ is small enough so that $|p| <
|\pi^p|_{r_0}$, and that $\psi$ is defined.  For $0 < s \leq r \leq
r_0$, because of the decomposition $\calR_A^{[s/p^n,r/p^n]} =
\bigoplus_{i \in \ZZ/p^n\ZZ}
(1+\pi)^{\tilde\imath}\varphi^n\calR_A^{[s,r]}$, where $\tilde\imath$
denotes any lift of $i$ to $\ZZ$, we have a decomposition
\[
(M^{[s/p^n, r/p^n]})^{\psi =0} = \bigoplus_{i \in (\ZZ/p^n\ZZ)^\times} (1+\pi)^{\tilde\imath} \varphi^n M^{[s,r]} \cong (1+\pi)\varphi^nM^{[s,r]} \otimes_{\ZZ[\Gamma_n]} \ZZ[\Gamma]
\]
of $\ZZ[\Gamma]$-modules.  Since $\calR_A^{[s/p^n, r/p^n]}(\Gamma) =
\calR_A^{[s,r]}(\Gamma_n) \otimes_{\ZZ[\Gamma_n]} \ZZ[\Gamma]$ for all
$n \geq 0$ (using the restriction on $C(\pi_K)$ in
Definition~\ref{D:field of norms}), we need only to prove the
following: for \emph{some} $n\geq 2$, the module
$(1+\pi)\varphi^nM^{r_0}$ has invertible action of $\gamma_n-1$, it
admits a unique extension of the
$A[\Gamma_n,(\gamma_n-1)^{-1}]$-action by continuity to an
$\calR_A^{r_0}(\Gamma_n)$-module structure, and with respect to this
structure it is finite projective of the same rank as $M$.  Then the
theorem holds for $r_1 = r_0/p^n$.

Since $M$ is a finite projective $\calR_A^{r_0}$-module, there exists a finite (projective) $\calR_A^{r_0}$-module $N$ such that $M \oplus N$ is free with basis $\bbe_1, \dots, \bbe_m$ over $\calR_A^{r_0}$.  For $I = [r_0, r_0]$, $[r_0/p, r_0/p]$, or $[r_0/p, r_0]$, set $N^I = N \otimes_{\calR_A^{r_0}} \calR_A^I$.
We equip $M^I$ and $N^I$ with any Banach module norms $|\cdot|_{M,I}$ and $|\cdot|_{N,I}$, and $M^I \oplus N^I$ with the supremum of these norms, denoted $|\cdot|_{M \oplus N,I}$.

By Proposition~\ref{P:calR_A(Gamma_K) acts on M}(2), there exists $n_0
\in \NN$ such that for all $n \geq n_0$ the operator norms of
$\gamma_n-1$ satisfy $\|\gamma_n-1\|_{M, I}< |\pi|_I$ for any $I =
    [r_0, r_0]$, $[r_0/p, r_0/p]$, or $[r_0/p, r_0]$.

We will prove in Corollary~\ref{C:structure of 1+pi varphi^n} that,
for any $l \in \NN$, $(1+\pi)\varphi^{n_0} M^{[r_0/p^l,r_0/p^{l-1}]}$
has invertible action of $\gamma_{n_0}-1$, it admits a unique
extension of the $A[\Gamma_{n_0},(\gamma_{n_0}-1)^{-1}]$-action by
continuity to an $\calR_A^{[r_0/p^l,r_0/p^{l-1}]}(\Gamma_{n_0})$-module
structure, and it is finite projective of the same rank as $M$, in
fact generated by $m$ elements (and hence $(m,m)$-finitely presented
because of the projectivity), such that
\begin{align*}
&(1+\pi)\varphi^{n_0} M^{[r_0/p^l,r_0/p^{l-1}]} \otimes_{\calR_A^{[r_0/p^l,r_0/p^{l-1}]}(\Gamma_{n_0})} \calR_A^{[r_0/p^l,r_0/p^l]}(\Gamma_{n_0})
\\
\cong\ & (1+\pi)\varphi^{n_0} M^{[r_0/p^{l+1},r_0/p^l]} \otimes_{\calR_A^{[r_0/p^{l+1},r_0/p^l]}(\Gamma_{n_0})} \calR_A^{[r_0/p^l,r_0/p^l]}(\Gamma_{n_0})
\end{align*}
 for any $l$.
By Proposition~\ref{P:finite-generation}(3), this implies that $(1+\pi)\varphi^{n_0} M^{r_0}$ is a finite projective $\calR_A^{r_0}(\Gamma_{n_0})$-module of the same rank as $M$.
\end{proof}

\begin{lemma}
\label{L:(1+pi)phi(D) finite projective}
Retain the notations of the preceding proof.  Let $I = [r_0/p, r_0],
[r_0/p, r_0/p]$, or $[r_0,r_0]$.  If $n \geq n_0$, then
$(1+\pi)\varphi^n M^{I}$ has invertible action of $\gamma_n-1$, it
admits a unique extension of the
$A[\Gamma_n,(\gamma_n-1)^{-1}]$-action by continuity to an
$\calR_A^{I}(\Gamma_n)$-module structure, and it is finite projective
of the same rank as $M$ with (at most) $m$ generators.  Moreover, we
have natural isomorphisms
\begin{align*}
(1+\pi)\varphi^n M^{[r_0/p, r_0/p]} &\cong 
(1+\pi)\varphi^n M^{[r_0/p, r_0]} \otimes_{\calR_A^{[r_0/p, r_0]}(\Gamma_n)} \calR_A^{[r_0/p, r_0/p]}(\Gamma_n),\textrm{ and }\\
(1+\pi)\varphi^n M^{[r_0, r_0]} &\cong
(1+\pi)\varphi^n M^{[r_0/p, r_0]} \otimes_{\calR_A^{[r_0/p, r_0]}(\Gamma_n)} \calR_A^{[r_0, r_0]}(\Gamma_n).
\end{align*}
\end{lemma}

\begin{proof}
We observe that $\gamma_n-1$ acts on $(1+\pi)\varphi^n M^I$ by sending $(1+\pi)\varphi^n(x)$ to
\begin{align*}
\gamma_n((1+\pi)\varphi^n(x)) - (1+\pi)\varphi^n(x)
&= (1+\pi)(1+\pi)^{p^n}\varphi^n\gamma_n(x) - (1+\pi)\varphi^n(x) \\
&= (1+\pi)\varphi^n((1+\pi)\gamma_n(x)-x) \\
&= (1+\pi)\varphi^n(G_{\gamma_n}(x)),
\end{align*}
where we denote by $G_{\gamma_n}$ the operator on $M^I$ given by
\[
G_{\gamma_n} : x \mapsto (1+\pi)\gamma_n(x)-x = \pi \cdot
\left(1 + \frac{(1+\pi)}{\pi}(\gamma_n-1)\right)(x).
\]

We observe that, by
our choice of $n_0$, the series
\[
\sum_{k \geq 0} \left( -\frac{(1+\pi)}{\pi} (\gamma_n-1)
  \right)^k \cdot \pi^{-1}
\]
converges for the operator norm $\|\cdot\|_{M,I}$, and is inverse to
$G_{\gamma_n}$.  Thus $\gamma_n-1$ acts invertibly on
$(1+\pi)\varphi^nM^I$.  Moreover, we have $\|(G_{\gamma_n})^{\pm
  1}-(\pi)^{\pm1}\|_{M,I} < |\pi^{\pm1}|_I$.  As a result, $M^I$
admits a unique extension of the action of $A[G_{\gamma_n}^{\pm 1}]$
by continuity an $\calR_A^I(G_{\gamma_n})$-module structure, and by
transport of structure $(1+\pi)\varphi^n M^I$ admits a unique
extension of the action of $A[(\gamma_n-1)^{\pm1}]$ by continuity to
an $\calR_A^I(\Gamma_n)$-module structure.

We extend the action of $G_{\gamma_n}$ to $M^I\oplus N^I$ by having
$G_{\gamma_n}$ act as multiplication by $\pi$ on $N^I$.  We still have
\begin{equation}\label{E:norm condition}
\|(G_{\gamma_n})^{\pm 1}-(\pi)^{\pm1}\|_{M \oplus N,I} < |\pi^{\pm1}|_I.
\end{equation}

Consider the following two maps
\begin{align*}
\Phi: \bigoplus_{i=1}^m \calR_A^I(\Gamma_n)\bbe_i &\longrightarrow M^I \oplus N^I & \Phi': \bigoplus_{i=1}^m \calR_A^I(\Gamma_n)\bbe_i &\longrightarrow M^I \oplus N^I
\\
\sum_{i=1}^m f_i(\gamma_n-1) \bbe_i& \longmapsto \sum_{i=1}^m f_i(\pi) \bbe_i & \sum_{i=1}^m f_i(\gamma_n-1) \bbe_i& \longmapsto \sum_{i=1}^m f_i(G_{\gamma_n}) \bbe_i,
\end{align*}
where $f_i(G_{\gamma_n})$ is the formal substitution of $G_{\gamma_n}$
into the variable of the formal Laurent series.  If we provide
$\bigoplus_{i=1}^m \calR_A^I(\Gamma_n)\bbe_i$ with the norm defined by
the $\bbe_i$, the map $\Phi$ is a topological isomorphism. The norm
condition (\ref{E:norm condition}) implies that
\[
|\Phi' \circ \Phi^{-1} (x) - x |_{M \oplus N,I} < |x|_{M \oplus N,I}
\qquad
\text{for any } x \in M^I \oplus N^I,
\]
forcing $\Phi'$ to be an isomorphism too.  Hence, if we let
$\calR_A^I(G_{\gamma_n})$ denote the formal substitution of
$G_{\gamma_n}$ into the indeterminate of $\calR_A^I$, then $M^I \oplus
N^I$ is a free $\calR_A^I(G_{\gamma_n})$-module of rank $m$. In
particular, this implies that $(1+\pi)\varphi^n M^I$ is finite
projective over $\calR_A^I(\Gamma_n)$, and is generated by $m$
elements.  Moreover, since the definition of $\Phi'$ is compatible
with changing the interval $I$, we have the base change property
asserted in the lemma.

By our construction, we have $\rank_{\calR_A^I}(M^I \oplus N^I) = m =  \rank_{\calR_A^I(G_{\gamma_n})}(M^I \oplus N^I)$. Since the $G_{\gamma_n}$-action on $N^I$ is given by multiplication by $\pi$, we see immediately that $\rank_{\calR_A^I(\Gamma_n)}(1+\pi)\varphi^n M^I$ agrees with $\rank_{\calR_A}M$.
\end{proof}

\begin{cor}
\label{C:structure of 1+pi varphi^n}
Retain the notations of the preceding proof.
For any $l \geq 0$, let $I_l = [r_0/p^{l+1}, r_0/p^l], [r_0/p^{l+1}, r_0/p^{l+1}]$, or $[r_0/p^l,r_0/p^l]$.
Then $(1+\pi)\varphi^{n_0} M^I$ is a finite projective module over $\calR_A^{I}(\Gamma_{n_0})$ of the same rank as $M$ with (at most) $m$ generators.  Moreover, we have
\begin{align*}
(1+\pi)\varphi^n M^{[r_0/p^{l+1}, r_0/p^{l+1}]} &\cong
(1+\pi)\varphi^n M^{[r_0/p^{l+1}, r_0/p^l]} \otimes_{\calR_A^{[r_0/p^{l+1}, r_0/p^l]}(\Gamma_n)} \calR_A^{[r_0/p^{l+1}, r_0/p^{l+1}]}(\Gamma_n),\textrm{ and }\\
(1+\pi)\varphi^n M^{[r_0/p^l,r_0/p^l]} &\cong
(1+\pi)\varphi^n M^{[r_0/p^{l+1}, r_0/p^l]} \otimes_{\calR_A^{[r_0/p^{l+1}, r_0/p^l]}(\Gamma_n)} \calR_A^{[r_0/p^l,r_0/p^l]}(\Gamma_n).
\end{align*}
\end{cor}
\begin{proof}
This follows from Lemma~\ref{L:(1+pi)phi(D) finite projective} and the fact that, for any $0 <s \leq r\leq r_0$, 
\[
(1+\pi) \varphi^{n} M^{[s/p^l, r/p^l]} = (1+\pi)\varphi^{n+l} M^{[s,r]} \otimes_{\ZZ[\Gamma_{n+l}]} \ZZ[\Gamma_{n}] = (1+\pi)\varphi^{n+l} M^{[s,r]} \otimes_{\calR_A^{[s,r]}(\Gamma_{n+l})} \calR_A^{[s/p^l, r/p^l]}(\Gamma_{n}).
\]
\end{proof}

\subsection{$\psi$-cohomology of $M / tM$}

An important ingredient of Liu's proof \cite{liu} of Theorem~\ref{T:Liu} is the careful study of the $(\varphi, \Gamma)$-cohomology of ``torsion" modules $M/tM$.  This allows one to freely shift the slopes of the modules.  We  generalize this result to the relative setting, and give more information concerning the $\Gamma$-action.

\begin{hypothesis}
In this subsection, we assume that $K = \Qp$.
\end{hypothesis}

\begin{notation}
We define the special elements $q = \varphi(\pi)/\pi$, $q_n =
\varphi^{n-1}(q)$ for $n \geq 1$, and $t = \log(1+\pi) =
\pi\prod_{n\geq 1} (q_n/p)$ of $\calR^\infty$.  We have $\varphi(t) =
pt$ and $\gamma(t) = \chi(\gamma)t$ for any $\gamma \in \Gamma$.  For
$0<s\leq r$, we have $q_n \notin (\calR^{[s,r]})^\times$ if and only
if $s \leq 1/p^{n-1}\leq r$; in this case, $\calR^{[s,r]} / q_n \cong
\Qp[\pi]/q_n$ may be identified with $\Qp(\zeta_{p^n})$ by identifying
$1+\pi$ with $\zeta_{p^n}$.  This identification is equivariant for
$\Gamma$, and for $\varphi$ in the sense that $\varphi: \calR^{[s,r]}
/ q_n \to \calR^{[s/p,r/p]} / q_{n+1}$ is identified with the natural
embedding $\Qp(\zeta_{p^n}) \to \Qp(\zeta_{p^{n+1}})$.

We define similar special elements of $\calR_A^\infty(\Gamma_{n_0})$,
namely $q_{0, \gamma_{n_0}} =\gamma_{n_0}-1$, $q_{n, \gamma_{n_0}} =
p^{-1}(\gamma_{n_0}^{p^n}-1)/ (\gamma_{n_0}^{p^{n-1}}-1)$ for $n
\geq1$, and $\ell_{\gamma_{n_0}} = \prod_{n \geq 0} q_{n,
  \gamma_{n_0}}$.
\end{notation}

The following lemma is essentially \cite[Proposition~2.16]{colmez}.

\begin{lemma}
\label{L:structure of M/tM}
Let $M^{r_0}$ be a $(\varphi,\Gamma)$-module over $\calR_A^{r_0}$, and
let $n_0 = \lceil -\log_p r_0 \rceil$.  If we denote $M_n^{r_0} =
M^{r_0} / q_nM^{r_0}$ for $n \geq n_0$, then we have the
following.
\begin{itemize}
\item[(1)] $M^{r_0} / tM^{r_0} \cong \prod_{n \geq n_0}M_n^{r_0}$;
\item[(2)] $\varphi^{n'-n} \otimes 1$ induces an isomorphism $M^{r_0}_n \otimes_{\Qp(\zeta_{p^n})} \Qp(\zeta_{p^{n'}}) \cong M^{r_0}_{n'}$ as $A[\Gamma]$-modules for any $n' \geq n \geq n_0$;
\item[(3)] under the product decomposition of (1), the map $\varphi: M^{r_0}/tM^{r_0} \to M^{r_0/p}/tM^{r_0/p}$ takes $(x_n)_n$ to $(x_{n-1})_n$;
\item[(4)] $\psi: M^{r_0/p}/tM^{r_0/p} \to M^{r_0}/tM^{r_0}$ takes $(x_n)_n$ to $(p^{-1}\Tr_{M_n^{r_0}/M_{n-1}^{r_0}}(x_n))_n$, where $\Tr_{M_n^{r_0}/M_{n-1}^{r_0}}: M_n^{r_0} \cong M_{n-1}^{r_0} \otimes_{\Qp(\zeta_{p^{n-1}})} \Qp(\zeta_{p^n}) \to M_{n-1}^{r_0}$ is given by the trace on the second factor.
\end{itemize}
\end{lemma}
\begin{proof}
(1) For $0<s< r_0$, we have $\calR_A^{[s,r_0]} / t\calR_A^{[s,r_0]} \stackrel \sim \rightarrow \bigoplus_{n_0 \leq n \leq -\log_ps} A \otimes_{\Qp} \Qp(\zeta_{p^n})$, compatible with varying $s$ and $r$.  This implies that $M^{[s,r_0]} / tM^{[s,r_0]} \stackrel \sim \rightarrow \bigoplus_{n_0 \leq n \leq -\log_ps} M_n^{r_0}$. Taking the inverse limit as $s \to 0^+$ and applying Lemma~\ref{L:schneider-teitelbaum}(2) proves (1).

The statement (2) follows from
$
M_n^{r_0}
  \otimes_{\Qp(\zeta_{p^n})} \Qp(\zeta_{p^{n+1}})
\cong
\varphi(M^{r_0}/q_nM^{r_0})
  \otimes_{\varphi(\calR_A^{r_0}/q_n)} \calR_A^{r_0/p}/q_{n+1}
\cong
M^{r_0/p}/q_{n+1}M^{r_0/p}
\cong
M_{n+1}^{r_0/p} = M_{n+1}^{r_0}$.

The statements (3) and (4) are straightforward.
\end{proof}

The following proposition is a na\"ive generalization of \cite[Theorem~3.7]{liu} for families of $(\varphi, \Gamma)$-modules.

\begin{prop}
\label{P:finite cohomology M/t}
Let $M$ be a $(\varphi, \Gamma)$-module over $\calR_A$.  The cohomology groups $H^*_{\varphi, \gamma_{\Qp}}(M/tM)$ are finitely generated $A$-modules, and vanish outside degrees $0$, $1$.
\end{prop}

\begin{proof}
Assume that $M$ is the base change of a $(\varphi, \Gamma)$-module
$M^{r_0}$ over $\calR_A^{r_0}$.  Denote $n_0 = \lceil -\log_p r_0
\rceil$. Then $M^{r_0} / tM^{r_0} \cong \prod_{n \geq n_0}M_n$.  We
first prove that $\varphi-1: M^{r_0}/tM^{r_0} \to
M^{r_0/p}/tM^{r_0/p}$ is surjective.  For $(x_n)_{n \geq n_0+1} \in
M^{r_0/p}/tM^{r_0/p}$, we take $y_{n_0} = 0$ and $y_n = -x_n + x_{n-1}
-x_{n-2}+\cdots +(-1)^{n-n_0} x_{n_0+1}$ for $n \geq n_0+1$ and then
$(\varphi-1)(y_n)_{n \geq n_0} = (x_n)_{n \geq n_0+1}$.  This proves
that the cohomology vanishes outside degrees $0$, $1$.

By the description of $\varphi$-action in Lemma~\ref{L:structure of
  M/tM}(3), we have $(M^r/tM^r)^{\varphi=1} \cong
M_{\lceil-\log_pr\rceil}$ for $0<r\leq r_0$.  Taking the limit as $r
\to 0^+$, we have $(M/tM)^{\varphi=1} \cong \varinjlim_n M_n$, and
hence $H^*_{\varphi, \gamma_{\Qp}}(M/tM)$ is the cohomology of the
complex $[(\varinjlim_n M_n)^\Delta \xrightarrow{\gamma_{\Qp}-1}
  (\varinjlim_n M_n)^\Delta]$. By the description of the
$\Gamma$-action in Lemma~\ref{L:structure of M/tM}(2) and the
knowledge that $\Qp(\mu_{p^n})$ is noncanonically the sum of all the
irreducible $\Qp$-representations of $\Gamma/\Gamma_n$, it suffices to
know the following claim: for any $A$-valued $\Gamma$-representation
$W$, the operator $\gamma_{\Qp}-1$ is bijective on all twists
$W(\psi)$ by finite order characters $\psi$ of sufficiently large
conductor.  Indeed, it suffices to check bijectivity of $\gamma_n-1$
for some $n$. Choose $n$ large enough that the operator norm
$||\gamma_n-1||_W$ is less than $p^{-1/(p-1)}$ (using an argument
similar to Proposition~\ref{P:calR_A(Gamma_K) acts on M}(2), but
easier because the $\Gamma_K$-action is $A$-linear).  Then for any
character $\psi$ such that $\psi|_{\Gamma_n}$ is nontrivial, it is
easy to see that $\gamma_n-1$ is bijective on $W(\psi)$.
\end{proof}

Since we will later be studying Iwasawa cohomology, we need a version of Proposition~\ref{P:finite cohomology M/t} for $\psi$-cohomology rather than $(\varphi,\Gamma)$-cohomology, which gives more understanding of the $\calR_A^\infty(\Gamma)$-module structure.

\begin{prop}
\label{P:finite-torsion}
Let $M$ be a $(\varphi,\Gamma)$-module over $\calR_A$.
The map $\psi-1: M / tM \to M/tM$ is surjective, and its kernel admits a resolution $0 \to P \to Q \to (M/tM)^{\psi=1} \to 0$ by finite projective $\calR_A^\infty(\Gamma)$-modules $P$ and $Q$.
\end{prop}
\begin{proof}
Assume that $M$ is the base change of a $(\varphi, \Gamma)$-module $M^{r_0}$ over $\calR_A^{r_0}$ with $r_0<1/p$.  Set $n_0 = \lceil -\log_p r_0 \rceil \geq 2$, so that $M^{r_0} / tM^{r_0} \cong \prod_{n \geq n_0}M_n$ by Lemma~\ref{L:structure of M/tM}(1).  To see the surjectivity of $\psi -1$, given $(x_n)_{n \geq n_0} \in M^{r_0}/tM^{r_0}$, the tuple $(y_n)$ defined by $y_{n_0} = 0$ and $y_n = x_{n-1} + x_{n-2} + \cdots +x_{n_0}$ for $n \geq n_0+1$ satisfies $(\psi-1)(y_n)_{n \geq n_0} = (x_n)_{n \geq n_0}$.

We now discuss the $\calR_A^\infty(\Gamma)$-module structure of
$(M^{r_0}/tM^{r_0})^{\psi=1}$.  By Lemma~\ref{L:resolution stable
  under finite Galois base change}(1) below, it suffices to find a
two-term resolution by finite projective
$\calR_A^\infty(\Gamma_{n_0})$-modules.

By the description of the $\psi$-operator in Lemma~\ref{L:structure of
  M/tM}(4), $(M^{r_0}/tM^{r_0})^{\psi=1} \cong \varprojlim_{n \geq
  n_0} M_n$ with transition map $\psi$ (which admits a right inverse).
In particular, this implies that, in contrast to the $(\varphi,
\Gamma)$-cohomology case (Proposition~\ref{P:finite cohomology M/t}),
$(M^{r_0}/tM^{r_0})^{\psi=1}$ does \emph{not} depend on the choice of
$r_0$, i.e.\ $(M/tM)^{\psi=1} = \varinjlim_r (M^r/tM^r)^{\psi=1} \cong
\varprojlim_{n \geq n_0} M_n$.  On the other hand,
Lemma~\ref{L:structure of M/tM}(2) gives $\Gamma_{n_0}$-equivariant
isomorphisms $M_n \cong M_{n_0} \otimes_{\Qp(\zeta_{p^{n_0}})}
\Qp(\zeta_{p^n})$, with $\psi$ on the left corresponding to $1 \otimes
\frac1p\Tr_{\Qp(\zeta_{p^n})/\Qp(\zeta_{p^{n-1}})}$ on the right.
Thus
\[
(M/tM)^{\psi=1} \cong \varprojlim_{n \geq n_0,1\otimes\frac1p\Tr}
M_{n_0} \otimes_{\Qp(\zeta_{p^{n_0}})} \Qp(\zeta_{p^n}).
\]
The normal basis theorem allows us to choose a system of
$\Gamma_{n_0}$-equivariant isomorphisms $\Qp(\zeta_{p^n}) \cong
\Qp(\zeta_{p^{n_0}})[\Gamma_{n_0}/\Gamma_n]$ such that
$\frac1p\Tr_{\Qp(\zeta_{p^n})/\Qp(\zeta_{p^{n-1}})}$ on the left
corresponds to the projection onto
$\Qp(\zeta_{p^{n_0}})[\Gamma_{n_0}/\Gamma_{n-1}]$ on the right.
On the other hand, one has a $\Gamma_{n_0}$-equivariant identification
$\Qp(\zeta_{p^{n_0}})[\Gamma_{n_0}/\Gamma_n] \cong
\calR_{\Qp(\zeta_{p^{n_0}})}(\Gamma_{n_0})/(\gamma_{n_0}^{p^{n-n_0}}-1)$.
Thus
\begin{multline*}
(M/tM)^{\psi=1}
\cong
\varprojlim_n
M_{n_0} \otimes_{\Qp(\zeta_{p^{n_0}})}
  \calR_{\Qp(\zeta_{p^{n_0}})}(\Gamma_{n_0})/(\gamma_{n_0}^{p^{n-n_0}}-1) \\
\cong
\varprojlim_n
M_{n_0} \otimes_{\Qp}
  \calR(\Gamma_{n_0})/(\gamma_{n_0}^{p^{n-n_0}}-1)
\cong
\varprojlim_n
M_{n_0} \otimes_A
  \calR_A(\Gamma_{n_0})/(\gamma_{n_0}^{p^{n-n_0}}-1).
\end{multline*}
But $M_{n_0}$ is a finite $A$-module, and $\gamma_{n_0}^{p^{n-n_0}}-1 =
p^{n-n_0}\prod_{i=0}^{n-n_0} q_{i,\gamma_{n_0}}$ (with $p$ a unit in
$\calR(\Gamma_{n_0})$), so
\[
(M/tM)^{\psi=1}
\cong
M_{n_0} \otimes_A
  \varprojlim_n \calR_A(\Gamma_{n_0})/(\gamma_{n_0}^{p^{n-n_0}}-1)
\cong
M_{n_0} \otimes_A \calR_A(\Gamma_{n_0})/\ell_{\gamma_{n_0}}.
\]
We stress that $\Gamma_{n_0}$ acts diagonally on the two tensor
factors above.

Now we view $M_{n_0} \otimes_A \calR_A^\infty
(\Gamma_{n_0})/\ell_{\gamma_{n_0}}$ as the cokernel of the natural map
\begin{equation}\label{E:short resolution}
M_{n_0} \otimes_A \calR_A^\infty(\Gamma_{n_0})
\xrightarrow{\id \otimes \cdot \ell_{\gamma_{n_0}}}
M_{n_0} \otimes_A \calR_A^\infty(\Gamma_{n_0}).
\end{equation}
We will prove that $M_{n_0} \otimes_A \calR_A^\infty(\Gamma_{n_0})$ is
a finite projective $\calR_A^\infty(\Gamma_{n_0})$-module (again for
the diagonal $\Gamma_{n_0}$-action), which would conclude the
proposition.  For this, we use an argument similar to that of
Lemma~\ref{L:(1+pi)phi(D) finite projective}.  By
Lemma~\ref{L:resolution stable under finite Galois base change}(2)
below, it suffices to prove that $M_{n_0} \otimes_A
\calR_A^\infty(\Gamma_{n_0}) \approx M_{n_0} \otimes_A
\calR_A^\infty(\Gamma_m)^{\oplus p^{m-n_0}}$, and hence equivalently
$M_{n_0} \otimes_A \calR_A^\infty(\Gamma_m)$, is finite projective as
an $\calR_A^\infty(\Gamma_m)$-module for some $m\geq n_0$.  Since
$M_{n_0}$ is finite flat over $A$, there exists an $A$-module
$N_{n_0}$ (with trivial $\Gamma_{n_0}$-action) such that $M_{n_0}
\oplus N_{n_0}$ is a free $A$-module with basis $\bbe_1, \dots,
\bbe_d$.  We choose the Banach norm $|\cdot|_{M_{n_0} \oplus N_{n_0}}$
on $M_{n_0} \oplus N_{n_0}$ to be the one given by the basis $\bbe_1,
\dots, \bbe_d$.  By an argument similar to
Proposition~\ref{P:calR_A(Gamma_K) acts on M}(2) (but easier, because
the action is now linear), there exists $m \in \NN$ such that the
operator norm $\|\gamma_m -1\|_{M_{n_0} \oplus N_{n_0}} < 1$; let
$\epsilon >0$ be so that this quantity is $\omega^\epsilon$.  As in
Lemma~\ref{L:(1+pi)phi(D) finite projective}, we consider the
following two maps:
\begin{align*}
& \bigoplus_{i=1}^d \calR_A^\infty(\Gamma_m)\bbe_i
\longrightarrow
(M_{n_0} \oplus N_{n_0}) \otimes_A \calR_A^\infty(\Gamma_m)
\\
\Phi: \ & \sum_{i=1}^d f_i(\gamma_m-1) \bbe_i
\longmapsto
\sum_{i=1}^d f_i(\gamma_m-1) \cdot (\bbe_i \otimes 1)
\\
\Phi': \ & \sum_{i=1}^d f_i(\gamma_m-1) \bbe_i
\longmapsto
\sum_{i=1}^d \bbe_i \otimes f_i(\gamma_m-1).
\end{align*}
Here, $\Phi'$ is a topological isomorphism of $A$-modules (not
necessarily respecting the $\Gamma_m$-action); and $\Phi$ is an
$\calR_A^\infty(\Gamma_m)$-homomorphism which we will now prove to be
an isomorphism.

For each $r>0$, let $|\cdot|_{M_{n_0} \oplus N_{n_0},r}$ denote the
tensor product norm on $(M_{n_0} \oplus N_{n_0}) \otimes_A
\calR_A^\infty(\Gamma_m)$ given by the prescribed norm
$|\cdot|_{M_{n_0} \oplus N_{n_0}}$ on $M_{n_0} \oplus N_{n_0}$ and the
$r$-Gauss norm on $\calR_A^\infty(\Gamma_m)$.  The bound on the
operator norm of $\gamma_m-1$ on $M_{n_0} \oplus N_{n_0}$ implies that
for any $r \in (0, \epsilon)$, we have
\[
|\Phi \circ \Phi'^{-1}(x) - x|_{M_{n_0} \oplus N_{n_0},r}
  < |x|_{M_{n_0} \oplus N_{n_0},r}
\qquad
\text{for any } x \in (M_{n_0} \oplus N_{n_0}) \otimes_A
  \calR_A^\infty(\Gamma_m).
\]
This implies that $\Phi$ is a topological isomorphism. Hence $M_{n_0}
\otimes_A \calR_A^\infty(\Gamma_m)$ is a finite projective
$\calR_A^\infty(\Gamma_m)$-module. This finishes the proof.
\end{proof}

The proof of the following lemma follows from elementary homological
algebra, and is left to the reader.

\begin{lemma}
\label{L:resolution stable under finite Galois base change}
(1) Let $S$ be a finite Galois algebra over a ring $R$ with Galois
group $G$ and let $N$ be an $S$-module.  Assume that $N$, viewed as an
$R$-module, admits a $d$-term resolution $0 \to P'^{-d+1} \to \cdots
\to P'^0 \to N \to 0$ by finite projective $R$-modules.  Then $N$
admits a $d$-term resolution by finite projective $S$-modules.

(2) Let $S$ be a finite flat algebra over a ring $R$, and let $M$ be a
finitely generated $R$-module.  If $\Hom_R(S,M)$ is projective over
$S$ then $M$ is projective over $R$.
\end{lemma}

\begin{remark}
With some mild generalization of Lemma~\ref{L:resolution stable under finite Galois base change}, one can show that the $\calR_A^\infty(\Gamma)$-modules $P$ and $Q$ in Proposition~\ref{P:finite-torsion} have the same rank.
\end{remark}

\subsection{Finiteness of $M / (\psi-1)$}

We now make a calculation to control the cokernel of $\psi-1$. Note
that this argument (which makes no use of $\Gamma$) is similar to
arguments used to control $\varphi$-cohomology in the development of
slope filtration theory, as in
\cite[Proposition~2.1.5]{kedlaya-relative}.

\begin{hypothesis}
In this subsection, we assume that $K = \Qp$.
\end{hypothesis}

\begin{prop} \label{P:psi cokernel}
Let $M$ be a $(\varphi, \Gamma)$-module over $\calR_A$.
\begin{enumerate}
\item[(1)]
The $A$-module $M/(\psi-1)$ is finitely generated.
\item[(2)]
For all sufficiently large integers $n$, the map $\psi-1: t^{-n}M \to t^{-n}M$ is surjective.
\end{enumerate}
\end{prop}
\begin{proof}
For a suitable choice of $r_0 \in (0,1)$, we can realize $M$ as the base extension of a $\varphi$-module $M^{r_0}$
over $\calR^{r_0}_A$. Put $r = r_0/p^2$, choose generators $\bbe'_1, \dots, \bbe'_n$ of $M^{r_0}$, and take
$\bbe_i = \varphi(\bbe'_i) \in M^{pr}$. Let $N^{pr}$ be the free module over $\calR^{pr}_A$ on the generators
$\bbe_1,\dots,\bbe_n$ and choose an $\calR_A^{pr}$-linear splitting $N^{pr} \cong M^{pr} \oplus P^{pr}$ of the surjection $\mathrm{proj}: N^{pr} \to M^{pr}$.
For $s \in (0,pr]$, let $|\cdot|_s$
denote the $s$-Gauss norm on $N^{pr}$ for the basis $\bbe_1,\dots,\bbe_n$, and restrict this norm to $M^{pr}$.
Let $||\mathrm{proj}||_s$ denote the operator norm of $\mathrm{proj}: N^{pr} \to M^{pr}$ with respect to the $s$-Gauss norm.

Choose a matrix $F'$ with entries in $\calR^{pr}_A$ such that $\bbe_j = \sum_i F'_{ij} \bbe'_i$
and put $F = \varphi(F')$, so that $\varphi(\bbe_j) = \sum_i F_{ij} \bbe_i$ and $F$ has entries in $\calR_A^r$.
Choose also a matrix $G$ with entries in $\calR^{pr}_A$ such that $\bbe'_j = \sum_i G_{ij} \bbe_i$, so that $\bbe_j =
\varphi(\sum_i G_{ij} \bbe_i)$.
We then have
\begin{align}
\label{E:expression of phi}
\varphi\left(\sum_j c_j \bbe_j\right)&= \sum_i \left( \sum_j F_{ij}\varphi(c_j) \right) \bbe_i \qquad (c_1,\dots,c_n \in \calR_A^r), \\
\label{E:expression of psi}
\psi\left(\sum_j c_j \bbe_j\right) &= \sum_i \left( \sum_j G_{ij}\psi(c_j)\right) \bbe_i \qquad (c_1, \dots, c_n \in 
\calR_A^{r/p}).
\end{align}
Choose $\epsilon \in (0,1)$.
Note that for $a \in \RR$ sufficiently small and $b \in \RR$ sufficiently large (they need not be integers), one has
\[
\omega^{a(p^{-1}-1)r} ||\mathrm{proj}||_r \cdot \|G\|_{r}\leq \epsilon, \qquad \omega^{b(p-1)r}
||\mathrm{proj}||_r \cdot \|F\|_{r} \leq \epsilon.
\]
Choose some such $a,b$ and define the $A$-linear map $\Pi: M^r \to
M^r$ by inclusion into $N^r$ followed by the formula
\[
\Pi\left( \sum_{j=1}^n \sum_{i \in \ZZ} a_{ij} \pi^i \bbe_j\right) = \mathrm{proj}\  \sum_{j=1}^n \sum_{i \in [a,b] \cap \ZZ} a_{ij} \pi^i \bbe_j.
\]
Note that $\Pi$ projects $M^r$ onto a finite $A$-submodule of $M^r$
and that $|\Pi(\bbv)|_r\leq ||\mathrm{proj}||_r \cdot |\bbv|_r$ for
$\bbv \in M^r$.  Define the map $\lambda: M^r \to M^r$ by inclusion
into $N^r$ followed by the formula
\begin{equation}
\label{E:definition of lambda}
\lambda\left( \sum_{j=1}^n \sum_{i \in \ZZ} a_{ij} \pi^i \bbe_j \right)
= \mathrm{proj}\ \sum_{j=1}^n \sum_{i<a} a_{ij} \pi^i \bbe_j -
\mathrm{proj}\  \sum_{j=1}^n \sum_{i>b} \varphi(a_{ij} \pi^i \bbe_j).
\end{equation}
Here we use \eqref{E:expression of phi} to write $\varphi(a_{ij} \pi^i
\bbe_j)$ as $\sum_{l}F_{lj} \varphi(a_{ij}\pi^i)\bbe_l$; hence the
right hand side of \eqref{E:definition of lambda} lies in $M^r$.  From
the definition, we have the following formula for a given element
$\bbv = \sum_{j= 1}^n\sum_{i \in \ZZ}a_{ij}\pi^i\bbe_j$ in $ M^r$:
\begin{align}
\nonumber
\bbv - \Pi(\bbv) + (\psi-1)(\lambda(\bbv)) &= \mathrm{proj} \  \sum_{j=1}^n \sum_{i<a} \psi( a_{ij} \pi^i \bbe_j) -
\mathrm{proj}\  \sum_{j=1}^n \sum_{i>b} \varphi(a_{ij} \pi^i \bbe_j)\\
\label{E:expression of iteration}
&=\mathrm{proj} \  
\sum_{j, l=1}^n
\sum_{i<a}
G_{lj}\psi(a_{ij}\pi^i) \bbe_j
 - \mathrm{proj}\
\sum_{j,l=1}^n
\sum_{i>b}
F_{lj} \varphi(a_{ij} \pi^i)\bbe_j.
\end{align}
By our choice of $a,b$, we have
\[
|\bbv - \Pi(\bbv) + (\psi-1)(\lambda(\bbv))|_r  \leq \epsilon |\bbv|_r.
\]
Given $\bbv \in M^r$, put $\bbv_0 = \bbv$ and 
\begin{equation}
\label{E:bound on v_l+1}
\bbv_{l+1} = \bbv_l - \Pi(\bbv_l) + (\psi-1)(\lambda(\bbv_l)).
\end{equation}
Then $\bbw = \sum_{l=0}^\infty \bbv_l$, provided this sum converges, is an element
of $M^r$ satisfying
\[
\bbv-\Pi(\bbw)+(\psi-1)(\lambda(\bbw)) = 0.
\]
This sum converges for the $r$-Gauss norm by \eqref{E:bound on v_l+1}.
We will check the convergence for $s$-Gauss norm when $ s < r/p$; this is enough because knowing the convergence for smaller $s$ gives it for larger $s$.
For any $s \in (0,r/p)$, we again choose $a'\leq a$ sufficiently small and $b'\geq b$ sufficiently large so that 
\[
\omega^{a'(p^{-1}-1)s}||\mathrm{proj}||_s\cdot  \|G\|_{s} \leq \epsilon \quad \textrm{and} \quad \omega^{b'(p-1)s}||\mathrm{proj}||_s\cdot  \|F\|_{s} \leq \epsilon.
\]
We separate the
powers of $\pi$ in the expression \eqref{E:expression of iteration} defining $\bbv_l$ into the ranges $[a', a)$, $(b,b']$, and $(-\infty, a') \cup (b', \infty)$.  By the same argument, the summation over all powers of $\pi$ in   $(-\infty, a') \cup (b', \infty)$ has $s$-Gauss norm less than or equal to $\epsilon |\bbv_l|_s$.  The summation over powers of $\pi$ in $[a', a)$ has $s$-Gauss norm less than or equal to $ \omega^{a(p^{-1}s-r)}\|G\|_s |\bbv_l|_r$ (note that we are comparing it with $|\bbv_l|_r$ as opposed to $|\bbv_l|_s$).  The summation over powers of $\pi$ in $(b, b']$ has $s$-Gauss norm less than or equal to $ \omega^{b'(sp-r)}\|F\|_s|\bbv_l|_r$.  In summary, we have
\[
|\bbv_{l+1}|_s \leq \max\{ \omega^{b'(sp-r)}\|F\|_s|\bbv_l|_r,\ \omega^{a(p^{-1}s-r)}\|G\|_s |\bbv_l|_r,\ \epsilon |\bbv_l|_s\}.
\]
From this, it follows that the series defining $\bbw$ converges also under $|\cdot|_s$; since $s$ was an arbitrary
element of $(0,r/p)$, we deduce that $\bbw \in M^r$.

Given any $\bbv \in M$, we can find $s \in (0,r]$ such that $\bbv \in M^s$.
Choose a nonnegative integer $m$ for which $p^m s \geq r$; then $\psi^m(\bbv)$ is an element of $M^r$
representing the same class in $M/(\psi-1)M$ as $\bbv$. By the previous paragraph, however, this class is represented
by an element in the image of $\Pi$; we thus deduce (1).

To prove (2), note that replacing $M$ with $t^{-n}M$ has the effect of replacing $F$ with $p^{-n} F$ and 
$G$ with $p^n G$. It is thus sufficient to check that for $n$ large, we can find a single value
$a \in \RR$ with 
\[
\omega^{a(p^{-1}-1)r} p^{-n}||\mathrm{proj}||_r \cdot \|G\|_r \leq \epsilon, \qquad \omega^{a(p-1)r} p^n ||\mathrm{proj}||_r \cdot \|F\|_r \leq \epsilon,
\]
as then the previous argument applies with an empty generating set. Namely, such a choice of $a$ exists if and only if
\[
p^{-n(p-1)} ||\mathrm{proj}||_r^{p+1}\cdot\|G\|_r^p \cdot \|F\|_r \leq \epsilon^{p+1},
\]
and it is clear that this holds for $n$ large enough.
\end{proof}

\section{Finiteness of cohomology}
\label{S:finiteness of cohomology}

In this section, we give our main results on finiteness of $(\varphi,
\Gamma)$-cohomology and Iwasawa cohomology for relative $(\varphi,
\Gamma)$-modules.  To clarify the presentation, we leave one key
statement (Theorem~\ref{T:finite Iwasawa cohomology}) unproven for the
moment, in order to illustrate the overall structure of the arguments;
we take up the technical arguments needed to prove
Theorem~\ref{T:finite Iwasawa cohomology} in the next section.

\subsection{Formalism of derived categories}

We give a short discussion of the derived category of complexes of modules.

\begin{hypothesis}
In this subsection, we fix a commutative ring $R$.
\end{hypothesis}

\begin{notation}
Let $\bbD_{\perf}^{[a,b]}(R)$ (resp. $\bbD_\perf^\flat(R)$, $\bbD_\perf^-(R)$) denote the subcategory of the derived category of complexes of $R$-modules consisting of the complexes of $R$-modules which are quasi-isomorphic to complexes of finite projective $R$-modules concentrated in degrees in $[a,b]$ (resp. bounded degrees, degrees bounded above).
\end{notation}

\begin{lemma}\label{L:flat bottom cokernel}
Let $P^\bullet$ be a complex of projective (resp. flat) $R$-modules
concentrated in degree $[0,d]$ and $Q^\bullet$ a complex of projective
(resp. flat) $R$-modules bounded above.  Suppose that we have a
quasi-isomorphism $P^\bullet \to Q^\bullet$ or $Q^\bullet \to
P^\bullet$.  Then the complex $Q'^\bullet = [ \Coker d_Q^{-1}
\xrightarrow{d_Q^0} Q^1 \xrightarrow{d_Q^1} \cdots ]$ is quasi isomorphic
to $Q^\bullet$ and $\Coker d_Q^{-1}$ is a projective (resp. flat)
$R$-module.
\end{lemma}
\begin{proof}
Since $P^\bullet$ and $Q^\bullet$ are quasi-isomorphic,
$H^i(Q^\bullet) = 0$ for $i<0$.  This implies that $Q'^\bullet$ is
quasi-isomorphic to $Q^\bullet$.  To see the projectivity
(resp. flatness) of $\Coker d_Q^{-1}$, we observe that, under either
condition, $\Fib(P^\bullet \to Q'^\bullet)$ or $\Fib(Q'^\bullet \to
P^\bullet)$ is an \emph{acyclic} complex of projective (resp. flat)
$R$-modules, except possibly the term $\Coker d_Q^{-1}$ which appears
either in the leftmost nonzero term, or as a direct summand of the
second leftmost nonzero term.  In either case, it is easy to deduce
that $\Coker d_Q^{-1}$ is a projective (resp. flat) $R$-module.
\end{proof}

\begin{lemma}
\label{L:complex with cohomology capped above}
If a complex in $ \bbD_\perf^-(R)$ has trivial cohomology in degrees strictly greater than $b$, then it is quasi-isomorphic to a complex of finite projective $R$-modules $P^\bullet$ with $P^{b+1} = P^{b+2} = \cdots =0$.  If the given complex lies in $\bbD_\perf^\flat(R)$, we may choose $P^\bullet $ such that, in addition, we have $P^{-n} = 0$ for $n$ sufficiently large.
\end{lemma}
\begin{proof}
The given complex is  quasi-isomorphic to a complex $[\cdots
  \xrightarrow{d^{b'-2}} Q^{b'-1} \xrightarrow{d^{b'-1}} Q^{b'}]$ of
finite projective $R$-modules in degrees bounded above, and also below
if the given complex lies in $\bbD_\perf^\flat(R)$.  Then the given
complex is also quasi-isomorphic to the complex $[\cdots
  \xrightarrow{d^{b-2}} Q^{b-1} \xrightarrow{d^{b-1}} \Ker(d^b)]$.  It
suffices to prove that $\Ker d^b$ is finite projective.  Indeed, the
condition implies that $0 \to \Ker d^b \xrightarrow{d^b} Q^b \to
Q^{b+1} \to \cdots \to Q^{b'} \to0$ is a long exact sequence, in which
all terms except $\Ker d^b$ are known to be finite projective.
Therefore, $\Ker d^b$ is also finite projective.
\end{proof}

\begin{lemma}
\label{L:acyclic pointwise}
Let $P^\bullet \to Q^\bullet$ be a morphism of complexes in $\bbD_\perf^-(R)$.  Then it is a quasi-isomorphism if and only if $P^\bullet \stackrel\bbL\otimes_R  R/\gothm \to Q^\bullet \stackrel\bbL\otimes_R  R/\gothm$ is a quasi-isomorphism for every maximal ideal $\gothm$ of $ R$.
\end{lemma}
\begin{proof}
We may assume that both complexes are complexes of finite projective
$R$-modules in degrees bounded above.  It suffices to show that the
mapping fiber complex $S^\bullet = \Fib(P^\bullet \to Q^\bullet)$ is acyclic.
Suppose not.  Let $n$ be the maximal integer such that
$H^n(S^\bullet)$ is nontrivial.  By Lemma \ref{L:complex with
  cohomology capped above}, $H^n(S^\bullet)$ is finitely generated.
Hence there exists a maximal ideal $\gothm$ of $R$ such that
$H^n(S^\bullet) \otimes_R R/\gothm \neq 0$.  Now, the spectral
sequence $\Tor_j^R( H^i(S^\bullet), R/\gothm) \Rightarrow
H^{i-j}(S^\bullet \otimes_R R/\gothm) = 0$ would lead to a
contradiction at degree $n$, which is nonzero according to the
spectral sequence and is zero by assumption.
\end{proof}

\begin{lemma}
\label{L:finite A-mod is afp A=field}
If $A$ is a finite extension of $\Qp$ and $M$ is an $\calR_A^{r_0}$-module which is finite over $A$,  we have $M \in \bbD_\perf^\flat(\calR_A^{r_0})$.
\end{lemma}
\begin{proof}
In this case, $M$ is supported at a finite set of closed points on $\Max(A) \times \rmA^1(0,r_0]$, so
the lemma is obvious.
\end{proof}

\begin{defn}
For $P$ a finite projective $R$-module, we denote its rank by $\rank_R P$, viewed as a function on the set of irreducible components of $\Spec R$ (and when $\Spec R$ is irreducible we regard it simply as a nonnegative integer).

For $P^\bullet \in \bbD_\perf^{[a,b]}(R)$ a complex of finite projective $R$-modules concentrated in degrees in $[a,b]$, we define its \emph{Euler characteristic} by $\chi_R(P^\bullet) = \sum_{i=a}^b (-1)^i \rank_R P^i$; it is invariant under quasi-isomorphisms, and is additive for distinguished triangles.
\end{defn}

\begin{defn}
There is a duality functor $\RHom_R(-,R): \bbD_\perf^{[a,b]}(R) \to \bbD_\perf^{[-b,-a]}(R)$ given by sending a bounded complex $P^\bullet$ of finite projective $R$-modules to $(\Hom_R(P^{-i}, R))_{i \in\ZZ}$.  The functor is well-defined and $\RHom_R(-,R) \circ \RHom_R(-,R)$ is the identity.
\end{defn}

\subsection{Iwasawa cohomology of $(\varphi, \Gamma)$-modules}

Following the work of Fontaine (see \cite{cherbonnier-colmez}) on computing the Iwasawa cohomology of a Galois representation using $(\varphi, \Gamma)$-modules,
the second author pointed out in \cite{pottharst2} that the Iwasawa cohomology of general $(\varphi, \Gamma)$-modules has many number-theoretic applications.  This motivates a systematic study of the Iwasawa cohomology of a $(\varphi, \Gamma)$-module.

\begin{notation} \label{N:cyclotomic deformation space}
Let $\Max(\calR_A^\infty(\Gamma_K))$ denote the union of the
$\Max(\calR_A^{[s,\infty]}(\Gamma_K))$ for all $s>0$.  In other words,
it is the set of closed points of the quasi-Stein rigid analytic space
associated to $\calR_A^\infty(\Gamma_K)$, viewed as a disjoint union
of relative discs.  A closed point of $\Max(\calR_A^\infty(\Gamma_K))$
is the same datum as an equivalence class of pairs $(z, \eta)$, where
$z \in \Max(A)$ and $\eta$ is a character of $\Gamma_K$ with values in
some finite extension of $\Qp$, and two such pairs are considered
equivalent if they become equal upon embedding their target fields
into some common finite extension of $\Qp$.  We let
$\gothm_{(z,\eta)}$ be the corresponding maximal ideal and
$\kappa_{(z,\eta)}$ the residue field (which is finite over $\Qp$), so
that $\eta$ is given by the composition $\Gamma_K \to
\calR_A^\infty(\Gamma_K)^\times \to
(\calR_{A}^\infty(\Gamma_K)/\gothm_{(z, \eta)})^\times =
\kappa_{(z,\eta)}^\times$.  Let $\bar \gothm_{(z, \eta)}$ denote the
corresponding maximal ideal of $ \calR_{A/\gothm_z}^\infty(\Gamma_K)$.

For $M$ a $(\varphi, \Gamma_K)$-module over $\calR_A(\pi_K)$, let $M_z
\otimes \eta$ denote the $(\varphi,\Gamma_K)$-module over
$\calR_{\kappa_{(z,\eta)}}(\pi_K)$-module given by $M_z
\otimes_{\kappa_z} \kappa_{(z, \eta)}$ with $\Gamma_K$ acting on
$\kappa_{(z,\eta)}$ through $\eta$, and $\varphi$ acting on
$\kappa_{(z,\eta)}$ trivially.

When $A$ is a finite extension of $\Qp$, we omit $z$ from the notation
by simply writing, for example, $M\otimes \eta$ and $\gothm_\eta$.
\end{notation}

\begin{defn}
\label{D:Iwasawa cohomology}
Let $M$ be a $(\varphi, \Gamma_K)$-module over $\calR_A(\pi_K)$.  We
write $\rmC^\bullet_\psi(M)$ for the complex $[M
  \xrightarrow{\psi-1}M]$ concentrated in degrees $1$ and $2$.  Its
cohomology groups $H^*_\psi(M)$ are called the \emph{Iwasawa
  cohomology} of $M$; using Proposition~\ref{P:calR_A(Gamma_K) acts on
  M}, we view them as modules over $\calR_A^\infty(\Gamma_K)$.
\end{defn}

The following base change result from Iwasawa cohomology to Galois cohomology follows immediately from the definitions, and will often be used implicitly in the remainder of this paper.

\begin{prop}
\label{P:psi-gamma complex v.s. psi-complex}
For $\eta \in \Max(\calR^\infty(\Gamma_K))$ a continuous character of $\Gamma_K$ with coefficients in $L = \calR^\infty(\Gamma_K)/\gothm_\eta$, we have a natural quasi-isomorphism
\begin{equation}
\label{E:psi-gamma complex v.s. psi-complex}
\rmC^\bullet_\psi(M) \Lotimes_{\calR^\infty_A(\Gamma_K)}
\calR^\infty_A(\Gamma_K)/\gothm_\eta \calR^\infty_A(\Gamma_K)
\longrightarrow \rmC^\bullet_{\psi, \gamma_K}(M \otimes_{\Qp} L(\eta^{-1})),
\end{equation}
where $L(\eta^{-1})$ is a one-dimensional $L$-vector space with the trivial $\varphi$-action and the $\Gamma_K$-action by $\eta^{-1}$.
\end{prop}

\begin{notation}
The involution $\iota: \Gamma_K \to \Gamma_K$ given by $\iota(\gamma)
= \gamma^{-1}$ for $\gamma \in \Gamma_K$ induces an involution $\iota:
\calR_A^?(\Gamma_K) \to \calR_A^?(\Gamma_K)$, where $? =
     [s,r],r,\emptyset$.  For an $\calR_A^?(\Gamma_K)$-module $N$, we
     write $N^\iota$ for the $\calR_A^?(\Gamma_K)$-module with the
     same underlying abelian group $N$ but with module structure
     twisted through $\iota$.
\end{notation}

For the remainder of this subsection we specialize to the case $K=\Qp$
and, in this case, describe an explicit construction of duality
pairings for Iwasawa cohomology.  We use the notations of
Definition~\ref{D:omega-Gamma} with $C=\Gamma$, $\ell = \log \circ
\chi$, and $c=\gamma_{\Qp}$.  Recall that the residue map on
$\calR_A(\Gamma)$ is denoted $\Res_\Gamma$.

\begin{lemma}
Let $M$ be a $(\varphi, \Gamma)$-module over $\calR_A$.  There is a
continuous isomorphism of $\calR_A(\Gamma)$-modules
\begin{equation}\label{E:D^psi=0 duality}
(M^*)^{\psi=0} \cong \Hom_{\calR_A(\Gamma)}(M^{\psi=0}, \calR_A(\Gamma))^\iota
\end{equation}
sending $y \in (M^*)^{\psi=0}$ to the unique $\calR_A(\Gamma)$-linear
homomorphism $\{-,y\}_\Iw^0: M^{\psi=0} \to \calR_A(\Gamma)$ such that
for all $x \in M^{\psi=0}$ one has
$\Res_\Gamma(\{x,y\}_\Iw^0\omega_{\log\circ\chi}) = \{x,y\}_{\Qp}$.
In particular,
\[
\{-,-\}_\Iw^0: M^{\psi=0} \times (M^*)^{\psi=0,\iota} \to \calR_A(\Gamma)
\]
is $\calR_A(\Gamma)$-bilinear.
\end{lemma}
\begin{proof}
We first note that the pairing $\{-,-\}_{\Qp}: M \times M^* \to A$ is
perfect.  This perfectness does not involve any $\varphi$- or
$\Gamma$-action, only that $M$ is finite projective over $\calR_A$;
one reduces to the free case by adding a complementary projective
module, then reduces to the free rank one case, which is
Lemma~\ref{L:calR duality}.

Returning to the $(\varphi,\Gamma)$-action on $M$, we recall that for
any $x \in M$ and $y \in M^*$ one has $\{\gamma x, \gamma y\}_{\Qp} =
\{x,y\}_{\Qp}$ for all $\gamma \in \Gamma$, $\{\varphi x, y\}_{\Qp} =
\{x, \psi y\}_{\Qp}$, and $\{\psi x, y\}_{\Qp} = \{x, \varphi
y\}_{\Qp}$.  From the identity for $\gamma \in \Gamma$ it follows that
the duality isomorphism is $A[\Gamma]$-linear (with the shown
involution $\iota$), and hence by continuity of the action it is
$\calR_A^\infty(\Gamma)$-linear.  From the identities for $\varphi$
and $\psi$ and the direct sum decompositions $M = \varphi(M) \oplus
M^{\psi=0}$ and $M^* = \varphi(M^*) \oplus (M^*)^{\psi=0}$, it follows
that $M^{\psi=0}$ and $\varphi(M^*)$ are exact orthogonal complements.
Therefore, $\{-,-\}_{\Qp}$ restricts to a perfect pairing between
$M^{\psi=0}$ and $(M^*)^{\psi=0}$, which in turn gives by adjunction a
topological isomorphism $(M^*)^{\psi=0} \stackrel\sim\to
\Hom_{A,\cont}(M^{\psi=0},A)^\iota$.  This isomorphism is
$\calR_A^\infty(\Gamma)[(\gamma_{\Qp}-1)^{-1}]$-linear, and hence by
continuity of the action it is $\calR_A(\Gamma)$-linear.

Combining the preceding isomorphism with tensor-hom adjunction and
residue duality \eqref{E:gamma residue pairing} for $\calR_A(\Gamma)$,
we obtain
\begin{align}
\nonumber
(M^*)^{\psi=0}
&\cong \Hom_{A,\cont}(M^{\psi=0}, A)^\iota
 \cong \Hom_{A,\cont}(M^{\psi=0} \otimes_{\calR_A(\Gamma)} \calR_A(\Gamma), A)^\iota \\
\label{E:partial duality}
&\cong \Hom_{\calR_A(\Gamma)}(M^{\psi=0}, \Hom_{A, \cont}(\calR_A(\Gamma), A))^\iota\\
\nonumber
&\cong\Hom_{\calR_A(\Gamma)}(M^{\psi=0}, \calR_A(\Gamma))^\iota.
\end{align}
One verifies immediately from the definitions that this isomorphism
has desired form.
\end{proof}

\begin{cor}
\label{C:D^psi=0 duality}
Let $M$ be a $(\varphi, \Gamma)$-module over $\calR_A$.  Then there exists $r_0>0$ such that $(M^{*r_0})^{\psi=0} \cong \Hom_{\calR_A^{r_0}(\Gamma)}((M^{r_0})^{\psi=0}, \calR_A^{r_0}(\Gamma))^\iota$.
\end{cor}
\begin{proof}
It is easy to see that for sufficiently small $r_0>0$, the isomorphism in \eqref{E:D^psi=0 duality} sends the $\calR_A^{r_0}(\Gamma)$-structure of the left hand side into that of the right hand side.  By Theorem~\ref{T:structure of D^psi=0} and Remark~\ref{R:isom on Robba descent}, after perhaps shrinking $r_0$, the resulting morphism is an isomorphism.
\end{proof}

\begin{lemma}\label{L:psi=1 pair to infty}
For $x \in M^{\psi=1}$ and $y \in (M^*)^{\psi=1}$, then $(\varphi-1)x
\in M^{\psi=0}$ and $(\varphi-1)y \in (M^*)^{\psi=0}$, and we have
$\{(\varphi-1)x,(\varphi-1)y\}_\Iw^0 \in \calR_A^\infty(\Gamma)
\subset \calR_A(\Gamma)$.
\end{lemma}
\begin{proof}
The claim that $(\varphi-1)x \in M^{\psi=0}$ and $(\varphi-1)y \in
M^{\psi=0}$ follows immediately from the identity $\psi \circ \varphi
= \id$.  Unwinding the last step of \eqref{E:partial duality} through
\eqref{E:gamma residue pairing}, the second claim is equivalent to
showing that for all $r \in \calR_A^\infty(\Gamma) \subseteq
\calR_A(\Gamma)$, one has $\{r \cdot (\varphi-1)x,(\varphi-1)y\}_{\Qp}
= 0$.  Indeed, $r$ commutes with $\varphi$ and preserves $M^{\psi=1}$,
so
\begin{align*}
\{r \cdot (\varphi-1)x, (\varphi-1)y\}_{\Qp}
&= \{(\varphi-1)(rx), (\varphi-1)y\}_{\Qp} \\
&= \{rx,y\}_{\Qp} - \{\varphi(rx),y\}_{\Qp}
 - \{rx, \varphi(y)\}_{\Qp} + \{\varphi(rx), \varphi(y)\}_{\Qp} \\
&=\{rx,y\}_{\Qp} - \{rx, \psi(y)\}_{\Qp} -\{\psi(rx), y\}_{\Qp} + \{rx,y\}_{\Qp}
 = 0
\end{align*}
because $x$, $rx$, and $y$ are fixed by $\psi$.
\end{proof}

\begin{defn}
\label{D:Iwasawa pairing}
By Lemma~\ref{L:psi=1 pair to infty} above, the rule
\[
\{-,-\}_\Iw:
M^{\psi=1} \times (M^*)^{\psi=1,\iota} \to \calR_A^\infty(\Gamma),
\qquad (x,y) \mapsto \{(\varphi-1)x,(\varphi-1)y\}_\Iw^0
\]
gives a continuous, $\calR_A^\infty(\Gamma)$-bilinear pairing called
the \emph{Iwasawa pairing}.  From the definition, if $x \in
M^{\psi=1}$ and $y \in (M^*)^{\psi=1}$ then $\{x,y\}_\Iw$ is the
unique element of $\calR_A^\infty(\Gamma)$ such that for all $r \in
\calR_A(\Gamma)$ one has
$\Res_\Gamma(r\{x,y\}_\Iw\omega_{\log\circ\chi}) = \{r \cdot
(\varphi-1)x,\varphi-1)y\}_{\Qp}$.
\end{defn}

\begin{prop}
\label{P:Iwasawa pairing = Tate pairing}
For any point $(z,\eta) \in \Max(\calR_A^\infty(\Gamma))$, the Iwasawa
pairing $\{-,-\}_\Iw$ at $(z,\eta)$ is compatible with the Tate
pairing for $M_z \otimes \eta^{-1}$, in the sense that the diagram
\[
\xymatrix{
M^{\psi=1}/  \gothm_{(z,\eta)} \ar@{^(->}[d]_{\log\chi(\gamma_{\Qp})}
 \ar@{}[r]|{\displaystyle\times}
 & (M^*)^{\psi=1}/\gothm_{(z,\eta^{-1})} \ar@{^(->}[d]^{\log\chi(\gamma_{\Qp})} \ar[rrr]^-{\{-,-\}_\Iw \bmod \gothm_{(z,\eta)}}
 &&& \calR_{A}^\infty(\Gamma)/ \gothm_{(z,\eta)}\ar@{=}[dd] \\
H^1_{\psi,\gamma_{\Qp}}(M_z \otimes \eta^{-1})
 \ar@{}[r]|{\displaystyle\times}
 & H^1_{\psi,\gamma_{\Qp}}(M^*_z \otimes \eta) \\
H^1_{\varphi,\gamma_{\Qp}}(M_z \otimes \eta^{-1}) \ar[u]^{\Psi_{M_z \otimes \eta^{-1}}}_\cong
 \ar@{}[r]|{\displaystyle\times}
 & H^1_{\psi,\gamma_{\Qp}^{-1}}(M^*_z \otimes \eta) \ar[u]_{\Gamma_{\gamma_{\Qp}^{-1},\gamma_{\Qp},M_z^* \otimes \eta}}^\cong \ar[rrr]^-{C_p\cdot\cup_\Ta}
 &&& \kappa_{(z,\eta)}
}
\]
commutes.  (The the upper horizontal arrow is obtained by identifying
$(M^*)^{\psi=1,\iota}/\gothm_{(z,\eta)} \cong
(M^*)^{\psi=1}/\gothm_{(z,\eta^{-1})}$, and the first vertical arrows
are $\log\chi(\gamma_{\Qp})$ times the na\"ive maps induced by
\eqref{E:psi-gamma complex v.s. psi-complex}, or by
Remark~\ref{R:useful-sequence}.)
\end{prop}
\begin{proof}
We assume that $\Delta$ is trivial for simplicity, as the proof for
the case of nontrivial $\Delta$ is similar.  By twisting the action of
$\Gamma$ on $M$, we may assume that $\eta$ is trivial.

Let $x \in M^{\psi=1}$ and $y \in (M^*)^{\psi=1}$.  Their images in
$H^1_{\psi, \gamma_{\Qp}}(M_z)$ and $H^1_{\psi, \gamma_{\Qp}}(M^*_z)$
are $\log\chi(\gamma_{\Qp})\overline{(0, x)}$ and
$\log\chi(\gamma_{\Qp})\overline{(0, y)}$, respectively.  Since
$(\varphi-1)x \in M^{\psi=0}$, and $\gamma_{\Qp}-1$ acts invertibly on
$M^{\psi=0}$, the element $(\gamma_{\Qp}-1)^{-1}(\varphi-1)x$ makes
sense, and $\Psi_{M_z}$ sends
$\overline{((\gamma_{\Qp}-1)^{-1}(\varphi-1)x,x)}$ to
$\overline{(0,x)}$.  By definition,
$\Gamma_{\gamma_{\Qp}^{-1},\gamma_{\Qp},M_z^*}$ sends
$\overline{(0,\frac{\gamma_{\Qp}^{-1}-1}{\gamma_{\Qp}-1}y)}$ to
$\overline{(0,y)}$.  We now compute that the bottom-left route around
the diagram sends $(\log\chi(\gamma_{\Qp}))^{-2}(x,y)$ to
\begin{align*}
&\{
((\gamma_{\Qp}-1)^{-1}(\varphi-1)x,x),
(0,\frac{\gamma_{\Qp}^{-1}-1}{\gamma_{\Qp}-1}y)
\}_{\gamma_{\Qp},1} \\
&\qquad=
\{
(\gamma_{\Qp}-1)^{-1}(\varphi-1)x,
\frac{\gamma_{\Qp}^{-1}-1}{\gamma_{\Qp}-1}y
\}_{\gamma_{\Qp}} \\
&\qquad=
\{
(\gamma_{\Qp}^{-1}-1)^{-1}(\varphi-1)x,
y
\}_{\gamma_{\Qp}} \\
&\qquad=
\{
(\gamma_{\Qp}^{-1}-1)^{-1}(\varphi-1)x,
(1-\varphi)y
\}_{\gamma_{\Qp}} \\
&\qquad=
(\log\chi(\gamma_{\Qp}))^{-1}\{
(\gamma_{\Qp}^{-1}-1)^{-1}(\varphi-1)x,
(1-\varphi)y
\}_{\Qp} \\
&\qquad=
(\log\chi(\gamma_{\Qp}))^{-1}\Res_\Gamma(\frac{-1}{\gamma_{\Qp}^{-1}-1}\{x,y\}_{\Iw,\gamma_{\Qp}}\omega_{\log\circ\chi}),
\end{align*}
where the first equality is by definition of the residue pairing, the
second equality is because of the identity $\{x,\gamma
y\}_{\gamma_{\Qp}} = \{\gamma^{-1}x,y\}_{\gamma_{\Qp}}$, the third
equality is because $(\varphi-1)x$ belongs to the orthogonal
complement $M^{\psi=0}$ of $\varphi(M^*)$, and the final two
equalities are by the definitions of normalized trace and the Iwasawa
pairings, respectively.

The claim of the proposition reduces to showing that for any $f \in
\calR_A^\infty(\Gamma)$ the quantity
\[
(\log\chi(\gamma_{\Qp}))
\Res_\Gamma(\frac{-1}{\gamma_{\Qp}^{-1}-1}f\omega_{\log\circ\chi}) \in A
\]
is equal to the image of $f$ under the natural projection to
$\calR_A^\infty(\Gamma)/(\gamma_{\Qp}-1) \cong A$.  To see this in
general, it suffices to take $f=1$, in which case
\begin{align*}
(\log\chi(\gamma_{\Qp}))
\Res_\Gamma(\frac{-1}{\gamma_{\Qp}^{-1}-1}\omega_{\log\circ\chi})
&=
(\log\chi(\gamma_{\Qp}))
\Res_\Gamma(\frac{\gamma_{\Qp}}{\gamma_{\Qp}-1}\frac{d\gamma_{\Qp}}{(\log\chi(\gamma_{\Qp}))\gamma_{\Qp}}) \\
&=
\Res_\Gamma(\frac{d(\gamma_{\Qp}-1)}{\gamma_{\Qp}-1})
=
1,
\end{align*}
as was desired.
\end{proof}

\subsection{Iwasawa cohomology over a field}

In this subsection, we study properties of Iwasawa cohomology for $(\varphi, \Gamma)$-modules over the standard Robba rings.  Most of the results are already included in \cite{pottharst2}; we reproduce them here for the convenience of the reader.

\begin{hypothesis}
In this subsection, assume that  $A$ is a finite extension of $\Qp$ and $M$ is a $(\varphi, \Gamma_K)$-module over $\calR_A(\pi_K)$.
\end{hypothesis}

\begin{notation}
Since we will be making use of the exact sequence $0 \to M^{\varphi=1}
\to M^{\psi=1} \xrightarrow{\varphi-1} M^{\psi=0}$, we set $\scrC =
(\varphi-1)M^{\psi=1} \subseteq M^{\psi=0}$.
\end{notation}

\begin{remark}
\label{R:Bezout+torsion free=>free}
Under this hypothesis, the rings $R = \calR_A^{r_0}(\pi_K)$ and
$\calR_A^\infty(\Gamma_K)$ are products of B\'ezout domains that are
one-dimensional Fr\'echet-Stein algebras, and the rings $R =
\calR_A(\pi_K)$ and $\calR_A(\Gamma_K)$ are direct limits of such
rings.  Therefore, any finitely generated or coadmissible,
torsion-free $R$-module is automatically finite free on each connected
component of $\Spec R$.
\end{remark}

\begin{lemma}\label{L:A=field finite coad = perfect}
An $\calR_A^\infty(\Gamma_K)$-module $N$ is finitely generated and coadmissible if and only if it lies in $\bbD^\flat_\perf(\calR_A^\infty(\Gamma_K))$.
\end{lemma}
\begin{proof}
The backward implication follows from Lemma~\ref{L:complex with cohomology capped above}.  We now prove the forward implication.  Assume that $N$ is generated by $n$ elements, giving rise to an exact sequence $0 \to \Ker \to \calR_A^\infty(\Gamma_K)^{\oplus n} \to N \to 0$.
By Lemma~\ref{L:schneider-teitelbaum}(5),  $\Ker$ is coadmissible. By Remark~\ref{R:Bezout+torsion free=>free}, the torsion-free coadmissible module $\Ker$ is finite projective over $\calR_A^\infty(\Gamma_K)$.  This gives a finite resolution of $N$.
\end{proof}

\begin{lemma}\label{L:A=field structure of M^phi=1}
The $A$-module $M^{\varphi=1}$ is finite.
\end{lemma}
Note that although we assumed $A$ is a field, the proof applies whenever $A$ is a strongly noetherian Banach algebra over $\Qp$.
\begin{proof}
Recall that the residue pairing $\{-,-\}_{\Qp}: M \times M^* \to A$
given in Notation~\ref{N:residue pairing} is nondegenerate
and thus induces an injective map $M^{\varphi=1} \to \Hom_A(M^*, A)$.
However, this map factors through $\Hom_A(M^*/(\psi-1),A)$ due to the
identity $\{\varphi(x), y\} = \{x, \psi(y)\}$.
Since $M^*/(\psi-1)$ is finite over $A$ by Proposition~\ref{P:psi cokernel}(1), so then is $M^{\varphi=1}$.
\end{proof}

\begin{prop}
\label{P:Iwasawa cohomology A=field}
The complex $\rmC^\bullet_\psi(M)$ lies in $\bbD_\perf^\flat(\calR_A^\infty(\Gamma_K))$.
\end{prop}
This argument uses some basic properties of slope filtrations for $\varphi$-modules,
for which see  \cite[Theorem~1.7.1]{kedlaya-relative}.
\begin{proof}
First, we have
$\rmC^\bullet_\psi(\Ind_K^{\Qp} M)
= \Hom_{\ZZ[\Gamma_K]}(\ZZ[\Gamma], \rmC^\bullet_\psi(M))$ and hence it suffices to prove the proposition for $K=\Qp$. Since
$M/(\psi-1)$ is a finite $A$-module by Proposition~\ref{P:psi cokernel}(1), it lies in $\bbD_\perf^\flat(\calR_A^\infty(\Gamma))$ by Lemma~\ref{L:finite A-mod is afp A=field}.  Now, we prove $M^{\psi=1} \in \bbD_\perf^\flat(\calR_A^\infty(\Gamma))$.  By d\'evissage and slope filtrations, one may assume that $M$ is of pure slope.  When $M$ is \'etale, this is well-known (e.g. \cite[Th\'eor\`eme I.5.2 and Proposition V.1.18]{colmez-kirillov}, plus the well-known identity $\bbD(\mathbf{V}(M))^{\varphi=1} = \mathbf{V}(M)^{H_K}$).
For general integer slopes, it follows from comparing $\rmC^\bullet_\psi(M)$ with $\rmC^\bullet_\psi(t^nM)$ for $n \in \ZZ$ and applying Proposition~\ref{P:finite-torsion}.

When the slope of $M$ is $c/d$ with $c,d$ integers, we fix $d \in \NN$ and do induction on the residue $c \pmod d$ with known case $c=0$.  Assume the statement is known for some $c-1$ and we check it for $c$.  By Propositions~\ref{P:finite-torsion} and \ref{P:psi cokernel}(2), we may replace $M$ by $t^{-n}M$ for $n\gg0$, so that $(t^iM)/(\psi-1)=0$ for $i=0,\pm1$.

Recall that \cite[Lemma~5.2]{liu} gives a $(\varphi, \Gamma)$-module
$E$ over $\calR$ of pure slope $-1/d$ which is a successive extension
of $d-1$ copies of $\calR$ with a copy of $t\calR$ as the subobject.
Since $M \otimes_\calR E$ is pure of slope equal to that of $M$ minus
$1/d$, the inductive hypothesis gives $(M \otimes_\calR E)^{\psi=1}
\in \bbD_\perf^\flat (\calR_A^\infty(\Gamma))$, and hence $(M
\otimes_\calR E)^{\psi=1}$ is a finitely generated
$\calR_A^\infty(\Gamma)$-module.  This module, by the vanishing of
cokernels of $\psi-1$, is a successive extension of $M^{\psi=1}$ and
$(tM)^{\psi=1}$ with $d-1$ copies of $M^{\psi=1}$ as the quotient
object.  In particular, this implies that $M^{\psi=1}$ is a finitely
generated $\calR_A^\infty(\Gamma)$-module.  Running the same argument
with $t^{-1}M$ in place of $M$ shows that $M^{\psi=1}$, which is
already known to be finitely generated, is an
$\calR_A^\infty(\Gamma)$-submodule of the coadmissible module
$(t^{-1}M \otimes_\calR E)^{\psi=1}$.  By
Lemma~\ref{L:schneider-teitelbaum}(6), $M^{\psi=1}$ is also
coadmissible, and hence lies in
$\bbD_\perf^\flat(\calR_A^\infty(\Gamma))$ by Lemma~\ref{L:A=field
  finite coad = perfect}.  This finishes the induction and the proof.
\end{proof}

\begin{cor}
\label{C:scrC is free}
For $M$ a $(\varphi, \Gamma_K)$-module over $\calR_A(\pi_K)$, the subspace $\scrC \subseteq M^{\psi=0}$ is a finite free $\calR_A^\infty(\Gamma_K)$-module of rank equal to $[K:\Qp]\rank M$.  Moreover, the cohomology groups $H_\psi^i(M)$ for $i=0,1,2$ are of respective generic ranks $0, [K:\Qp]\rank M, 0$ as $\calR_A^\infty(\Gamma_K)$-modules.
\end{cor}
\begin{proof}
Again by induction it suffices to prove the proposition for $K=\Qp$.
By Proposition~\ref{P:Iwasawa cohomology A=field}, $\scrC$ is finitely
generated over $\calR_A^\infty(\Gamma)$.  Sitting inside $M^{\psi=0}$,
a finite projective $\calR_A(\Gamma)$-module, $\scrC$ is torsion-free
and hence finite free over each connected component of $\Spec
\calR_A^\infty(\Gamma)$.  We compute its rank as follows.  By
Proposition~\ref{P:psi cokernel}(1) and Lemma~\ref{L:A=field structure
  of M^phi=1}, the two modules $M/(\psi-1)$ and $M^{\varphi=1}$ have
isolated support.  It follows that there exists $\eta \in
\Max(\calR_A^\infty(\Gamma))$ in each connected component of $\Spec
\calR_A^\infty(\Gamma)$ such that $H^0_{\varphi, \gamma_{\Qp}}(M
\otimes \eta^{-1}) = H^2_{\varphi, \gamma_{\Qp}}(M \otimes \eta^{-1})
= 0$ and $H^1_{\varphi, \gamma_{\Qp}}(M \otimes \eta^{-1}) \cong
M^{\psi=1}/\gothm_\eta \simeq \scrC / \gothm_\eta$.  By
Theorem~\ref{T:Liu}(2), $\scrC$ has rank equal to $[K:\Qp]\rank M$.
The generic rank computation in degree $0$ is obvious, and in degree
$2$ follows from the fact that $M/(\psi-1)$ is a finite $A$-module as
noted above.  To see the correct rank in degree $1$, it suffices to
note that, in the exact sequence
\[
0 \to M^{\varphi=1} \to M^{\psi=1} \to \scrC \to 0,
\]
we have already shown $M^{\varphi=1}$ to be $A$-finite and $\scrC$ to
have generic rank $[K:\Qp]\rank M$.
\end{proof}

\begin{prop}
\label{P:Iwasawa pairing over fields}
Assume $K=\Qp$, and let $M$ be a $(\varphi, \Gamma)$-module over
$\calR_A$ such that $M / (\psi -1) = M^*/(\psi-1) = 0$.
\begin{itemize}
\item[(1)] The natural maps $\varphi-1: M^{\psi=1} \to M^{\psi=0}$ and $\varphi-1: (M^*)^{\psi=1} \to (M^*)^{\psi=0}$ are injective.
\item[(2)] The Iwasawa pairing $\{-,-\}_\Iw$ is perfect, inducing an isomorphism between two finite free $\calR_A^\infty(\Gamma)$-modules:
\[
(M^*)^{\psi =1} \cong \Hom_{\calR_A^\infty(\Gamma)}
\big(M^{\psi=1}, \calR_A^\infty(\Gamma) \big)^\iota.
\]
\item[(3)] We have an isomorphism of $\calR_A^{r_0}(\Gamma)$-modules $\scrC \otimes_{\calR_A^\infty(\Gamma)} \calR_A^{r_0}(\Gamma) \cong (M^{r_0})^{\psi=0}$ with the precise $r_0>0$ for which Corollary~\ref{C:D^psi=0 duality} holds.
\end{itemize}
\end{prop}
\begin{proof}
(1) Vanishing of the cokernels of $\psi-1$ implies that for any $\eta
  \in \Max(\calR_A^\infty(\Gamma))$, $H^2_{\psi,
    \gamma_{\Qp}}(M\otimes \eta^{-1}) = H^2_{\psi,
    \gamma_{\Qp}}(M^*\otimes \eta^{-1}) =0$.  By Tate duality
  (Theorem~\ref{T:Liu}(3)), $H^0_{\psi, \gamma_{\Qp}}(M\otimes
  \eta^{-1}) = H^0_{\psi, \gamma_{\Qp}}(M^*\otimes \eta^{-1}) =0$ for
  any $\eta$.  Note that, by Lemma~\ref{L:A=field structure of
    M^phi=1}, $M^{\varphi=1}$ and $(M^{*})^{\varphi=1}$ are unions of
  finite $A$-modules with a $\Gamma$-action.  So the vanishing of
  $H^0$ for all twists of $M$ and $M^*$ implies that $M^{\varphi=1}=0$
  and $(M^*)^{\varphi=1}=0$, forcing $\varphi-1$ to be injective.

(2) Vanishing of the cokernels of $\psi-1$ implies that the two
vertical injections in Proposition~\ref{P:Iwasawa pairing = Tate
  pairing} are actually isomorphisms.  This means that $M^{\psi=1}$
and $(M^*)^{\psi=1}$ are dual to each other when reduced modulo
$\gothm_\eta$ for each $\eta \in \Max(\calR_A^\infty(\Gamma))$.  Since
both $M^{\psi=1}$ and $(M^*)^{\psi=1}$ are finite free over
$\calR_A^\infty(\Gamma)$ (Corollary~\ref{C:scrC is free}), they are in
perfect duality with one another.

(3) By (2) and Corollary~\ref{C:D^psi=0 duality}, we have the following commutative diagram.
\[
\xymatrix{
\scrC^*\otimes_{\calR_A^\infty(\Gamma)} \calR_A^{r_0}(\Gamma) \ar[r]^-{\cong} \ar[d]& \Hom_{\calR_A^{r_0}(\Gamma)}\big(\scrC \otimes_{\calR_A^\infty(\Gamma)} \calR_A^{r_0}(\Gamma), \calR_A^{r_0}(\Gamma) \big)^\iota\\
(M^{*r_0})^{\psi=0} \ar[r]^-\cong &
\Hom_{\calR_A^{r_0}(\Gamma)}\big( (M^{r_0})^{\psi=0},\calR_A^{r_0}(\Gamma) \big)^\iota \ar[u]
}
\]
This implies that the $\calR_A^{r_0}(\Gamma)$-module in the top line is a direct summand of the $\calR_A^{r_0}(\Gamma)$-module in the bottom line.  Since they are both free of same rank $[K:\Qp]\rank M$ by Corollary~\ref{C:scrC is free} and Theorem~\ref{T:structure of D^psi=0}, they must be isomorphic.
\end{proof}

\begin{remark}
It will be important in the proof of Proposition~\ref{P:N^s uniformly fp} that we have a choice of $r_0$ in Proposition~\ref{P:Iwasawa pairing over fields} that is uniform with respect to points $z \in \Max(A)$.
\end{remark}

\subsection{Statement of main results}
\label{S:statement of main results}

The main result of this paper is the following.

\begin{theorem}
\label{T:finite Iwasawa cohomology}
Let $M$ be a $(\varphi, \Gamma_K)$-module
over $\calR_A(\pi_K)$.  We have $\rmC^\bullet_\psi(M) \in \bbD_\perf^- (\calR_A^\infty(\Gamma_K))$.
\end{theorem}

The proof of this theorem takes up Section~\ref{S:proof}; see the very end of Subsection~\ref{S:devissage}. 
We first list a few useful corollaries of the theorem, including the compatibility with base change, and the comparison between the Iwasawa cohomology and the $(\varphi, \Gamma)$-cohomology of the cyclotomic deformation.

Keeping in mind Proposition \ref{P:psi-gamma complex v.s. psi-complex}, the following theorem is an immediate consequence of the main result.

\begin{theorem}
\label{T:finite cohomology 1}
Let $M$ be a $(\varphi,\Gamma_K)$-module over $\calR_A(\pi_K)$.  Then we have $\rmC_{\varphi,\gamma_K}^\bullet(M) \in \bbD_\perf^-(A)$.  In particular, the cohomology groups $H_{\varphi,\gamma_K}^i(M)$ are finite $A$-modules.
\end{theorem}

The following theorem is essentially \cite[Proposition~2.6]{pottharst1}. But having proved Theorems~\ref{T:finite Iwasawa cohomology} and \ref{T:finite cohomology 1}, we can state it as an unconditional result, which we reproduce here for the convenience of the reader.

\begin{theorem}
\label{T:base change}
Let $M$ be a $(\varphi,\Gamma_K)$-module over $\calR_A(\pi_K)$, and let $A \to B$ be
a morphism of $\Qp$-affinoid algebras.  
\begin{itemize}
\item[(1)] The canonical morphism
\[
\rmC^\bullet_\psi(M) \stackrel{\bbL}\otimes_{\calR_A^\infty(\Gamma_K)} \calR_B^\infty(\Gamma_K) \ \longrightarrow\ \rmC^\bullet_\psi(M \widehat \otimes_A B)
\]
is a quasi-isomorphism.
\item[(2)] The canonical morphism $\rmC^\bullet_{\varphi, \gamma_K}(M) \Lotimes_A B \to 
\rmC^\bullet_{\varphi, \gamma_K}(M \widehat \otimes_A B)$ is a quasi-isomorphism.
\end{itemize}
In particular, if $B$ is flat over $A$, we have 
\[
H^i_\psi(M) \otimes_{\calR_A^\infty(\Gamma_K)}
\calR_B^\infty(\Gamma_K) \cong H^i_\psi(M \widehat \otimes_A B)\ \textrm{ and }\ H_{\varphi, \gamma_K}^i(M) \otimes_A B \cong H_{\varphi, \gamma_K}^i(M \widehat \otimes_A B).
\]
\end{theorem}
\begin{proof}
When $B = A/\gothm_z$ for some $z \in \Max(A)$, $M \widehat \otimes_A B \cong M \otimes_A B$ and $\calR_B^\infty(\Gamma_K) \cong \calR_A^\infty(\Gamma_K)\otimes_AB$ as the latter objects are already complete.  In this case, both (1) and (2) are tautologies because $M$ is flat over $A$ by Corollary~\ref{C:R_A flat over A}.

We now consider the general case.
In case (2), applying the discussion above to every closed point $y \in \Max(B)$, we have a quasi-isomorphism
\begin{align*}
\big(\rmC^\bullet_{\varphi, \gamma_K}(M) \stackrel{\bbL}\otimes_A B\big)\stackrel{\bbL}\otimes_B B/\gothm_y &\xrightarrow\sim \rmC^\bullet_{\varphi, \gamma_K}(M) \stackrel{\bbL}\otimes_A B/\gothm_y
\xrightarrow \sim
\rmC^\bullet_{\varphi, \gamma_K}(M  \otimes_A B/\gothm_y)\\
&\xrightarrow \sim
\rmC^\bullet_{\varphi, \gamma_K}(M \widehat \otimes_A B)\stackrel{\bbL}\otimes_B B/\gothm_y.
\end{align*}
Then (2) follows from the result on detecting quasi-isomorphisms pointwise (Lemma~\ref{L:acyclic pointwise}).

In case (1), we know both sides lie in $\bbD_\perf^- (\calR_B^\infty(\Gamma_K))$ by Theorem~\ref{T:finite Iwasawa cohomology} and hence all cohomology groups are coadmissible.  By Lemma~\ref{L:schneider-teitelbaum}(2), it suffices to prove that, for any $s>0$, the natural morphism
\begin{equation}
\label{E:base-change-Iwasawa}
\rmC^\bullet_\psi(M) \stackrel{\bbL}\otimes_{\calR_A^\infty(\Gamma_K)} \calR_B^{[s,\infty]}(\Gamma_K) \ \longrightarrow\ \rmC^\bullet_\psi(M \widehat \otimes_A B)\stackrel{\bbL}\otimes_{\calR_B^\infty(\Gamma_K)} \calR_B^{[s,\infty]}(\Gamma_K)
\end{equation}
is a quasi-isomorphism. Similar to (2), for each $(y,\eta) \in \Max(\calR_B^\infty(\Gamma_K))$, the natural morphism
\begin{multline*}
\rmC^\bullet_\psi(M)
  \stackrel{\bbL}\otimes_{ \calR_A^\infty(\Gamma_K)}
    \calR_B^\infty(\Gamma_K)
  \stackrel{\bbL}\otimes_{ \calR_B^\infty(\Gamma_K)}
    \calR_B^\infty(\Gamma_K)/\gothm_{(y,\eta)} \\
\longrightarrow
\rmC^\bullet_\psi(M \widehat \otimes_A B)
  \stackrel{\bbL}\otimes_{ \calR_B^\infty(\Gamma_K)}
    \calR_B^\infty(\Gamma_K)/\gothm_{(y,\eta)}
\end{multline*}
is a quasi-isomorphism.  Applying Lemma~\ref{L:acyclic pointwise} to
\eqref{E:base-change-Iwasawa} proves (1).
\end{proof}

\begin{remark}
Using Theorem~\ref{T:base change}, as we restrict $M$ over varying affinoid subdomains of $\Max(A)$, each cohomology group $H^*_{\varphi, \gamma_K}(M)$ yields a coherent sheaf on $\Max(A)$ by Kiehl's theorem.  Similarly, each Iwasawa cohomology $H^*_\psi(M)$ yields a coherent sheaf on $\Max(\calR_A^\infty(\Gamma_K))$.
\end{remark}

With unconditional base change in hand, we may prove more precise versions of the preceding finiteness results.

\begin{theorem}
\label{T:finite cohomology}
Let $M$ be a $(\varphi, \Gamma_K)$-module over $\calR_A(\pi_K)$.  
\begin{itemize}
\item[(1)] We have $\rmC^\bullet_{\varphi, \gamma_K}(M) \in \bbD_\perf^{[0,2]}(A)$.
\item[(2)](Euler characteristic formula) We have $\chi_A\big(\rmC^\bullet_{\varphi, \gamma_K}(M)\big) = -[K:\Qp]\rank M$.
\item[(3)](Tate duality)
The Tate duality pairing \eqref{E:Tate duality complex} induces a quasi-isomorphism
\begin{equation}
\label{E:Tate duality family}
\rmC^\bullet_{\varphi, \gamma_K}(M) \stackrel\sim\longrightarrow \RHom_A(\rmC^\bullet_{\varphi, \gamma_K}(M^*), A)[-2].
\end{equation}
\end{itemize}
\end{theorem}
\begin{proof}
(1)
By Theorem~\ref{T:finite cohomology 1}, we have
$\rmC^\bullet_{\varphi, \gamma_K}(M) \in \bbD_\perf^-(A)$.  Since
$\rmC^\bullet_{\varphi, \gamma_K}(M)$ is concentrated in degrees $0$,
$1$ and $2$, by Lemma~\ref{L:complex with cohomology capped above},
this implies that $\rmC^\bullet_{\varphi, \gamma_K}(M)$ is
quasi-isomorphic to a complex $[\cdots \xrightarrow{d^0} P^1
  \xrightarrow{d^1} P^2]$ of finite projective $A$-modules, and hence
quasi-isomorphic to the complex $[\Coker(d^{-1}) \xrightarrow{d^0} P^1
  \xrightarrow{d^1} P^2]$.  By Corollary~\ref{C:R_A flat over A}, the
complex $\rmC^\bullet_{\varphi, \gamma_K}(M)$ consists of flat
$A$-modules.  This forces $\Coker(d^{-1})$ to be also flat (and hence
projective) over $A$, by Lemma~\ref{L:flat bottom cokernel}.  Hence
$\rmC^\bullet_{\varphi,\gamma_K}(M) \in \bbD_\perf^{[0,2]}(A)$, as
claimed.

(2)
To check the Euler characteristic formula, it suffices to check it at one closed point on each connected component.  This follows from Theorem~\ref{T:Liu}(2).

(3) It is clear that \eqref{E:Tate duality complex} induces the natural morphism \eqref{E:Tate duality family}. By Theorem~\ref{T:Liu}(3), \eqref{E:Tate duality family} is a quasi-isomorphism when tensored with $A/\gothm_x$ for any $x \in \Max(A)$.   Our statement  then follows from Lemma~\ref{L:acyclic pointwise}.
\end{proof}

We may use Theorem~\ref{T:finite cohomology} to strengthen Theorem~\ref{T:finite Iwasawa cohomology}.

\begin{theorem}
\label{T:strong Iwasawa cohomology}
Let $M$ be a $(\varphi, \Gamma_K)$-module
over $\calR_A(\pi_K)$.  

\begin{itemize}
\item[(1)]  We have $\rmC^\bullet_\psi(M) \in \bbD_\perf^{[0,2]} (\calR_A^\infty(\Gamma_K))$.
\item[(2)] The $\calR_A^\infty(\Gamma_K)$-module $M^{\psi=1}$ is finitely presented.
\item[(3)](Euler characteristic formula) We have
  $\chi_{\calR_A^\infty(\Gamma_K)} \big( \rmC^\bullet_\psi(M)\big) =
  -[K:\Qp]\rank M$.  Moreover, $H_\psi^0(M) = 0$, and as
  $\calR_A^\infty(\Gamma_K)$-modules $H_\psi^2(M)$ is torsion and
  $H_\psi^1(M)$ has generic rank $[K:\Qp]\rank M$.  (Note this does
  not force $\rmC^\bullet_\psi(M) \in
  \bbD_\perf^{[1,2]}(\calR_A^\infty(\Gamma_K))$ because, for example,
  there may be torsion in $H^1_\psi(M)$.)
\end{itemize}

\end{theorem}
\begin{proof}
(1)
Since $\rmC^\bullet_\psi(M)$ is concentrated in degrees $1$ and $2$, Lemma~\ref{L:complex with cohomology capped above} implies that it is quasi-isomorphic to a complex $[\cdots \xrightarrow{d^{-1}} P^0 \xrightarrow{d^0} P^1 \xrightarrow{d^1} P^2 ]$ of finite projective $\calR_A^\infty(\Gamma_K)$-modules, and hence to $[Q \to P^1 \to P^2]$ with $Q = \Coker d^{-1}$.
To prove the first statement, we need to show that $Q$ is a projective $\calR_A^\infty(\Gamma_K)$-module (note that $\Coker d^0$ may not be projective in general and hence $\rmC^\bullet_\psi(M) \notin \bbD_\perf^{[1,2]} (\calR_A^\infty(\Gamma_K))$). By Lemma~\ref{L:schneider-teitelbaum}(5), $Q$ is a coadmissible $\calR_A^\infty(\Gamma_K)$-module. By 
Lemma~\ref{L:fp+bundle=>finite-proj}, it suffices to prove that $Q$ gives a vector bundle over $\calR_A^\infty(\Gamma_K)$, or equivalently $Q \otimes_{\calR_A^\infty(\Gamma_K)} \calR_A^{[s, \infty]}(\Gamma_K)$ is a finite flat $\calR_A^{[s, \infty]}(\Gamma_K)$-module for all $s>0$.  For this, we may assume that $A$ is geometrically connected.

For any $\eta \in \Max(\calR^\infty(\Gamma_K))$, we have
$\rmC^\bullet_\psi(M) \stackrel{\bbL}\otimes_{\calR^\infty(\Gamma_K)}
\calR^\infty(\Gamma_K)/\gothm_\eta \cong \rmC^\bullet_{\psi,
  \gamma_K}(M \otimes \eta^{-1})$ by Proposition~\ref{P:psi-gamma
  complex v.s. psi-complex}.  Since $P^\bullet$ consists of finite
projective $\calR_A^\infty(\Gamma_K)$-modules, it follows that the
complex $P^\bullet \otimes_{\calR_A^\infty(\Gamma_K)}
\calR_A^\infty(\Gamma_K)/\gothm_\eta$ represents $\rmC_\psi^\bullet(M)
\stackrel{\bbL}\otimes_{\calR_A^\infty(\Gamma_K)}
\calR_A^\infty(\Gamma_K)/\gothm_\eta$.  But $P^\bullet
\otimes_{\calR_A^\infty(\Gamma_K)}
\calR_A^\infty(\Gamma_K)/\gothm_\eta$ consists of finite projective
$\calR_A^\infty(\Gamma_K)/\gothm_\eta$-modules, so there exists a
morphism of complexes $P^\bullet \otimes_{\calR_A^\infty(\Gamma_K)}
\calR_A^\infty(\Gamma_K)/\gothm_\eta \to
\rmC^\bullet_{\psi,\gamma_K}(M \otimes \eta^{-1})$ inducing an
isomorphism on cohomology.  Lemma~\ref{L:flat bottom cokernel} and
Theorem~\ref{T:finite cohomology} then show $Q
\otimes_{\calR_A^\infty(\Gamma_K)}
\calR_A^\infty(\Gamma_K)/\gothm_\eta$ to be finite projective over
$\calR_A^\infty(\Gamma_K)/\gothm_\eta \cong A \otimes_{\Qp}
\kappa_\eta$.  Moreover, by Theorem~\ref{T:finite cohomology}(2) and
the flatness of $P^1$ and $P^2$ over $\calR_A^\infty(\Gamma_K)$, we
know that $\eta \mapsto \rank_{A \otimes_{\Qp} \kappa_\eta} (Q /
\gothm_\eta)$ defines a locally constant function on
$\Max(\calR^\infty(\Gamma_K))$ as $\eta$ varies.  By
Lemma~\ref{L:constant dimension function implies flat}(2), $Q
\otimes_{\calR_A^\infty(\Gamma_K)} \calR_A^{[s, \infty]}(\Gamma_K)$ is
a finite flat $\calR_A^{[s, \infty]}(\Gamma_K)$-module for any $s>0$.
(1) follows.

(2) As $M^{\psi=1}$ is a cohomology group of the complex
$\rmC^\bullet_\psi(M) \in
\bbD_\perf^{[0,2]}(\calR_A^\infty(\Gamma_K))$, it is coadmissible by
Lemma~\ref{L:schneider-teitelbaum}.  By Proposition~\ref{P:psi
  cokernel}, $M/(\psi-1)$ is finite over $A$.  By Lemma~\ref{L:finite
  A-mod as R_A-mod}, there exists $r >0$ such that $M/(\psi-1)
\otimes_{\calR_A^\infty(\Gamma_K)} \calR_A^{r}(\Gamma_K) = 0$.  By
part (1) above, this means that $M^{\psi=1}
\otimes_{\calR_A^\infty(\Gamma_K)} \calR_A^{r}(\Gamma_K)$ is the only
nonzero cohomology group of a complex in
$\bbD_\perf^{[0,2]}(\calR_A^{r}(\Gamma_K))$.  In particular, it is
finitely presented as an $\calR_A^{r}(\Gamma_K)$-module.  By
Lemma~\ref{L:schneider-teitelbaum}(3) the coherent sheaf on
$\calR_A^r(\Gamma_K)$ determined by $M^{\psi=1}
\otimes_{\calR_A^\infty(\Gamma_K)} \calR_A^{r}(\Gamma_K)$ is then
uniformly finitely presented, and combining this with the fact that
$\calR_A^{[r,\infty]}(\Gamma_K)$ is noetherian, the coherent sheaf on
$\calR_A^\infty(\Gamma_K)$ determined by $M^{\psi=1}$ is uniformly
finitely presented over $\calR_A^\infty(\Gamma_K)$.  By
Proposition~\ref{P:finite-generation}(2), we deduce that $M^{\psi=1}$
is a finitely presented $\calR_A^\infty(\Gamma_K)$-module.

(3) Since the Euler characteristic can be tested at a closed point on each connected component, the first claim follows from Theorem~\ref{T:Liu}(2).  For the second claim, one may base change to any sufficiently general closed point to assume that $A$ is a finite field extension of $\Qp$, in which case the claim is treated by Corollary \ref{C:scrC is free}.
\end{proof}

We now explain the relation between the Iwasawa cohomology of $M$ and
the $(\varphi, \Gamma)$-cohomology of the cyclotomic deformation of
$M$, following \cite{pottharst2}.

\begin{defn}
\label{D:Iwasawa deformation space}
For $n \geq 1$, let $\Lambda_n = \calR^{[1/p^n,\infty]}(\Gamma_K)$ and
let $X_n = \Max(\Lambda_n)$.  We view $X = \cup_{n\geq1} X_n$ as a
quasi-Stein rigid analytic space, which is a disjoint union of open
discs.  For $n \geq 1$, consider the rank one $(\varphi,
\Gamma_K)$-module $\mathbf{Dfm}_n = \Lambda_n \widehat \otimes_{\Qp}
\calR(\pi_K)\bbe = \calR_{\Lambda_n}(\pi_K)\bbe$, where $\varphi(1
\otimes \bbe) = 1\otimes \bbe$ and $\gamma(1 \otimes \bbe) =
\gamma^{-1}\otimes \bbe$ for $\gamma \in \Gamma_K$.  We put $\Dfm =
\varprojlim_n \Dfm_n$; this is a $(\varphi, \Gamma_K)$-module over the
relative Robba ring over $X$.  For a closed point $\eta \in
\Max(\Lambda_n)$, we have $\Dfm/\gothm_\eta\cong \Dfm_n/\gothm_\eta
\cong \calR(\pi_K)(\eta^{-1})$, where $\calR(\pi_K)(\eta^{-1})$ is the
$(\varphi, \Gamma_K)$-module associated to $\eta$.

For $M$ a $(\varphi, \Gamma_K)$-module over $\calR_A(\pi_K)$, we
define the \emph{cyclotomic deformation of $M$} to be $\Dfm(M) = M
\widehat \otimes_{\calR(\pi_K)} \Dfm = \varprojlim_n M
\widehat\otimes_{\calR(\pi_K)} \Dfm_n$; it is a $(\varphi,
\Gamma_K)$-module over the relative Robba ring over $\Max(A) \times
X$.
\end{defn}

\begin{theorem}
\label{T:comparison iwasawa v.s. deformation} 
Let  $M$ be a $(\varphi, \Gamma_K)$-module over $\calR_A(\pi_K)$.  Then we have a natural quasi-isomorphism of complexes of $\calR_A^\infty(\Gamma_K)$-modules
\begin{equation}
\label{E:comparison iwasawa v.s. deformation}
\rmC^\bullet_{\psi, \gamma_K}(\Dfm(M)) \xrightarrow\sim \rmC^\bullet_\psi(M),
\end{equation}
where $\Gamma_K$ acts on the left through the $\calR_{\Lambda_n}(\pi_K)$-structures on the respective $\Dfm_n$, and on the right via the $(\varphi,\Gamma_K)$-actions on $M$.
Consequently, we have a natural Iwasawa duality quasi-isomorphism
\begin{equation}\label{E:iwasawa duality qisom}
\rmC^\bullet_\psi(M^*)
\xrightarrow\sim \RHom_{\calR_A^\infty(\Gamma_K)} \big(\rmC^\bullet_\psi(M), \calR_A^\infty(\Gamma_K))^\iota[-2].
\end{equation}
\end{theorem}

\begin{proof}
The isomorphisms $(\Dfm(M)^{\Delta_K})^{\gamma_K=1} =
\Dfm(M)^{\Gamma_K} = 0$ and
\[
\Dfm(M)^{\Delta_K}/(\gamma_K-1) = (M \widehat\otimes_{\calR(\pi_K)}
\Dfm)^{\Delta_K}/(\gamma_K \otimes \gamma_K-1) \cong M
\]
can be packaged into a quasi-isomorphism $[\Dfm(M)^{\Delta_K}
  \xrightarrow{\gamma_K-1} \Dfm(M)^{\Delta_K}] \to M[-1]$.  Taking the
mapping fiber of $1-\psi$ on both sides and adjusting signs and
degrees, one obtains the quasi-isomorphism
$\rmC_{\psi,\gamma_K}^\bullet(\Dfm(M)) \to \rmC_\psi^\bullet(M)$.

The second statement follows from combining the first statement with Theorem~\ref{T:finite cohomology}(3). 
\end{proof}

\begin{remark}
It would be interesting to know if one can define the Iwasawa duality morphism \eqref{E:iwasawa duality qisom} directly on the level of complexes without comparing it to the cyclotomic deformation.
\end{remark}

\begin{remark}
One can give a direct proof of the following corollary without
referring to our main theorem, but it would require some setup on the
construction of the functor $\bbD_\rig$.  In the case $A=\Qp$, the
comparison of the Iwasawa cohomology of $V$ to that of its cyclotomic
deformation (over $\Zp[\![\Gamma_K]\!]$, not $\calR^\infty(\Gamma_K)$)
is essentially a variant of Shapiro's lemma, and to our knowledge
first occurs in Iwasawa theory in works of Greenberg \cite{greenberg}.
It also occurs in works of Colmez \cite{colmez-Iw}, together with
consideration of the cyclotomic deformation over the subring of
$\calR^\infty(\Gamma_K)$ of tempered distributions.  The claim over
the full ring $\calR^\infty(\Gamma_K)$ appears in \cite{pottharst2}.
\end{remark}

\begin{cor}
Let $V$ be a finite projective $A$-module equipped with a continuous $A$-linear action of $G_K$.  If we use $H^i_\Iw(G_K,V)$ to denote the inverse limit $[\varprojlim_{n \to \infty} H^i(G_{K(\mu_{p^n})}, T)] \otimes \QQ$ (where $T \subseteq V$ is the unit ball for a Galois-invariant Banach module norm) of the cohomology groups under the corestriction maps for any $i$, we have a functorial isomorphism $H^i_\Iw(G_K,V) \widehat \otimes_{\ZZ_p[\![\Gamma_K]\!]} \calR^\infty(\Gamma_K) \cong
H^i_\psi(\bbD_\rig(V))$ of $\calR_A^\infty(\Gamma_K)$-modules for any $i$ compatible with base change.
\end{cor}
\begin{proof}
By Shapiro's lemma, $H^i_\Iw(G_K,V)  \cong H^i(G_K, V \widehat \otimes_{\Zp} \tilde \Lambda)$, where $\tilde \Lambda = \Zp\llbracket \Gamma\rrbracket$ with the Galois action factoring through $\Gamma$ and via $\iota$.  The corollary follows from a sequence of isomorphisms
\begin{multline*}
H^i_\Iw(G_K,V) \otimes_{\Lambda(\Gamma_K)} \calR^\infty(\Gamma_K) \xrightarrow{\textrm{ \cite[Theorem~1.6]{pottharst1}}} H^i(G_K, V \widehat \otimes_{\Qp} \varprojlim_n \Lambda_n )\\
\xrightarrow{\textrm{Theorem~\ref{P:comparison with Galois cohomology}}} H^i_{\varphi, \gamma_K}(\bbD_\rig(V) \widehat \otimes_{\calR(\pi_K)} \Dfm) \xrightarrow{\textrm{Theorem~\ref{T:comparison iwasawa v.s. deformation}}} H^i_\psi(\bbD_\rig(V)).
\end{multline*}
\end{proof}

\section{Proof of the main theorem}
\label{S:proof}

We now complete the arguments of the previous section by proving
finiteness of Iwasawa cohomology (Theorem~\ref{T:finite Iwasawa
  cohomology}).  The reader will note that, in order to avoid any
vicious circles, we refrain from invoking any results from
Subsection~\ref{S:statement of main results}, although the rest of
Section~\ref{S:finiteness of cohomology} is allowed. The central
argument is to use the duality pairing to play the Iwasawa $H^1$
groups of a given $(\varphi, \Gamma)$-module and its Cartier dual
against each other, forcing them both to be finite. This argument only
works if both Iwasawa $H^2$ groups vanish; the remainder of the
argument is a d\'evissage to reduce to this case, in the spirit of
some arguments from \cite{liu} but not quite identical to any of them.

\subsection{Preliminary reductions}

We make some preliminary reductions for the proof of Theorem~\ref{T:finite Iwasawa cohomology}.

\begin{lemma}
\label{L:theorem for K=>for Qp}
If Theorem~\ref{T:finite Iwasawa cohomology} is true for $K = \Qp$ (and any $(\varphi, \Gamma)$-module over $\calR_A$), then it is true for any finite extension $K$ of $\Qp$.
\end{lemma}
\begin{proof}
If $M$ is a $(\varphi, \Gamma_K)$-module over $\calR_A(\pi_K)$, then we have
\[
\rmC^\bullet_\psi(\Ind_K^{\Qp}M) = \rmC^\bullet_\psi\big(\Hom_{\ZZ[\Gamma_K]}(\ZZ[\Gamma], M)\big) \cong \Hom_{\ZZ[\Gamma_K]}(\ZZ[\Gamma], \rmC^\bullet_\psi(M)).
\]
Hence Theorem~\ref{T:finite Iwasawa cohomology} for $\Ind_K^{\Qp}M$ implies that for $M$.
\end{proof}

\begin{hypothesis}
\label{H:K=Qp}
For the rest of this section, we assume that $K = \Qp$.  We hence write $\calR_A$ for $\calR_A(\pi_K)$, $\calR_A(\Gamma)$ for $\calR_A(\Gamma_K)$, and $\calR_A^\infty(\Gamma)$ for $\calR_A^\infty(\Gamma_K)$.
\end{hypothesis}

\begin{lemma}
\label{L:R_A/I(Gamma) is perfect}
If $I$ is an ideal of $A$, then a complex of $\calR^\infty_{A/I}(\Gamma)$-modules in $\bbD_\perf^-(\calR^\infty_{A/I}(\Gamma))$, when viewed as a complex of $\calR^\infty_A(\Gamma)$-modules, lies in $\bbD_\perf^-(\calR^\infty_A(\Gamma))$.
\end{lemma}
\begin{proof}
Any resolution of $A/I$: $\cdots \to A^{\oplus n_1} \to A^{\oplus n_0} \to A/I \to 0$ by finite free $A$-modules  is strict and hence induces an exact sequence
\[
\cdots \to \calR^\infty_A(\Gamma)^{\oplus n_1} \to \calR^\infty_A(\Gamma)^{\oplus n_0} \to \calR^\infty_{A/I}(\Gamma) \to 0.
\]
The lemma follows from this.
\end{proof}

\begin{lemma}
\label{L:finite base change}
Let $M$ be a $(\varphi, \Gamma)$-module over $\calR_A$ and let $B$ be a \emph{finite} $A$-algebra.  Then we have a natural quasi-isomorphism 
\begin{equation}
\label{E:finite base change}
\rmC^\bullet_\psi(M) \stackrel{\bbL}\otimes_{\calR_A^\infty(\Gamma)} \calR_B^\infty(\Gamma) \ \longrightarrow\ \rmC^\bullet_\psi( M  \otimes_A B)
\end{equation}
In particular, if Theorem~\ref{T:finite Iwasawa cohomology} holds for $M$, then it holds for $M  \otimes_A B \cong M\widehat \otimes_A B$.
\end{lemma}
\begin{proof}
Straightforward.
\end{proof}

\begin{lemma}
\label{L:reduced}
Let $A_\red$ denote the reduced quotient of  $A$.
Let $M$ be a $(\varphi,\Gamma)$-module over $\calR_A$.  If
Theorem~\ref{T:finite Iwasawa cohomology} holds for $M \otimes_A A_\red$ then it
holds for $M$.
\end{lemma}
\begin{proof}
By d\'evissage, we may write $A$ as a finite successive extension of $A$-modules, each of which is isomorphic to $A /I_i$ for some radical ideal $I_i$ of $A$.  Then $M_i = M \otimes_A A/I_i$ is a $(\varphi, \Gamma)$-module over $\calR_{A/I_i}$.  By
Lemma~\ref{L:finite base change} (noting that $A/I_i$ is finite over $A_\red$), Theorem~\ref{T:finite Iwasawa cohomology} for $M \otimes_A A_\red$ implies Theorem~\ref{T:finite Iwasawa cohomology} for each $M_i$.  Therefore, each $\rmC^\bullet_\psi(M_i)$ belongs to $\bbD_\perf^-(\calR^\infty_{A/I_i}(\Gamma))$ and thus belongs to $\bbD_\perf^-(\calR^\infty_A(\Gamma))$ by Lemma~\ref{L:R_A/I(Gamma) is perfect}.  Now, $\rmC^\bullet_\psi(M)$, as a successive extension of complexes $\rmC^\bullet_\psi(M_i)$, must also lie in $\bbD_\perf^- (\calR^\infty_A(\Gamma))$.
\end{proof}

\begin{lemma}
\label{L:twist by t}
For $M$ a $(\varphi, \Gamma)$-module over $\calR_A$ and any integer $n$,
Theorem~\ref{T:finite Iwasawa cohomology} holds for $M$ if and only if it holds for $t^nM$.
\end{lemma}
\begin{proof}
This follows from Proposition~\ref{P:finite-torsion} by induction on $n$.
\end{proof}

\subsection{Main theorem in a special case}

Using the Iwasawa duality pairing we established in Definition~\ref{D:Iwasawa pairing} and its compatibility with the Tate local duality as discussed in Proposition~\ref{P:Iwasawa pairing = Tate pairing}, one may deduce the finite projectivity of $M^{\psi=1}$ over $\calR_A^\infty(\Gamma)$ in the special case when $M/(\psi-1) =  M^*/(\psi-1) = 0$.  This is one of the key steps in proving Theorem~\ref{T:finite Iwasawa cohomology}.  We will reduce the proof of the theorem to this special case in the next subsection.

We remind the reader that Hypothesis~\ref{H:K=Qp} is still in force.

\begin{hypothesis}\label{H:reduced vanishing H2Iw}
Throughout this subsection, we assume that $A$ is a \emph{reduced} $\Qp$-affinoid algebra. 
We also take $M$ to be
a $(\varphi, \Gamma)$-module of constant rank $d$ over $\calR_A$ such that $M / (\psi-1) = M^*/(\psi-1)=0$.
\end{hypothesis}

\begin{lemma}
\label{L:lifting elements in M^psi=1}
For any $(z, \eta) \in \Max(\calR_A^\infty(\Gamma))$, we have $M^{\psi=1} \otimes_{\calR_A^\infty(\Gamma)} \kappa_{(z, \eta)} \cong H^1_{\varphi, \gamma_{\Qp}}(M_z \otimes \eta^{-1}))$, and the same holds for $M^*$ in place of $M$.
\end{lemma}
\begin{proof}
Since $M / (\psi-1)=0$, the spectral sequence $E_2^{j,-i} =\Tor^{\calR_A^\infty(\Gamma)}_i \big(H^j_{\psi}(M), \kappa_{(z, \eta)} \big) \Rightarrow H^{j-i}_{\psi, \gamma_{\Qp}}(M_z \otimes \eta^{-1})$  stabilizes at $E_2$, giving the isomorphism in the lemma.  The same argument holds with $M$ replaced by $M^*$.
\end{proof}

\begin{lemma}
\label{L:pointwise cohomology dimension}
For any $(z, \eta) \in \Max(\calR_A^\infty(\Gamma))$,
we have $H^i_{\varphi, \gamma_{\Qp}}(M_z \otimes \eta^{-1}) = H^i_{\varphi, \gamma_{\Qp}}(M^*_z \otimes \eta) =0$ for $i=0,2$.  Hence $\dim_{\kappa_{(z,\eta)}}H^1_{\varphi,\gamma_{\Qp}}(M_z \otimes \eta^{-1}) = \dim_{\kappa_{(z,\eta)}}H^1_{\varphi,\gamma_{\Qp}}(M^*_z \otimes \eta) = d$.
\end{lemma}
\begin{proof}
For $i=2$, this follows from the fact that $H^2_{\psi, \gamma_{\Qp}}(M_z \otimes \eta^{-1}) = M / (\psi-1, \gothm_{(z,\eta)})$ and $H^2_{\psi, \gamma_{\Qp}}(M^*_z \otimes \eta) = M^* / (\psi-1, \gothm_{(z,\eta^{-1})})$.  The result for $H^0_{\varphi, \gamma_{\Qp}}$ follows from 
Tate duality over $\kappa_{(z,\eta)}$ (Theorem~\ref{T:Liu}(3)).  The dimension of $H^1_{\varphi, \gamma_{\Qp}}$ is computed using the Euler characteristic formula (Theorem~\ref{T:Liu}(2)).
\end{proof}

\begin{lemma}
\label{C:pointwise Iwasawa cohomology}
For any $z \in \Max(A)$, we have $M_z/(\psi-1) = M^*_z/(\psi-1) = 0$.  Hence $M_z^{\psi=1}$ and $(M^*_z)^{\psi=1, \iota}$ are free $\calR_{\kappa_z}^\infty(\Gamma)$-modules of rank $d$ and are perfect dual to each other.
\end{lemma}
\begin{proof}
The first statement is obvious, while the second one follows from Proposition~\ref{P:Iwasawa pairing over fields}(2).
\end{proof}

\begin{lemma}
\label{L:injectivity to H^1}
We have injective homomorphisms
\[
M^{\psi=1} \hookrightarrow \prod_{(z,\eta)\in \Max(\calR_A^\infty(\Gamma))}H^1_{\psi, \gamma_{\Qp}}(M_z \otimes \eta^{-1})\ \textrm{ and }\ (M^*)^{\psi=1} \hookrightarrow \prod_{(z,\eta)\in \Max(\calR_A^\infty(\Gamma))}H^1_{\psi, \gamma_{\Qp}}(M_z^* \otimes \eta).
\]
\end{lemma}
\begin{proof}
We indicate the proof of the first injection, and the second can be proved in the same way.
We first observe that $\calR_A \hookrightarrow \prod_{z \in \Max(A)} \calR_{\kappa_z}$.  It follows that $M \hookrightarrow \prod_{z \in \Max(A)} M_z$.  As $\calR^\infty_A(\Gamma)$-submodules, we have $M^{\psi=1}
\hookrightarrow \prod_{z \in \Max(A)} M_z^{\psi=1}$.
By Corollary~\ref{C:pointwise Iwasawa cohomology}, each $M_z^{\psi=1}$ is a  finite free module over $\calR_{\kappa_z}^\infty(\Gamma)$. Since $\calR^\infty_{\kappa_z}(\Gamma) \hookrightarrow \prod_{\eta \in \Max(\calR^\infty_{\kappa_z}(\Gamma))} \calR^\infty_{\kappa_z}(\Gamma)/\gothm_\eta$, we can continue the above inclusion for $M^{\psi=1}$ to get
\[
M^{\psi=1} \hookrightarrow \prod_{(z, \eta) \in \Max(\calR_A^\infty(\Gamma))} \big(M_z^{\psi=1}/\bar \gothm_{(z,\eta)}\big) \cong \prod_{(z, \eta) \in \Max(\calR_A^\infty(\Gamma))}H^1_{\psi, \gamma_{\Qp}}(M_z \otimes \eta^{-1}).
\]
\end{proof}

\begin{prop} 
\label{P:generators nearby}
Fix $(z, \eta) \in \Max(\calR_A^\infty(\Gamma))$.
Suppose that $S$ is a finite subset of $M^{\psi=1}$ generating $M^{\psi=1}/\gothm_{(z,\eta)} \cong H^1_{\varphi, \gamma_{\Qp}}(M_z \otimes \eta^{-1})$.
Then there exists $f \in \calR_A^\infty(\Gamma)$ which is nonzero modulo $\gothm_{(z, \eta)}$,
such that every element of $f M^{\psi=1}$ is equal to an $\calR_A^\infty(\Gamma)$-linear combination of $S$.
\end{prop}
\begin{proof}
We may choose $S = \{\alpha_1, \dots,\alpha_d\}$ so that it maps to a
basis $S_{(z, \eta)}$ of $H^1_{\varphi, \gamma_{\Qp}}(M_z \otimes
\eta^{-1})$.  By Theorem~\ref{T:Liu}(3), $H^1_{\varphi,
  \gamma_{\Qp}}(M^*_z \otimes \eta)$ may be identified with the dual
space of $H^1_{\varphi, \gamma_{\Qp}}(M_z \otimes \eta^{-1})$; let
$S_{(z, \eta^{-1})}^*$ denote the dual basis to $S_{(z, \eta)}$.  We
may thus lift $S_{(z, \eta^{-1})}^*$ to a finite subset $S^* =
\{\beta_1, \dots, \beta_d\}$ of $(M^*)^{\psi=1}$ by
Lemma~\ref{L:lifting elements in M^psi=1}.

The $n \times n$ matrix $F$ over $\calR_A^\infty(\Gamma)$ defined by
$F_{ij} = \{ \alpha_i, \beta_j \}_\Iw$ reduces to the identity modulo
$\gothm_{(z, \eta)}$, so $\det(F)$ is not zero modulo $\gothm_{(z,
  \eta)}$.  Then the adjugate matrix $G$ of $F$ satisfies $FG = GF =
\det(F) \mathrm{Id}_n$, and $f = \det(F)^2$ has the desired property:
given $\delta \in M^{\psi=1}$, the element
\[
\det(F) \delta - \sum_{i,j} G_{ji} \{ \delta, \beta_j \}_\Iw \alpha_i
\]
pairs to zero with each $\beta_j$.  Thanks to the duality isomorphism
from Theorem~\ref{T:Liu}(3) it vanishes in $H^1_{\varphi,
  \gamma_{\Qp}}(M_{z'} \otimes \eta'^{-1})$ for any $(z', \eta') \in
\Max(\calR_A^\infty(\Gamma))$ for which $S^*$ reduces to a basis,
namely on the nonvanishing locus of $\det(F)$.  Multiplying the above
displayed quantity through by $\det(F)$ gives an element vanishing in
$H^1_{\varphi, \gamma_{\Qp}}(M_{z'} \otimes \eta'^{-1})$ for all $(z',
\eta')$, which must be zero by Lemma~\ref{L:injectivity to H^1}. The
proposition follows.
\end{proof}

\begin{notation}
We set $N^s = M^{\psi=1} \otimes_{\calR_A^\infty(\Gamma)} \calR_A^{[s,\infty]}(\Gamma)$ and $N^{s*} = (M^*)^{\psi=1} \otimes_{\calR_A^\infty(\Gamma)} \calR_A^{[s,\infty]}(\Gamma)$.
Note that we simply took the algebraic tensor product, which means that neither space carries any topology \textit{a priori}.  This partly reflects the difficulties we are trying to work around.
\end{notation}

\begin{cor}
\label{C:N_s finite flat}
For any $s >0$, the $\calR_A^{[s,\infty]}(\Gamma)$-modules $N^s$ and $N^{s*}$ are finite and flat.
\end{cor}
\begin{proof}
All maximal ideals of $\calR_A^{[s,\infty]}(\Gamma)$ are of the form
$\gothm_{(z,\eta)}$, so the ideal generated by all the $f$ as in
Proposition~\ref{P:generators nearby} is the unit ideal.  Writing $1$
in terms of these generators shows that only finitely many such $f$
are in fact necessary.  Hence $N^s$ is generated by the corresponding
lifts of local bases.  The flatness follows from the dimension
calculation in Lemma~\ref{L:pointwise cohomology dimension}.
\end{proof}

\begin{remark}
Since our goal is to prove the finiteness theorem of Iwasawa cohomology, we will have to make some additional efforts to tackle the finite generation properties.  If one is only interested in the finiteness of $(\varphi,\Gamma)$-cohomology, Corollary~\ref{C:N_s finite flat} together with the d\'evissage argument in the next subsection is enough to deduce Theorem~\ref{T:finite cohomology}. 
\end{remark}

\begin{prop}
\label{P:N^s uniformly fp}
The modules $N^s$ form a uniformly finitely presented vector bundle over $\calR_A^\infty(\Gamma)$.  Moreover, if we use $N$ to denote  the module of global sections of the vector bundle $(N^s)$, then $N$ is a finite projective $\calR_A^\infty(\Gamma)$-module.
\end{prop}
\begin{proof}
Let $r_0>0$ be the number provided by Corollary~\ref{C:D^psi=0 duality}. 
For $0<s\leq r_0$, consider the natural map
\[
\varphi-1: N^s \otimes_{\calR_A^{[s, \infty]}(\Gamma)}\calR_A^{[s,r_0]}(\Gamma) \cong M^{\psi=1} \otimes_{\calR_A^\infty(\Gamma)}
\calR_A^{[s,r_0]}(\Gamma) \to (M^{[s,r_0]})^{\psi=0}
\]
of two finite flat $\calR_A^{[s,r_0]}(\Gamma)$-modules (the latter by Corollary~\ref{S:statement of main results}).  By Proposition~\ref{P:Iwasawa pairing over fields}(3), it is an isomorphism modulo any $\gothm_z$ for $z \in \Max(A)$.  Therefore, it has to be an isomorphism itself.  Since $(M^{[s,r_0]})^{\psi=0}$ as $s\to0^+$ form a uniformly  finitely presented vector bundle over $\calR_A^{r_0}(\Gamma)$, we see that $N^s$ also form a uniformly  finitely presented vector bundle over $\calR_A^\infty(\Gamma)$.
The last statement now follows from Proposition~\ref{P:finite-generation}(3).
\end{proof}

\begin{lemma}
\label{L:M^psi=1 injects to N}
The natural map $M^{\psi=1} \to N$ is injective.
\end{lemma}
\begin{proof}
Proposition~\ref{P:N^s uniformly fp} implies that $N$ is a finite projective $\calR_A^\infty(\Gamma)$-module.  Since $\calR_A^\infty(\Gamma) \hookrightarrow \prod_{(z,\eta)\in\Max(\calR_A^\infty(\Gamma))} \kappa_{(z,\eta)}$ is injective,  the module $N$ embeds into the product 
\[
\prod_{(z,\eta)\in\Max(\calR_A^\infty(\Gamma))} N / \gothm_{(z,\eta)} \cong \prod_{(z,\eta)\in \Max(\calR_A^\infty(\Gamma))} M^{\psi=1} / \gothm_{(z,\eta)} \cong \prod_{(z,\eta)\in \Max(\calR_A^\infty(\Gamma))} H^1_{\psi, \gamma_{\Qp}}(M_z \otimes \eta^{-1})
\]
By Lemma~\ref{L:injectivity to H^1}, both $M^{\psi=1}$ and $N$ embed into the same module.  Hence the natural map $M^{\psi=1} \to N$ is injective.
\end{proof}

\begin{theorem}
\label{T:structure of M^psi=1}
Under Hypothesis~\ref{H:reduced vanishing H2Iw}, $M^{\psi=1}$ is a
finite projective $\calR_A^\infty(\Gamma)$-module.
\end{theorem}
\begin{proof}
By Proposition~\ref{P:N^s uniformly fp}, it suffices to prove that the natural map $M^{\psi=1} \to N$ in Lemma~\ref{L:M^psi=1 injects to N} is in fact an isomorphism.  This follows from combining Proposition~\ref{P:N^s uniformly fp} and Lemma~\ref{L:M^psi=1 injects to N} with
Lemma~\ref{L:variant finite generation}, where the continuity assumptions required in the second lemma are verified by Proposition~\ref{P:calR_A(Gamma_K) acts on M}(3).
\end{proof}

\subsection{D\'evissage}
\label{S:devissage}

To complete the proof of the finiteness of Iwasawa cohomology of $(\varphi, \Gamma)$-modules, we 
must perform some d\'evissage in the style of \cite{liu} in order to arrive at a case in which we can control $M/(\psi-1)$ and $M^*/(\psi-1)$.

\begin{hypothesis}
\label{H:devissage}
We continue to assume that $K=\Qp$, except in the proof of
Theorem~\ref{T:finite Iwasawa cohomology}.
\end{hypothesis}

\begin{notation}\label{N:char type over Qp}
Let $L$ be a finite extension of $\Qp$ and assume $A$ is an
$L$-affinoid algebra.  For $\delta: \Qp^\times \to L^\times$ a
continuous character and $M$ a $(\varphi,\Gamma)$-module over
$\calR_A$, we equip the twist $M(\delta) = M \otimes_L L\bbe_\delta$
with the diagonal action of $\varphi$ and $\Gamma$, where
$\varphi(\bbe_\delta) = \delta(p)\bbe_\delta$ and $\gamma(\bbe_\delta)
= \delta(\chi(\gamma))\bbe_\delta$.  We define $w(\delta) =
\log(\delta(a)) / \log(a)$ for nontorsion $a \in \Zp^\times$, called
the \emph{weight} of $\delta$.  We write $x$ for the natural embedding
$\Qp^\times \to \Qp^\times$.  We have $\calR(x) \cong t\calR$ and it
has weight $1$ (from the action of $\Zp^\times$) and slope $-1$ (from
the action of $p$; note the sign convention).  Similarly,
$\calR(x^k\delta) \cong t^k\calR(\delta)$ for $k \in \ZZ$.
\end{notation}

\begin{remark}
An easy application of local class field theory, as in
\cite[Proposition~3.1]{colmez}, shows that every rank one
$(\varphi,\Gamma)$-module over $\calR_L$ is of the form
$\calR(\delta)$ for some $\delta : \Qp^\times \to L^\times$.  In
Section~\ref{SS:rank one}, we reprise and extend the preceding
notation, replacing $\Qp^\times$ by $K^\times$ for $K/\Qp$ an
arbitrary finite extension, and replacing the target by
$\Gamma(X,\calO_X)^\times$ for $X$ any $L$-rigid analytic space, where
$L$ contains a full set of Galois conjugates of $K$.  We then prove
the corresponding generalization of the classification result.
\end{remark}

\begin{lemma} \label{L:special rank 1}
Let $\delta: \Qp^\times \to L^\times$ be a continuous character and let $\calR(\delta)$ be the associated $(\varphi, \Gamma)$-module.
\begin{itemize}
\item[(1)] We have $H^0_{\varphi, \gamma_{\Qp}}(\calR(\delta)) \neq0$ if and only if $\delta = x^i$ for $i \in \ZZ_{\leq0}$.  For $i\leq 0$, we have $H^0_{\varphi, \gamma_{\Qp}}(\calR(x^i))= L t^{-i}\bbe$.
\item[(2)] Any $(\varphi, \Gamma)$-submodule of $\calR(\delta)$ is of the form $t^i\calR(\delta)$ for some $i \in \ZZ_{\geq 0}$.
\item[(3)] If $\delta(p) \notin p^\ZZ$, $H^0_{\varphi, \gamma_{\Qp}}(\calR(\delta)) = H^2_{\varphi, \gamma_{\Qp}}(\calR(\delta)) =0$.
\item[(4)] If $w(\delta) \notin \ZZ$, the natural morphism $H^1_{\varphi,\gamma_{\Qp}}(x^k\delta) \to H^1_{\varphi,\gamma_{\Qp}}(x^{k'}\delta)$ is an isomorphism for any $k \geq k'$.
\end{itemize}
\end{lemma}
\begin{proof}
(1) is \cite[Proposition~2.1]{colmez}. From this, we have for any
  continuous character $\delta': \Qp^\times \to L^\times$,
  $\Hom_{\varphi, \gamma_{\Qp}}(\calR(\delta'), \calR(\delta)) \cong
  H^0_{\varphi, \gamma_{\Qp}}(\calR(\delta\delta'^{-1}))$ is nonzero if and
  only if $\delta' = \delta x^i$ for some $i\in \ZZ_{\geq 0}$, hence
  (2).  (3) follows from (2) and the Tate duality
  (Theorem~\ref{T:Liu}(3)).  (4) is \cite[Theorem~2.22]{colmez}.
\end{proof}

\begin{lemma} \label{L:base change to point}
If $N$ is a $(\varphi, \Gamma)$-module over $\calR_A$ such that $N/(\psi-1)=0$, then for any $z \in \Max(A)$,
the map $H^1_{\varphi, \gamma_{\Qp}}(N) \otimes_A A /\gothm_z \to H^1_{\varphi,\gamma_{\Qp}}(N_z)$ is bijective.
\end{lemma}
\begin{proof}
By Proposition~\ref{P:phi cohomology = psi cohomology}, it suffices to prove it for the $(\psi, \Gamma)$-cohomology.  Consider the spectral sequence $\Tor_i^{A}(H^j_{\psi, \gamma_{\Qp}}(N), A/\gothm_z) \Rightarrow H^{j-i}_{\psi, \gamma_{\Qp}}(N_z)$ (this case of base change being known because $A \to A/\gothm_z$ is finite).  Since $N/(\psi-1) = 0$, we have $H^2_{\psi, \gamma_{\Qp}}(N)= 0$.  The lemma follows from an inspection of this spectral sequence.
\end{proof}

\begin{lemma} \label{L:devissage}
Let $M$ be a $(\varphi, \Gamma)$-module over $\calR_A$ such that $M^*/(\psi-1)=0$ and $M/(\psi-1)\neq 0$.
Then there exists a short exact sequence $0 \to M \to N \to P \to 0$ of $(\varphi, \Gamma)$-modules
over $\calR_A$ with the following properties.
\begin{enumerate}
\item[(a)]
The $(\varphi, \Gamma)$-module $P$ is of rank one and $P/(\psi-1) = P^*/(\psi-1)=0$.
\item[(b)]
The connecting homomorphism $P^{\psi=1} \to M/(\psi-1)$ is nonzero.
\end{enumerate}
In particular, $N^*/(\psi-1)=0$ and $N/(\psi-1)$ is a quotient of
$M/(\psi-1)$ with nontrivial kernel.
\end{lemma}
\begin{proof}
Since $M / (\psi-1)$ is nonzero and is finite over $A$ by
Proposition~\ref{P:psi cokernel}(1), we may find (and fix) $(z,\eta)
\in \Max(\calR_A^\infty(\Gamma))$ such that $H^2_{\varphi,
  \gamma_{\Qp}}(M_z \otimes \eta^{-1}) \neq 0$.  Let $\delta: \Qp^\times
\to \Qp^\times$ be any continuous character such that $\delta(p) \notin p^\ZZ$
and $w(\delta \otimes \eta) \notin \ZZ$, and take $P =
\calR_A(\delta)$.  We will show that property (a) holds for any such
$\delta$, and that after replacing $\delta$ by $x^n\delta$ for $n \gg
0$ (which does not disturb the hypotheses on $\delta$), we may choose
$\alpha \in H^1_{\varphi,\gamma_{\Qp}}(M(\delta^{-1}))$ corresponding
to an extension $N$ as above such that (b) holds.

Lemma~\ref{L:special rank 1}(3) shows that for any $\eta' \in
\Max(\calR^\infty(\Gamma))$, we have $P / (\psi-1, \gothm_{\eta'})
\cong H^2_{\varphi, \gamma_{\Qp}}(\calR(\delta) \otimes \eta'^{-1}) =
0 $.  Hence $P / (\psi-1) = 0$ by its finiteness in
Proposition~\ref{P:psi cokernel}(1).  The same argument proves $P^* /
(\psi-1) = 0$.  Thus (a) holds.

Recall our fixed $(z,\eta) \in \Max (\calR_A^\infty(\Gamma))$ from
above.  We denote the map $H^1_{\varphi,\gamma_{\Qp}}(M(\delta^{-1}))
\to H^1_{\varphi,\gamma_{\Qp}}(M_z(\delta^{-1}))$ by $\alpha \mapsto
\alpha_z$.  To fulfill condition (b), we just need that the connecting
map
\[
P^{\psi=1} / \gothm_{(z,\eta)} \cong H^1_{\varphi,
\gamma_{\Qp}}(\calR(\delta) \otimes \eta^{-1})
\xrightarrow{d^1_{\alpha_z \otimes \eta^{-1}}}
H^2_{\varphi, \gamma_{\Qp}}(M_z
\otimes \eta^{-1}) = M/(\psi-1, \gothm_{(z,\eta)})
\]
is nontrivial.  In fact, dualizing the above map gives the boundary
map
\[
H^0_{\varphi, \gamma_{\Qp}}(M^*_z \otimes \eta)
\xrightarrow{d^0_{\alpha_z^* \otimes \eta}}
H^1_{\varphi, \gamma_{\Qp}}(\calR(\delta^{-1})(1) \otimes \eta),
\]
where $\alpha_z^* \otimes \eta \in
\mathrm{Ext}^1_{\varphi,\gamma_{\Qp}}(M^*_z \otimes \eta,
\calR(\delta^{-1})(1) \otimes \eta)$; fixing any nonzero $\beta_z \in
H^0_{\varphi, \gamma_{\Qp}}(M^*_z \otimes \eta)$, we will arrange a
choice of $\alpha$ so that $d^0_{\alpha_z^* \otimes \eta}(\beta_z)
\neq 0$.

As a first step, we compute $d^0_{\alpha_z^* \otimes \eta}(\beta_z)$.
Considering $\beta_z$ as a morphism $\calR_{\kappa_{(z,\eta)}} \to
M_z^* \otimes \eta$, and its twisted dual $\beta_z^* \otimes
\delta^{-1} \eta$ as a map $M_z \otimes \delta^{-1} \to
\calR_{\kappa_{(z,\eta)}}(\delta^{-1})(1) \otimes \eta$, we have a
commutative diagram (whose downward arrows are obtained by dualizing
and twisting extensions)
\[
\xymatrix@C=70pt{
\mathrm{Ext}^1_{\varphi, \gamma_{\Qp}}(M_z^* \otimes \eta,
\calR(\delta^{-1})(1) \otimes \eta) \ar[d]^\cong
\ar[r]^{\mathrm{Ext}^1_{\varphi,\gamma_{\Qp}}(\beta_z,-)} &
\mathrm{Ext}^1_{\varphi, \gamma_{\Qp}}(\calR_{\kappa_{(z,\eta)}},
\calR(\delta^{-1})(1) \otimes \eta ) \ar[d]^\cong
\\
H^1_{\varphi, \gamma_{\Qp}}(M_z (\delta^{-1}) ) \otimes_{\kappa_z}
\kappa_{(z,\eta)} \ar[r]^{H^1(\beta_z^* \otimes \delta^{-1} \eta)}
& H^1_{\varphi,\gamma_{\Qp}}( \calR(\delta^{-1})(1) \otimes \eta).
}
\]
It follows that
\[
d^0_{\alpha_z^* \otimes \eta}(\beta_z) = \beta_z \cup (\alpha_z^*
\otimes \eta) =
\mathrm{Ext}^1_{\varphi,\gamma_{\Qp}}(\beta_z,-)(\alpha_z^* \otimes
\eta) = H^1(\beta_z^* \otimes \delta^{-1}\eta)(\alpha_z
\otimes 1).
\]

Let $E_z$ be the kernel of $\beta_z^*$, so that Lemma~\ref{L:special
  rank 1}(2) gives a short exact sequence $0 \to E_z \to M_z \otimes
\eta^{-1} \xrightarrow{\beta_z^*} t^r \calR_{\kappa_{(z, \eta)}}(1)
\to 0$ for some $r \geq 0$.  By Proposition~\ref{P:psi cokernel}(2),
we may replace $\delta$ by $x^n\delta$ for $n$ sufficiently large so
that $H^2_{\varphi, \gamma_{\Qp}}(E_z(\delta^{-1}) \otimes \eta) = 0$
and $M(\delta^{-1}) / (\psi-1) = 0$.  We have the following diagram
\begin{equation}
\label{E:devissage}
\xymatrix@C=15pt{
H^1_{\varphi, \gamma_{\Qp}}(M_z(\delta^{-1})) \ar[r] &
H^1_{\varphi, \gamma_{\Qp}}(t^r\calR_{\kappa_{(z,\eta)}}(\delta^{-1})(1)\otimes \eta) \ar[d]^{\cong}\ar[r] &
H^2_{\varphi, \gamma_{\Qp}}(E_z(\delta^{-1}) \otimes \eta) = 0\\
& H^1_{\varphi, \gamma_{\Qp}}(\calR_{\kappa_{(z,\eta)}}(\delta^{-1})(1)\otimes \eta).
}
\end{equation}
The top row is an exact sequence and the vertical arrow is an
isomorphism by Lemma~\ref{L:special rank 1}(4).  By the Euler
characteristic formula (Theorem~\ref{T:Liu}(2)), $H^1_{\varphi,
  \gamma_{\Qp}}(t^r \calR_{\kappa_{(z,\eta)}}(\delta^{-1})(1)\otimes
\eta) \neq 0$.  To complete the proof, it suffices to choose an
element $\alpha_z \in H^1_{\varphi, \gamma_{\Qp}}(M_z(\delta^{-1}))$
with nontrivial image in $H^1_{\varphi,
  \gamma_{\Qp}}(\calR_{\kappa_{(z,\eta)}}(\delta^{-1})(1)\otimes
\eta)$, and use Lemma~\ref{L:base change to point} to lift this
to an element $\alpha$ of $H^1_{\varphi,
  \gamma_{\Qp}}(M(\delta^{-1}))$.
\end{proof}

We now can finish the proof of Theorem~\ref{T:finite Iwasawa cohomology}.  For this, we return to the setting of its statement.

\begin{proof}[Proof of Theorem~\ref{T:finite Iwasawa cohomology}]  
\label{Proof:finite Iwasawa cohomology}
First, by Lemmas~\ref{L:theorem for K=>for Qp}, \ref{L:reduced}, and \ref{L:twist by t}, we may assume that $K=\Qp$, $A$ is reduced, and $M^*/(\psi-1) = 0$ (by Proposition~\ref{P:psi cokernel}(2)).
We then apply Lemma~\ref{L:devissage} and noetherian induction to get a $(\varphi, \Gamma)$-module $N$ over $\calR_A$ which is a successive extension of $M$ by finitely many rank one $(\varphi, \Gamma)$-modules $P_i$ such that $P_i/(\psi-1) = P_i^*/(\psi-1) = 0$, and $N/(\psi-1) = N^*/(\psi-1)=0$.

By Theorem~\ref{T:structure of M^psi=1}, $N^{\psi=1}$ and $P_i^{\psi=1}$ are all finite projective $\calR_A^\infty(\Gamma)$-modules.  This implies that $\rmC^\bullet_\psi(M) \in \bbD_\perf^-(\calR_A^\infty(\Gamma))$, finishing the proof of the main theorem.
\end{proof}

\section{Triangulation}
\label{S:triangulation}

We conclude by giving an application of our finiteness results for
cohomology of $(\varphi, \Gamma)$-modules to the study of global
triangulations of $(\varphi, \Gamma)$-modules. Our main result in this
area (Corollary~\ref{C:global triangulation}) shows that the existence
of triangulations at a small (Zariski dense) set of points implies the
existence of a global triangulation over a much larger (Zariski open)
subset.  More precisely, we prove that the global triangulation exists
after making a suitable blowup.  This has direct applications to
eigenvarieties; we get an especially precise result for the
Coleman-Mazur eigencurve, and relate the error term in the global
triangulation with the action of the Theta operator considered by
Coleman \cite{coleman}.  Although our language and methods are
different, this improves Kisin's result \cite{R:Kisin} on
interpolating crystalline periods (needing no ``$Y$-smallness"
condition).

\setcounter{theorem}{0}
\begin{hypothesis}
Throughout this section, let $K$ be a finite extension of $\Qp$ of degree $d$, with ring of integers $\calO_K$ and residue field $k_K$ of cardinality $p^f$.  Let $K_0$ denote the fraction field of the ring of Witt vectors $W(k_K)$.

The Artin map gives an isomorphism between the multiplicative group $K^\times$ and the maximal abelian quotient of the Weil group of $K$; we normalize it so that the image of a uniformizer of $K$ is a \emph{geometric} Frobenius.  We hereafter identify the two groups.

We also let $L$ denote a finite extension of $\Qp$ (not necessarily contained in our chosen $\overline\QQ_p$) such that the set $\Sigma$ of $\Qp$-algebra embeddings $K \hookrightarrow L$ has the maximal cardinality $d$.  In the decomposition $K \otimes_{\Qp} L = \prod_{\sigma \in \Sigma} L$, we write $e_\sigma$ for the idempotent projecting onto the $\sigma$-factor, which satisfies $e_\sigma(x \otimes 1) = \sigma(x)$ for $x \in K$.  Note that $L$ may be replaced by a finite extension at any time without altering the hypothesis or $\Sigma$.
\end{hypothesis}

\begin{notation}
For $M$ a finite projective module over a ring $R$, an $R$-submodule $N$ of $M$ is called \emph{saturated} if $M/N$ is a (finite) projective $R$-module.  In particular, this condition implies that $N$ itself is a finite projective $R$-module.

For $M$ a $(\varphi, \Gamma_K)$-module over $\calR_A(\pi_K)$, a $(\varphi, \Gamma_K)$-submodule $N$ of $M$ is called \emph{saturated} if it is saturated as an $\calR_A(\pi_K)$-submodule.  In this case, $M/N$ is a $(\varphi, \Gamma_K)$-module over $\calR_A(\pi_K)$.  When $A = L$ so that $\calR_L(\pi_K)$ is a B\'ezout domain, the submodule $N$ is saturated in $M$ if and only if there is no $\calR_L(\pi_K)$-submodule of $M$ strictly containing $N$ with the same rank.
\end{notation}

\subsection{Moduli spaces of continuous characters}

We will show later (Construction~\ref{Con: rank one varphi, Gamma_K modules} and Theorem~\ref{T:full rank 1 classification}) that
$(\varphi, \Gamma_K)$-modules of rank one on a rigid $L$-analytic space $X$ are essentially classified by continuous characters $\delta: K^\times \to \Gamma(X, \calO_X)^\times$. In preparation for this, we devote this subsection to the description of such characters in terms of a certain moduli space.

To begin with, we state a general result on moduli spaces of continuous characters.

\begin{prop} \label{P:moduli space of characters}
Let $G$ be a commutative $p$-adic Lie group that modulo some compact open subgroup becomes free abelian on finitely many generators.  There exist a rigid analytic space $X^\an(G)$ over $\Qp$ and a continuous character $\delta_G: G \to \Gamma(X^\an(G), \calO_{X^\an(G)})^\times$ with the following property: for any rigid analytic space $X$, every continuous character $\delta: G \to \Gamma(X, \calO_X)^\times$ arises by pullback from a unique morphism $X \to X^\an(G)$ of rigid analytic spaces. Moreover, $X^{\an}(G)$ is a smooth quasi-Stein space.
\end{prop}
\begin{proof}
Since this lemma is well-known, we just give a sketch here.
If $G$ can be written as a product of groups satisfying the hypotheses of the lemma, then we may treat each factor separately, and assemble the product of the respective results.  Thus, since $G$ is (noncanonically) isomorphic to $G_0 \times \Zp^r \times \ZZ^s$, where $G_0$ is a finite abelian group, it suffices to treat the cases of $G_0$, $\Zp$, and $\ZZ$.  Clearly one has $X^\an(G_0) = \Max(\Qp[G_0])$ and $X^\an(\ZZ) = \bbG_m^\an$.  One takes $X^\an(\Zp)$ to be the generic fiber of the $\gothm$-adic formal scheme $\mathop{\mathrm{Spf}}(\Zp[\![\Zp]\!])$, which via the Iwasawa isomorphism is the open unit disc in the variable $[\gamma_{\Zp}]-1$, where $\gamma_{\Zp} \in \Zp$ is a generator and the brackets denote a grouplike element in the completed group algebra.  In other words, as one can readily verify, for any affinoid algebra $A$ and any element $f \in A$ such that $f-1$ is topologically nilpotent, there exists a unique continuous character $\Zp \to A^\times$ sending $\gamma_{\Zp}$ to $f$, and 
conversely every such character arises in this fashion.
\end{proof}

\begin{remark}
For $G$ as in Proposition~\ref{P:moduli space of characters},
for any rigid $L$-analytic space $X$,
the natural bijection between continuous characters $\delta: G \to \Gamma(X, \calO_X)^\times$ and morphisms $X \to X^{\an}(G) \times_{\Qp} L$ has the property that a morphism defines a character by pulling back the universal character $\delta_G$.
In turn, morphisms $X \to X^{\an}(G) \times_{\Qp} L$ are in bijection with graphs of morphisms $X \to X^{\an}(G) \times_{\Qp} L$, i.e., closed analytic subvarieties of $X \times_L (X^{\an}(G) \times_{\Qp} L)$ which project isomorphically onto $X$.
\end{remark}

We now specialize this construction to various cases of interest.

\begin{example} \label{E:cyclotomic deformation space}
For $G = \Gamma_K$, the space $X^{\an}(G)$ provided by Proposition~\ref{P:moduli space of characters} is the cyclotomic deformation space $\Max(\calR_{\Qp}^\infty(\Gamma_K))$ (see Notation~\ref{N:cyclotomic deformation space}).
\end{example}

\begin{example} \label{E:moduli space of characters}
Take $G = K^\times$ in Proposition~\ref{P:moduli space of characters},
then put $X^\an_L = X^\an(G) \times_{\Qp} L$ and $\delta_L =
\delta_{G} \otimes 1$. This example will control $(\varphi, \Gamma_K)$-modules of rank one, as described in Construction~\ref{Con: rank one varphi, Gamma_K modules} and Theorem~\ref{T:full rank 1 classification}.
\end{example}

It will be useful to separate $X^{\an}_L$ into factors as follows.

\begin{example}  \label{E:factorization of characters}
Choose a uniformizer $\varpi_K$ of $K$ and use it to split $K^\times$ as a product $\calO_K^\times \times \varpi_K^\ZZ$. 
Then each continuous character $\delta$ on $K^\times$ factors uniquely as $\delta_1 \delta_2$ where $\delta_1$ is trivial on $\varpi_K^\ZZ$
and $\delta_2$ is trivial on $\calO_K^\times$. Correspondingly,
we may factor $X^{\an}_L$ as a product $X^{\an}_{L,1} \times X^{\an}_{L,2}$ where
$X^{\an}_{L,1} = X^{\an}(\calO_K^\times) \times_{\Qp} L$ and $X^{\an}_{L,2} = X^{\an}(\varpi^\ZZ) \times_{\Qp} L$.
In particular, there is a distinguished isomorphism $X^{\an}_{L,2} \cong \mathbb{G}^{\an}_{m,L}$ under which the coordinate on $\mathbb{G}^{\an}_{m,L}$ computes the evaluation of a character on $\varpi_K$.
\end{example}

It will also be helpful to further separate the factor $X^{\an}_{L,1}$
of $X^{\an}_L$.
\begin{defn}\label{D:weight}
Let $A$ be an $L$-affinoid algebra.
A continuous character $\delta: K^\times \to A^\times$ is
automatically locally $\Qp$-analytic \cite[Proposition~8.3]{buzzard}
(but not locally $K$-analytic in general), so we may look at the
action of the $\Qp$-Lie algebra of $K^\times$.  More precisely, the
\emph{weight} of $\delta$ is the image $\mathrm{wt}(\delta)$ under
$\prod_{\sigma \in \Sigma} A \cong K \otimes_{\Qp} A$ of the family of
elements $(\mathrm{wt}_{\sigma}(\delta))_{\sigma \in \Sigma}$ such
that
\[
\lim_{\substack{a \to 0 \\ a \in \calO_K}} \frac{|\delta( 1+a) - 1 - \sum_{\sigma \in \Sigma} \mathrm{wt}_\sigma(\delta)\sigma(a)|_A}{|a|_K} = 0
\]
(for any Banach algebra norm $|\cdot|_A$ defining the topology on
$A$).  The existence of such a set of elements follows from the fact
that $\delta$ is locally $\Qp$-analytic, and the uniqueness is
obvious.

Let $X$ be a rigid $L$-analytic space. By globalizing the previous definition, we may attach to any continuous character $\delta: K^\times \to \Gamma(X, \calO_X)^\times$ a \emph{weight} $\mathrm{wt}(\delta) \in K \otimes_{\Qp} \Gamma(X, \calO_X)$.
\end{defn}

\begin{defn} \label{D:Sen weight space}
Let $X^{\an}_{L,\mathrm{Sen}}$ be the rigid $L$-analytic variety obtained by forming
$\mathbb{G}^{\an}_{a,K}$, taking the Weil restriction from $K$ to $\Qp$, then base extending from $\Qp$ to $L$.
From Definition~\ref{D:weight},
we obtain a homomorphism $X^{\an}_{L,1} \to X^{\an}_{L,\mathrm{Sen}}$
such that for any rigid $L$-analytic space $X$ and any continuous character $\delta: K^\times \to \Gamma(X, \calO_X)^\times$,
the composition of the induced map $X \to X^{\an}_L$ with the projections $X^{\an}_L \to X^{\an}_{L,1} \to X^{\an}_{L,\mathrm{Sen}}$
is the map $X \to X^{\an}_{L,\mathrm{Sen}}$ associated to $\mathrm{wt}(\delta)$.

Let $X^{\an}_{L,\mathrm{fin}}$ denote the kernel of 
$X^{\an}_{L,1} \to X^{\an}_{L,\mathrm{Sen}}$.
The space $X^{\an}_{L,\mathrm{fin}}$ classifies characters on $\calO_K^\times$ of weight zero; since these characters are locally $\Qp$-analytic, these are exactly the characters of finite order. It follows that $X^{\an}_{L,\mathrm{fin}}$ consists of an \emph{infinite} discrete set of rigid analytic points.
\end{defn}

\subsection{$(\varphi,\Gamma_K)$-modules of character type}\label{SS:rank one}

We next make a detailed study of $(\varphi, \Gamma_K)$-modules of rank one on rigid analytic spaces. 
In one direction, we show that continuous characters on $K^\times$ give rise to $(\varphi, \Gamma_K)$-modules, said to be of \emph{character type}. The construction
amounts to a translation of a construction of Nakamura \cite{nakamura2}
from Berger's language of $\mathbb{B}$-pairs into the language of
$(\varphi, \Gamma_K)$-modules, which is better suited for arithmetic families. In the other direction, we show that every rank one $(\varphi, \Gamma_K)$-module on a rigid analytic space is of character type up to a twist by a line bundle on the base space.
Along the way, we develop some tools which will be useful later for analyzing trianguline $(\varphi, \Gamma_K)$-modules, including a criterion for when a rank one submodule of a $(\varphi,\Gamma_K)$-module over $\calR_L(\pi_K)$ is saturated (Lemma \ref{L:phi Gamma module saturated}), and a definition of Sen weights in families.

In order to be able to work over rigid analytic spaces which are not necessarily affinoid, we introduce the following definition.

\begin{defn}
For $X$ a rigid analytic space over $\Qp$ and $r>0$, we define
$\calR^r_X(\pi_K)$ to be the sheaf of analytic functions on $X \times
\Max(\calR^r_{\Qp}(\pi_K))$.  Let $\calR_X(\pi_K)$ be the direct limit
over $r>0$ of the $\calR_X^r(\pi_K)$.  For $?=r,\emptyset$, a
$(\varphi, \Gamma_K)$-module over $\calR_X^?(\pi_K)$ is a compatible
family of $(\varphi, \Gamma_K)$-modules over $\calR_A^?(\pi_K)$ for
each affinoid $\Max(A)$ of $X$.  Note, in particular, that for $X$ not
quasicompact a $(\varphi,\Gamma_K)$-module over $\calR_X(\pi_K)$ might
not arise as the base change from a single $\calR_X^r(\pi_K)$, as the
required $r$ may not be bounded away from zero as one ranges over the
affinoids of $X$.
\end{defn}

We begin by recalling a standard construction of
$(\varphi,\Gamma_{\Qp})$-modules over $\calR_X$, generalizing
Notation~\ref{N:char type over Qp} from the case $X = \Max(\Qp)$.

\begin{notation}
\label{N:rank one module}
Let $X$ be a rigid analytic space over $\Qp$.
For $\delta: \Qp^\times \to \Gamma(X,\calO_X)^\times$ a continuous character, we define $\calR_X(\delta)$ to be the free rank one $(\varphi, \Gamma)$-module $\calR_X\cdot\bbe$ with $\varphi(\bbe) = \delta(p)\bbe$ and $\gamma(\bbe) = \delta(\chi(\gamma))\bbe$ for $\gamma \in \Gamma$.
\end{notation}

The above notion generalizes from $\Qp$ to $K$, making use of the following variant of Hilbert's Theorem 90.

\begin{lemma}\label{L:make phi module}
Let $A$ be a $K_0$-algebra, and let $a \in A^\times$.  Up to isomorphism, there exists a unique free rank one $K_0 \otimes_{\Qp} A$-module $D_{f,a}$ equipped with a $\varphi \otimes 1$-semilinear operator $\varphi$ satisfying $\varphi^f = 1 \otimes a$.  One has $D_{f,ab} \simeq D_{f,a} \otimes_{K_0 \otimes_{\Qp} A} D_{f,b}$ for all $a,b \in A^\times$.
\end{lemma}
\begin{proof}
For existence, let $D_{f,a}$ be the free $A$-module on the basis
$e_0,\ldots,e_{f-1}$, extended to a $K_0 \otimes_{\Qp} A$-module via
the rule $(x \otimes 1)e_i = \varphi^i(x)e_i$.  One checks directly
that $D_{f,a}$ is free of rank one over $K_0 \otimes_{\Qp} A$, and
that the rule $\varphi : D_{f,a} \to D_{f,a}$ defined by
\[
\varphi(e_0) = ae_{f-1},
\quad
\varphi(e_1) = e_0,
\quad
\varphi(e_2) = e_1,
\quad \ldots, \quad
\varphi(e_{f-1}) = e_{f-2}
\]
and $A$-linearity is $(\varphi \otimes 1)$-linear with $\varphi^f = 1
\otimes a$.

For uniqueness, one reduces easily to the case $a=1$, and must show that any candidate $D_{f,1}$ is isomorphic as $K_0 \otimes_{\Qp} A$-module with $\varphi$-action to $K_0 \otimes_{\Qp} A$ itself equipped with $\varphi \otimes 1$.  Fix a basis $e$ of $D_{f,1}$ and write $\varphi(e) = be$, where necessarily $b \in (K_0 \otimes_{\Qp} A)^\times$ and $b \cdot (\varphi\otimes1)(b)\cdots(\varphi\otimes1)^{f-1}(b) = 1$.  We must find $c \in (K_0 \otimes_{\Qp} A)^\times$ such that $(\varphi \otimes 1)(c)/c = b$; given such $c$, the desired isomorphism sends $e \in D_{f,1}$ to $c \in K_0 \otimes_{\Qp} A$.  Writing, under the identification $K_0 \otimes_{\Qp} A = \prod_{i=0}^{f-1} A$, the element $b$ as the tuple $(b_0,\ldots,b_{n-1})$, one easily checks that each $b_i$ is a unit in $A$, and that the element $c$ identified to the tuple $(1,b_0,b_0b_1,\ldots,b_0b_1 \cdots b_{n-2})$ has the desired properties.

The identity $D_{f,ab} \simeq D_{f,a} \otimes_{K_0 \otimes_{\Qp} A} D_{f,b}$ is straightforward to see.
\end{proof}

We now give the construction of $(\varphi, \Gamma_K)$-modules of \emph{character type}.
\begin{construction}
\label{Con: rank one varphi, Gamma_K modules}
(Compare with Nakamura's construction in the language of $\BB$-pairs,
in \cite[\S1.4]{nakamura}.)  To a continuous character $\delta:
K^\times \to \Gamma(X,\calO_X)^\times$ we associate a $(\varphi,
\Gamma_K)$-module of rank $1$ over $\calR_X(\pi_K)$, denoted
$\calR_A(\pi_K)(\delta)$, as follows.  First, we may work locally, and
assume that $X = \Max(A)$ is affinoid.  Factor $\delta =
\delta_1\delta_2$ as in Example~\ref{E:factorization of characters}; then $\delta_1$, viewed as a character on $\calO_K^\times$, extends to a character
$\widehat\delta_1$ on $G_K^\mathrm{ab}$.
We let
$\calR_A(\pi_K)(\delta_1) = \bbD_\rig(\widehat\delta_1)$, and we let
$\calR_A(\pi_K)(\delta_2) = D_{f,\delta_2(\mathrm{uniformizer})}
\otimes_{(K_0 \otimes_{\Qp} A)} \calR_A(\pi_K)$ (in the notation of
the preceding lemma) with the evident induced $\varphi$-action, and
trivial $\Gamma_K$-action on the first factor.  Then, we put
$\calR_A(\pi_K)(\delta) = \calR_A(\pi_K)(\delta_1)
\otimes_{\calR_A(\pi_K)} \calR_A(\pi_K)(\delta_2)$.  Using the
tensor-functoriality of $\bbD_\rig$ in Theorem~\ref{T:Drig} and
$D_{f,a}$ in Lemma \ref{L:make phi module}, this definition is easily
checked to be independent of the choice of the uniformizer $\varpi_K$ in
Example~\ref{E:factorization of characters}
 and also tensor-functorial.

For a $(\varphi, \Gamma_K)$-module $M$ over $\calR_X(\pi_K)$, we put $M(\delta) = M \otimes_{\calR_X(\pi_K)} \calR_X(\pi_K)(\delta)$.
\end{construction}

We continue with some tools adapted to the case $A=L$, for use in treating the behavior of families at points of $\Max(A)$.

\begin{defn}
Let $M$ be a $(\varphi,\Gamma_K)$-module over $\calR_L(\pi_K)$ that is potentially semistable.  Thus $\bbD_\pst(M)$ is a filtered $(\varphi,N,\Gal(K'/K))$-module equipped with a $\Qp$-linear action of $L$ commuting with $(\varphi,N,\Gal(K'/K))$ and preserving the filtration; here $K'/K$ is a finite Galois extension inside $\overline\QQ_p$ over which $M$ becomes semistable, and we write $K_0'$ for its maximal absolutely unramified subfield.  It is well-known that $\bbD_\pst(M)$ is a free $K_0' \otimes_{\Qp} L$-module, so that $\bbD_\dR(M) = (K' \otimes_{K_0'} D)^{G_K}$ is a free $K \otimes_{\Qp} L$-module.  We decompose $\bbD_\dR(M) = \bigoplus_{\sigma \in \Sigma} e_\sigma \bbD_\dR(M)$ as a filtered $K \otimes_{\Qp} L$-module, and for $\sigma \in \Sigma$ we define the \emph{$\sigma$-Hodge-Tate weights} of $M$ to be the jumps in the filtration on $e_\sigma \bbD_\dR(M)$, counted with multiplicity given by the $L$-dimension of the respective graded pieces.
(For details on the above constructions, see Berger's dictionary \cite{berger-dictionary}.)
\end{defn}

\begin{example}
\label{Ex:rank one modules}  Choose a uniformizer $\varpi_K$ of $K$.

(1) The unique homomorphism $\LT_{\varpi_K}: K^\times \to K^\times$ that is the identity on $\calO_K^\times$ and satisfies $\LT_{\varpi_K}(\varpi_K) = 1$ extends to $G_K^\mathrm{ab}$, whereupon it describes the action on the torsion in the Lubin-Tate group over $\calO_K$ associated to $\varpi_K$.  The structure of $\bbD_\cris(\LT_{\varpi_K})$ is well-known.  As a $K_0 \otimes_{\Qp} K$-module with $\varphi \otimes 1$-linear operator, it is isomorphic to $D_{f,\varpi_K^{-1}}$ in the notation of Lemma \ref{L:make phi module}, because the $p^f$-power Frobenius of the special fiber of the Lubin-Tate group is by construction equal to formal multiplication by $\varpi_K$.  Choosing an embedding $\sigma_0 \in \Sigma$ and setting $\LT_{\varpi_K,\sigma_0} = \sigma_0 \circ \LT_{\varpi_K}$, the $\sigma$-Hodge-Tate weight of $\LT_{\varpi_K,\sigma_0}$ is $-1$ if $\sigma=\sigma_0$ and $0$ if $\sigma \neq \sigma_0$; this is because the height and dimension of the Lubin-Tate group are $d$ and $1$, and the formal multiplication acts on its tangent space via the structure map.

(2) Fix an embedding $\sigma_0 \in \Sigma$.  Use $x_{\sigma_0}$ to denote the character $K^\times \to L^\times$ given by this embedding, and $\delta_{\sigma_0}: K^\times \to L^\times$ to denote the character given by $\delta_{\sigma_0}|_{\calO_K^\times} = 1$ and $\delta_{\sigma_0}(\varpi_K) = \sigma_0(\varpi_K)$.  One has $x_{\sigma_0} = \LT_{\varpi_K, \sigma_0} \cdot \delta_{\sigma_0}$, so that
\[
\calR_L(\pi_K)(x_{\sigma_0}) \simeq \calR_L(\pi_K)(\LT_{\varpi_K,\sigma_0}) \otimes_{\calR_L(\pi_K)} \calR_L(\pi_K)(\delta_{\sigma_0}).
\]
This allows us to compute, under Berger's dictionary \cite[Th\'eor\`eme~A]{berger-dictionary}, that $\bbD_\cris(\calR_L(\pi_K)(x_{\sigma_0}))$ is isomorphic to $D_{f,1}$, and that its $\sigma$-Hodge-Tate weight is $-1$ if $\sigma=\sigma_0$ and $0$ if $\sigma \neq \sigma_0$.  In particular, $\calR_L(\pi_K)(x_{\sigma_0})$ is a $(\varphi, \Gamma_K)$-submodule of the trivial one $\calR_L(\pi_K)$.

(3) More generally, let $\delta: K^\times \to L^\times$ be a continuous character such that $\delta|_{\calO_K^\times} = \prod_{\sigma \in \Sigma} x_\sigma^{k_\sigma}|_{\calO_K^\times}$ for some integers $k_\sigma$.
Then $\bbD_\cris(\calR_L(\pi_K)(\delta))$ is isomorphic to $D_{f,\alpha}$ with $\alpha = \delta(\varpi_K) \cdot \prod_{\sigma \in \Sigma} \sigma(\varpi_K)^{-k_\sigma}$, and its $\sigma$-Hodge-Tate weight is $-k_\sigma$ for each $\sigma \in \Sigma$.
\end{example}

\begin{notation}
\label{N:t_sigma}
By Example~\ref{Ex:rank one modules}(2), for each $\sigma \in \Sigma$, the module $\calR_L(\pi_K)(x_\sigma)$ is a submodule of $\calR_L(\pi_K)$.  Since $\calR_L(\pi_K)$ is a product of B\'ezout domains, there exists an element $t_\sigma \in \calR_L(\pi_K)$ generating the submodule $\calR_L(\pi_K)(x_\sigma)$. 
The element $t_\sigma$ is not uniquely determined, but the ideal it generates is.

We remark that $\prod_{\sigma \in \Sigma} t_\sigma = tu$ for some unit $u \in \calR_L(\pi_K)^\times$; this follows from the fact that $\prod_{\sigma \in \Sigma} x_\sigma = N_{K/\Qp}$ as characters of $K^\times$.  Also, the $t_\sigma$ are strongly coprime, in that $(t_\sigma) + (t_\tau) = \calR_L(\pi_K)$ for $\sigma \neq \tau$, which one sees by applying Berger's dictionary to $(t_\sigma)+(t_\tau)$ and considering its possible Hodge-Tate weights.  In particular, one has
\begin{equation}\label{E:HT wts decompose}
\calR_L(\pi_K) / t \cong \bigoplus_{\sigma \in \Sigma} \calR_{L}(\pi_K) / t_\sigma.
\end{equation}
\end{notation}

We now compute the cohomology of $(\varphi,
\Gamma_K)$-modules of character type over $L$.

\begin{prop}
\label{P:cohomology of rank one module}
Let $\delta: K^\times \to L^\times$ be a continuous character.
\begin{itemize}
\item[(1)] The cohomology $H^0_{\varphi, \gamma_K}(\calR_L(\pi_K)(\delta))$ is trivial unless $\delta = \prod_{\sigma \in \Sigma} x_\sigma^{k_\sigma}$ with all  $k_\sigma \leq 0$, in which case the dimension is equal to one.

\item[(2)] The cohomology
$H^2_{\varphi, \gamma_K}(\calR_L(\pi_K)(\delta))$ is trivial unless $\delta =|N_{K/\QQ}|\cdot  \prod_{\sigma \in \Sigma} x_\sigma^{k_\sigma}$ with all $k_\sigma \geq 1$, in which case the dimension is equal to one.

\item[(3)] the cohomology $H^1_{\varphi, \gamma_K}(\calR_L(\pi_K)(\delta))$ has dimension equal to $[K:\Qp]$ unless either the zeroth or second cohomology does not vanish, in which case the dimension is equal to $[K:\Qp]+1$.
\end{itemize}
\end{prop}
\begin{proof}
This follows from Nakamura's computation \cite[\S2.3]{nakamura}, together with the comparison of cohomology.
For the convenience of the reader, we include a proof using the language of $(\varphi, \Gamma_K)$-modules.

By the Euler characteristic formula (Theorem~\ref{T:Liu}(2)), the last statement follows from the first two.  By Tate duality (Theorem~\ref{T:Liu}(3)), the second statement follows from the first one.  It then suffices to prove (1).

Suppose that $H^0_{\varphi, \gamma_K}(\calR_L(\pi_K)(\delta))$ is nontrivial for some continuous character $\delta: K^\times \to L^\times$.  Then we have a nonzero morphism $\calR_L(\pi_K) \to \calR_L(\pi_K)(\delta)$ of $(\varphi, \Gamma_K)$-modules; this is equivalent to the existence of a morphism $j: \calR_L(\pi_K)(\delta^{-1}) \to \calR_L(\pi_K)$ of $(\varphi, \Gamma_K)$-modules.  Since $\calR_L(\pi_K)$  is a finite direct product of integral domains and $\varphi$ acts transitively on these domains, this morphism $j$ must be injective, realizing $\calR_L(\pi_K)(\delta^{-1})$ as a subobject of $\calR_L(\pi_K)$.

As a subobject of a crystalline object is crystalline, $\calR_L(\pi_K)(\delta^{-1})$ corresponds to a nonzero subobject of the trivial filtered $\varphi$-module of rank one.  In particular, it must be isomorphic to $K_0 \otimes_{\Qp} L$ as a $\varphi$-module, and each $\sigma$-Hodge-Tate weight $-k_\sigma$ for $\sigma \in \Sigma$ is nonpositive.  Note that this subobject is exactly the filtered $\varphi$-module associated to the character $\prod_{\sigma \in \Sigma} x_\sigma^{k_\sigma}$ in the computation in Example~\ref{Ex:rank one modules}(3).  It follows that $\delta = \prod_{\sigma \in \Sigma} x_\sigma^{-k_\sigma}$.  As $\bbD_\cris$ is fully faithful on crystalline objects, and there is only one nonzero inclusion $\bbD_\cris(\calR_L(\pi_K)(\delta^{-1})) \hookrightarrow \bbD_\cris(\calR_L(\pi_K))$ up to scalar multiple, we see that $H^0_{\varphi,\gamma_K}(\calR_L(\pi_K)(\delta))$ is one-dimensional.
\end{proof}

\begin{cor}
\label{C:submodule of rank one module}
Let $\delta:K^\times \to L^\times$ be a continuous character.  Then a nonzero $(\varphi, \Gamma_K)$-submodule of $\calR_L(\pi_K)(\delta)$ over $\calR_L(\pi_K)$ is of the form $\left(\prod_{\sigma \in \Sigma} t_\sigma^{k_\sigma}\right) \calR_L(\pi_K)(\delta)$ for some $k_\sigma \geq 0$.
\end{cor}
\begin{proof}
This follows immediately from Proposition~\ref{P:cohomology of rank one module}(1).
\end{proof}

\begin{lemma}
\label{L:phi Gamma module saturated}
Let $M$ be a $(\varphi, \Gamma_K)$-module over $\calR_L(\pi_K)$, and let $\lambda$ be a nonzero element of $H^0_{\varphi, \gamma_K}(M) \cong \Hom_{\calR_L(\pi_K)[\varphi,\Gamma_K]}(\calR_L(\pi_K),M)$.  Then $\lambda(\calR_L(\pi_K))$ is saturated in $M$ if and only if the image of $\lambda$ in $H^0_{\varphi, \gamma_K}(M/t_\sigma)$ is nonzero for every $\sigma \in \Sigma$.
\end{lemma}
\begin{proof}
For each embedding $\sigma \in \Sigma$, we have an exact sequence $0 \to H^0_{\varphi, \gamma_K}(M(x_\sigma)) \to H^0_{\varphi, \gamma_K}(M) \to H^0_{\varphi, \gamma_K}(M/t_\sigma)$.  
The image of $\lambda$ in the last term is zero if and only if it is the image of some $\lambda' \in H^0_{\varphi, \gamma_K}( M(x_\sigma))$.
This implies that if $\lambda$ dies in $H^0_{\varphi, \gamma_K}(M/t_\sigma)$, then $\lambda(\calR_L(\pi_K))$ embeds into $t_\sigma M$ and hence $t_\sigma^{-1}\lambda(\calR_L(\pi_K))$ embeds into $M$, rendering the submodule $\lambda(\calR_L(\pi_K))$ not saturated.

Conversely, if $\lambda(\calR_L(\pi_K)) \hookrightarrow M$ is not saturated, then dualizing this inclusion gives a nonsurjective morphism $M^\vee \to \calR_L(\pi_K)$ of $(\varphi, \Gamma_K)$-modules.  By Corollary~\ref{C:submodule of rank one module}, the image of this map lies in $t_\sigma\calR_L(\pi_K)$ for some $\sigma \in \Sigma$.  Dualizing back this map and multiplying both the source and the target by $t_\sigma$ gives a homomorphism $\lambda(\calR_L(\pi_K)) \to t_\sigma M$.  This proves that $\lambda$ lies in the image of $H^0_{\varphi, \gamma_K}(t_\sigma M)$.
\end{proof}

We next discuss the notion of weights in families.  Henceforth $A$
denotes an $L$-affinoid algebra.

\begin{defn}
\label{D:generalized Hodge-Tate weights}
Let $M$ be a $(\varphi, \Gamma_K)$-module over $\calR_A(\pi_K)$ of rank $d$.  By Lemma~\ref{L:structure of M/tM}, $(M/t)^{\varphi = 1}$ is of the form
\[
\bbD_{\Sen, n}(M) \otimes_{K(\mu_{p^n}) \otimes_{\Qp} A} (K(\mu_{p^\infty}) \otimes_{\Qp} A)
\]
for a sufficiently large $n \in \NN$, where $\bbD_{\Sen, n}(M)$ is a
locally free module of rank $d$ over $K(\mu_{p^n}) \otimes_{\Qp} A$.
Consider the \emph{Sen operator} $\Theta_\Sen =
\frac{\log(\gamma)}{\log(\chi(\gamma))}$ on $M$ for $\gamma \in
\Gamma_K$ sufficiently close to $1$.  It induces a $K(\mu_{p^n})
\otimes_{\Qp} A$-\emph{linear} action on $\bbD_{\Sen, n}(M)$.  The
characteristic polynomial of this linear action is monic with
coefficients in $K(\mu_{p^n}) \otimes_{\Qp} A$, but because it must be
$\Gamma_K$-invariant it descends to $K \otimes_{\Qp} A$.  It is
independent of $n$; we call it the \emph{Sen polynomial} of $M$.

When the rank of $M$ is one, the Sen operator acts by multiplication by an element of $K\otimes_{\Qp} A$, which is the unique root of the Sen polynomial. We call this the \emph{Sen weight} of $M$.
\end{defn}

For a $(\varphi, \Gamma_K)$-module of character type, we may
reconcile the Sen weight of the $(\varphi, \Gamma_K)$-module
with the weight of the underlying character.

\begin{lemma}
\label{L:weight of rank one module}
Let $\delta:K^\times \to A^\times$ be a continuous character.  Then
the Sen weight of $\calR_A(\pi_K)(\delta)$ is $\mathrm{wt}(\delta)$,
and for each $\sigma \in \Sigma$, and the operator $\Theta_\Sen$ acts
on $(\calR_A(\pi_K)(\delta)/t_\sigma)^{\varphi=1}$ by multiplication
by $\mathrm{wt}_\sigma(\delta)$.
\end{lemma}
\begin{proof}
First suppose $\delta: K^\times \to L^\times $ is a continuous
character such that $\delta|_{\calO_K^\times} = \prod_{\sigma \in
  \Sigma} x_\sigma^{k_\sigma}|_{\calO_K^\times}$ for some integers
$k_\sigma$.  Then $\mathrm{wt}_\sigma(\delta) = k_\sigma$ for each
$\sigma \in \Sigma$ by the definition of weights, and so the first
claim follows from the calculation of filtered $\varphi$-modules in
Example~\ref{Ex:rank one modules}(3).

We let $X_L^\an = X^\an(K^\times) \times_{\Qp} L$ and $\delta_L =
\delta_{K^\times} \otimes 1$ be as in Example~\ref{E:moduli space of
  characters}, and we let $S$ be the set of points of $X_L^\an$
corresponding to characters treated above.  The weight and Sen weight
of $\delta_L$ agree at all points of $S$, and since $S$ is Zariski
dense in $X_L^\an$ and $X_L^\an$ is reduced, this implies the weight
and Sen weight must be equal in $\Gamma(X_L^\an,\calO_{X_L^\an})$.
Now let $f : \Max(A) \to X_L^\an$ be the unique morphism such that
$\delta = f^*\delta_L$.  Pulling back along $f$ the identity of the
weight and Sen weight for $\delta_L$, we deduce it for $\delta$.

The second claim is just a restatement of the first, broken up
according to the decomposition
\[
\bbD_{\Sen, n}(M) \cong \bigoplus_{\sigma \in \Sigma} \bbD_{\Sen, n}(M) \otimes_{K \otimes_{\Qp} A, \sigma \otimes 1} A \cong
\bigoplus_{\sigma \in \Sigma} (M/t_\sigma)^{\varphi = 1}
\]
for $n$ sufficiently large.
\end{proof}

The remainder of this section will be devoted to proving that
Construction~\ref{Con: rank one varphi, Gamma_K modules} is exhaustive up to twisting by line bundles on the base space. We begin with the case over an artinian ring. (Compare
the following lemma with the corresponding claim in the language of
$\BB$-pairs, in \cite[Proposition~2.15]{nakamura2}.)

\begin{lemma} \label{L:artinian rank 1 classification}
Let $A$ be an artinian $L$-algebra and let 
$M$ be a $(\varphi, \Gamma_K)$-module of rank $1$ over $\calR_A(\pi_K)$.
Then there exists a unique continuous character $\delta: K^\times \to A^\times$ such that
$\calL = H^0_{\varphi, \gamma_K}(M(\delta^{-1}))$ is free of rank $1$ over $A$ and the natural map
$\calR_A(\pi_K)(\delta) \otimes_A \calL \to M$ is an isomorphism.
\end{lemma}
\begin{proof}
First assume that $A$ is a field, hence a finite extension of $\Qp$.
To see uniqueness, if $\calR_A(\pi_K)(\delta) \simeq
\calR_A(\pi_K)(\delta')$ then $\calR_A(\pi_K) \simeq
\calR_A(\pi_K)(\delta'\delta^{-1})$, and local class field theory
forces $\delta'\delta^{-1} = 1$.  We now show existence.  If we view
the $(\varphi, \Gamma_K)$-module $M$ as a $\varphi$-module over
$\calR_A$ (of possibly larger rank), it has pure slope $s$.  When it
is \'etale, this is essentially local class field theory.  Assume that
$A$ contains an element $\alpha$ of valuation $s$.  Let $\delta:
\Qp^\times \to A^\times$ be the character trivial on $\Zp^\times$
which sends $p$ to $\alpha$.  Then it is clear that $M
\otimes_{\calR_A} \calR_A(\delta^{-1})$ is \'etale and hence is of the
form $\calR_A(\pi_K)(\delta^*)$ for some character $\delta^*: K^\times
\to A^\times$.  Then $M$ is isomorphic to
$\calR_A(\pi_K)(\delta^*\cdot \delta \circ N_{K/\Qp})$.  When the
assumption that $A$ contains an element of valuation $s$ does not
hold, choose a finite Galois extension $A'/A$ with this property and
apply the preceding argument to obtain $M \otimes_A A' \cong
\calR_{A'}(\pi_K)(\delta)$ for some $A'$-valued character $\delta$.
For each $\sigma \in \Gal(A'/A)$, one has $\calR_{A'}(\pi_K)(\sigma
\circ \delta) \cong M \otimes_A A' \otimes_{A',\sigma} A \cong M
\otimes_A A' \cong \calR_{A'}(\pi_K)(\delta)$, and therefore by
uniqueness $\sigma \circ \delta = \delta$.  It follows that $\delta$
takes values in $A$.  To relate $M$ to $\delta$ over $A$, note first
that any nonzero element of $H^0_{\varphi,\gamma_K}((M \otimes_A
A')(\delta^{-1}))$ is an isomorphism, since it is obtained by applying
$\otimes \delta$ to a nonzero homomorphism from $\calR_{A'}(\pi_K)$
into the trivial rank one $(\varphi,\Gamma_K)$-module $(M \otimes_A
A')(\delta^{-1})$.  Then note that
\[
H^0_{\varphi,\gamma_K}(M(\delta^{-1})) \otimes_A A'
  \cong H^0_{\varphi,\gamma_K}((M \otimes_A A')(\delta^{-1})) \neq 0,
\]
so we may choose a nonzero element $f: \calR_A(\pi_K)(\delta) \to M$
of $H^0_{\varphi,\gamma_K}(M(\delta^{-1}))$.  Since $A \to A'$ is
faithfully flat, its image $f: \calR_{A'}(\pi_K)(\delta) \to M
\otimes_A A'$ in $H^0_{\varphi,\gamma_K}((M \otimes_A
A')(\delta^{-1}))$ is nonzero, hence an isomorphism, and again by
faithful flatness $f$ is an isomorphism.

We now treat the case where $A$ is an arbitrary artinian
$L$-algebra, which we may freely assume is local.  Let $I$ be the maximal ideal of $A$; we induct on the
nilpotency index $e$ of $I$.  If $e=1$, then $A$ is a field and the
claim is treated in the preceding paragraph.  Otherwise, using the
induction hypothesis, we know that $M \otimes_A A/I^{e-1}$ is of the
form $\calR_{A/I^{e-1}}(\pi_K)(\delta_{e-1})$ for an
$A/I^{e-1}$-valued character $\delta_{e-1}$.  Since the spaces
constructed in Proposition~\ref{P:moduli space of characters} are smooth, we may
lift $\delta_{e-1}$ to some character $\delta_e$ with values in $A/I^e
= A$.  Replacing $M$ with $M(\delta_e^{-1})$, we may assume that $M
\otimes_A A/I^{e-1}$ is trivial.  In this case, the options for $M$
are classified by $H^1_{\varphi, \gamma_K}(\calR_A(\pi_K) \otimes_A
I^{e-1})$.  We may also assume that the composition $K^\times
\stackrel{\delta}{\to} A^\times \to (A/I^{e-1})^\times$ is trivial, so
that the options for $\delta$ are classified by $H^1(K, I^{e-1})$.
The construction $\delta \mapsto \calR_A(\pi_K)(\delta)$ defines a map
$H^1(K, I^{e-1}) \to H^1_{\varphi, \gamma_K}(\calR_A(\pi_K) \otimes_A
I^{e-1})$ which coincides with the natural bijection from
Proposition~\ref{P:comparison with Galois cohomology}; this completes
the induction.
\end{proof}

The corresponding statement over a general rigid $L$-analytic space is
the following theorem. In the case $K = \Qp$, it provides an affirmative answer to a question of Bella\"\i che \cite[\S 3, Question 1]{R:Bellaiche}. 

\begin{theorem} \label{T:full rank 1 classification}
Let $X$ be a rigid $L$-analytic space
and let $M$ be a $(\varphi, \Gamma_K)$-module of rank $1$ over $\calR_X(\pi_K)$.
\begin{enumerate}
\item[(1)]
There exist a continuous character $\delta: K^\times \to \Gamma(X, \calO_X)^\times$ and an invertible sheaf $\calL$ on $X$, such that $M \simeq \calR_X(\pi_K)(\delta) \otimes_{\calO_X} \calL$. 
\item[(2)]
The character $\delta$ corresponds to the morphism $X \to X^{\an}_L$
whose graph is the space $\Gamma_M$ defined in Definition~\ref{D:putative graph}. In particular, it is unique.
\item[(3)]
The sheaf $\calL$ is unique up to isomorphism.
Moreover, one can take $\calL = H^0_{\varphi, \gamma_K}(M(\delta^{-1}))$, in which case the isomorphism is canonical.
\end{enumerate}
\end{theorem}

To complete the statement of the theorem, we exhibit the space $\Gamma_M$ using the cohomology of $(\varphi, \Gamma_K)$-modules.

\begin{defn} \label{D:putative graph}
Form the $(\varphi, \Gamma_K)$-module $M(\delta_L^{-1})$ over $X \times_L X^{\an}_L$.
By Theorem~\ref{T:finite cohomology 1} and Theorem~\ref{T:base change},
we may view $N' = H^2_{\varphi, \gamma_K}(M(\delta_L^{-1})^*)$
and $N'' = H^2_{\varphi, \gamma_K}(M^\dual(\delta_L)^*)$
as coherent sheaves on $X \times_L X^{\an}_L$ whose formation commutes with arbitrary (not necessarily flat) base change on $X$.
Let $\Gamma_M'$ and $\Gamma_M''$ be the closed analytic subspaces of $X \times_L X^{\an}_L$ cut out by the annihilator ideal sheaves of $N'$ and $N''$, respectively. Put $\Gamma_M = \Gamma_M' \times_{X \times_L X^{\an}_L} \Gamma_M''$.
\end{defn}

The proof of Theorem~\ref{T:full rank 1 classification} occupies the remainder of this subsection. We first verify that the theorem is equivalent to the claim that $\Gamma_M$ is indeed a graph.
\begin{lemma} \label{L:graph criterion}
Take $X,M$ as in Theorem~\ref{T:full rank 1 classification}.
\begin{enumerate}
\item[(1)]
If part (1) of Theorem~\ref{T:full rank 1 classification} holds for $M$, then
so does part (2); in particular, the map $\Gamma_M \to X$ is an isomorphism.
\item[(2)]
If $\Gamma_M \to X$ is an isomorphism, then Theorem~\ref{T:full rank 1 classification} holds in full for $M$.
\item[(3)]
The formation of $\Gamma_M', \Gamma_M'', \Gamma_M$ commutes with arbitrary base change on $X$. (We will use this property without comment in what follows.)
\end{enumerate}
\end{lemma}
\begin{proof}
We first check (1).
By translating using the group structure on $X^{\an}_L$, we may reduce the claim to the case where $\delta$ is the trivial character. Since the claim is local on $X$, we may also assume that $M$ is itself a trivial $(\varphi, \Gamma)$-module. By Theorem~\ref{T:base change}, we may further reduce to the case $X = \Max(L)$.
By Proposition~\ref{P:cohomology of rank one module}(2), for
$\delta': K^\times \to L^\times$ a continuous character, 
in order for
$H^2_{\varphi, \gamma_K}(\calR_L(\pi_K)(\delta')^*)$ to be nontrivial
either $\delta'$ must be trivial or $\delta'$ must carry a uniformizer of $K$ to an element of negative valuation. In particular, for both
$H^2_{\varphi, \gamma_K}(\calR_L(\pi_K)(\delta')^*)$ 
and
$H^2_{\varphi, \gamma_K}(\calR_L(\pi_K)(\left.\delta' \right.^{-1})^*)$ to be nontrivial,
the character $\delta'$ must be trivial.
Consequently, $\Gamma_M$ is supported at the identity point of $X^{\an}_L$; it thus remains to check that $\Gamma_M$ is reduced. Suppose the contrary; we can then choose a subspace $Y = \Max(A)$ of $\Gamma_M$ which is local artinian of length 2.
Let $M_Y$ be the pullback of $M(\delta_L^{-1})$ to $\calR_Y(\pi_K)$;
then $\dim_L H^2_{\varphi,\gamma_K}(M_Y^*) > 1$. By restricting scalars from $A$ to $L$ and then applying Theorem~\ref{T:finite cohomology}(3), we see that $\dim_L H^0_{\varphi,\gamma_K}(M_Y) > 1$; since
$\dim_L H^0_{\varphi,\gamma_K}(\bbm_A M_Y) \leq 1$
and $M_Y / \bbm_A M_Y \cong \calR_L(\pi_K)$ as a $(\varphi, \Gamma)$-module, the element $1 \in \calR_L(\pi_K) \cong M_Y/\bbm_A M_Y$ must lift to some $\bbe \in H^0_{\varphi,\gamma_K}(M_Y)$. Since $\bbe$ maps to a free generator of $M_Y/\bbm_A M_Y$, by Nakayama's lemma it is also a free generator of $M_Y$. However, this means that we have a nontrivial character (namely the restriction of $\delta_L^{-1}$ to $Y$) whose corresponding $(\varphi, \Gamma_K)$-module is trivial, contradicting the uniqueness aspect of Lemma~\ref{L:artinian rank 1 classification}.
Hence $\Gamma_M$ must be reduced, which completes the proof that $\Gamma_M$ coincides with the graph of $\delta$.

We next check (2). Let $\delta$ be the character corresponding to the morphism with graph $\Gamma_M$. To check the full conclusion of Theorem~\ref{T:full rank 1 classification}, again by translating on $X^{\an}_L$ we reduce to the case where $\delta$ is the trivial character. In case $X = \Max(A)$ for $A$ an artinian local ring, we know by Lemma~\ref{L:artinian rank 1 classification} that the conclusion of Theorem~\ref{T:full rank 1 classification} holds, necessarily for the trivial character.
In general, this implies that for each closed immersion $\Max(A) \to X$ with $A$ an artinian local ring, $H^2_{\varphi,\gamma_K}(M_A) = 0$ and $H^0_{\varphi,\gamma_K}(M_A)$ is free of rank 1. By restricting scalars from $A$ to $L$ and applying Theorem~\ref{T:Liu}, we also have
\begin{equation} \label{eq:middle cohomology in artinian case}
\dim_L H^1_{\varphi,\gamma_K}(M_A) = (1 + [K:\Qp]) \dim_L A.
\end{equation}
In particular,
$H^2_{\varphi,\gamma_K}(M)$ is locally free of rank 0, so by 
Theorem~\ref{T:base change}(2), the formation of $H^1_{\varphi,\gamma_K}(M)$ also commutes with arbitrary base change. By \eqref{eq:middle cohomology in artinian case}, $H^1_{\varphi,\gamma_K}(M)$ must be locally free of rank $1 + [K:\Qp]$, so by Theorem~\ref{T:base change}(2) again, the formation of $H^0_{\varphi,\gamma_K}(M)$ also commutes with arbitrary base change. Therefore $H^0_{\varphi,\gamma_K}(M)$ is locally free of rank 1, so all parts of Theorem~\ref{T:full rank 1 classification} hold.

We finally check (3). It suffices to check that the sheaves $N'$ and $N''$ are locally monogenic; 
this follows from Proposition~\ref{P:cohomology of rank one module}(2) and Nakayama's lemma.
\end{proof}

Lemma~\ref{L:graph criterion} allows us to parlay the artinian case of
Theorem~\ref{T:full rank 1 classification} into a partial result in the general case.

\begin{lemma} \label{L:open immersion}
Take $X,M$ as in Theorem~\ref{T:full rank 1 classification}.
\begin{enumerate}
\item[(1)]
The map $\Gamma_M \to X$ is an open immersion, is bijective on rigid analytic points, and is an isomorphism on affinoid subdomains of $\Gamma_M$.
\item[(2)]
For any positive integer $m$, the multiplication-by-$m$ map on $X^{\an}_L$ induces an isomorphism $\Gamma_M \cong \Gamma_{M^{\otimes m}}$.
\end{enumerate}
\end{lemma}
\begin{proof}
For $X = \Max(A)$ with $A$ an artinian ring, both parts are immediate from Lemma~\ref{L:artinian rank 1 classification} and
Lemma~\ref{L:graph criterion}(1). To deduce (1), recall that an isomorphism of affinoid spaces may be detected on the level of completed local rings \cite[Proposition~7.3.3/5]{BGR},
or equivalently on the level of artinian subspaces.

To deduce (2), note that since $\Gamma_M$ and $\Gamma_{M^{\otimes m}}$ are closed analytic subspaces of $X \times_L X^{\an}_L$ and the multiplication-by-$m$ map on $X^{\an}_L$ is finite, the induced map $\Gamma_M \to \Gamma_{M^{\otimes m}}$ is finite and hence quasicompact. Consequently, we may again use
\cite[Proposition~7.3.3/5]{BGR} to reduce to the artinian case.
\end{proof}

Lemma~\ref{L:open immersion}(1) does not suffice to imply that $\Gamma_M \to X$ is an isomorphism (see Example~\ref{E:open immersion}); one must also know that $\Gamma_M \to X$ is quasicompact.
To establish quasicompactness, we exploit the factorization of $X^{\an}_L$
and the relationship of $X^{\an}_{\mathrm{Sen}}$ to the Sen weight.
In the following lemmas, take $X,M$ as in Theorem~\ref{T:full rank 1 classification} and assume also that $X$ is affinoid.
\begin{lemma} \label{L:affinoid image2}
The image of $\Gamma_M \to X^{\an}_{L,2}$ is contained in an affinoid subdomain of $X^{\an}_{L,2}$.
\end{lemma}
\begin{proof}
Let $T$ be the coordinate on $X^{\an}_{L,2} \cong \mathbb{G}^{\an}_{m,L}$ corresponding to evaluation at $\varpi_K$. By Proposition~\ref{P:psi cokernel}(2) applied to the restriction of scalars of $M$ from $\calR_X(\pi_K)$ to $\calR_X$, the function $\log |T|$ is bounded below on $\Gamma_M'$
and bounded above on $\Gamma_M''$, and hence bounded on $\Gamma_M$. This proves the claim.
\end{proof}
\begin{cor} \label{C:closed immersion}
For any point $\eta \in X^{\an}_{L,1}$, the map
$\Gamma_M \times_{X^{\an}_{L,1}} \eta \to X$ is a closed immersion.
\end{cor}
\begin{proof}
Identify $X^{\an}_{L,2}$ with $\mathbb{G}^{\an}_{m,L}$ and then embed the latter into $\mathbb{P}^{1,\an}_L$. Then $\Gamma_M \times_{X^{\an}_{L,1}} \eta \to X \times_L \eta \times_L X^{\an}_{L,2}$ is a closed immersion and hence proper. The composition with $X \times_L \eta \times_L X^{\an}_{L,2} \to X \times_L X^{\an}_{L,2}$ is also proper because the latter map is proper. The further composition with $X \times_L X^{\an}_{L,2} \to X \times_L \mathbb{P}^{1,\an}_L$ is also proper because the image
of $\Gamma_M$ in $X^{\an}_{L,2}$ is contained in an affinoid subdomain by Lemma~\ref{L:affinoid image2}. The further composition with $X \times_L \mathbb{P}^{1,\an}_L \to X$ is also proper because the latter map is proper. 
Consequently, the map $\Gamma_M \times_{X^{\an}_{L,1}} \eta \to X$ is both proper and a locally closed immersion (by Lemma~\ref{L:open immersion}), hence a closed immersion.
\end{proof}

\begin{lemma}\label{L:finite order decomposition}
There exist a positive integer $m$ and a continuous character $\delta: K^\times \to \Gamma(X, \calO_X)^\times$
such that
$\Gamma_{M^{\otimes m}(\delta^{-1})}$ is contained in $X \times_L X^{\an}_{L,\mathrm{fin}} \times_L X^{\an}_{L,2}$.
\end{lemma}
\begin{proof}
Note that the Sen weight of $M^{\otimes m}$ equals $m$ times the Sen weight of $M$. By choosing $m$ to be a suitably large power of $p$, we may ensure that the Sen weight of $M^{\otimes m}$ belongs to the domain of convergence of the exponential map. In this case, it lifts to a continuous character $\delta: \calO_K^\times \to \Gamma(X, \calO_X)^\times$, which we may extend to $K^\times$ by decreeing that $\delta(\varpi_K) = 1$. 
To check that this choice of $m$ and $\delta$ has the desired effect,
we may reduce to the artinian case, then apply
Lemma~\ref{L:artinian rank 1 classification} and
Lemma~\ref{L:weight of rank one module}.
\end{proof}

We can now establish that $\Gamma_M \to X$ is quasicompact and hence complete the proof of Theorem~\ref{T:full rank 1 classification}.
\begin{proof}[Proof of Theorem~\ref{T:full rank 1 classification}]
By Lemma~\ref{L:graph criterion}, it suffices to prove that $\Gamma_M \to X$ is an isomorphism. 
For this, we may assume $X$ is affinoid and connected.

By Lemma~\ref{L:open immersion}(1),
$\Gamma_M \to X$ is bijective on rigid analytic points and isomorphic on affinoid subdomains of $\Gamma_M$; it thus suffices to prove that $\Gamma_M \to X$ is quasicompact. 
By Lemma~\ref{L:open immersion}(2), it suffices to prove quasicompactness after replacing $M$ with $M^{\otimes m}(\delta)$
for some positive integer $m$ and some continuous character $\delta: K^\times \to \Gamma(X,\calO_X)^\times$. By Lemma~\ref{L:finite order decomposition}, we may thus reduce to the case where 
$\Gamma_M$ is contained in $X \times_L X^{\an}_{L,1} \times_L X^{\an}_{L,\mathrm{fin}}$. Since the map $\Gamma_M \to X$ is surjective on points, there exists a point $\eta \in X^{\an}_{L,\mathrm{fin}}$ such that $\Gamma_M \times_{X^{\an}_{L,\mathrm{fin}}} \eta \neq \emptyset$.
The map $\Gamma_M \times_{X^{\an}_{L,\mathrm{fin}}} \eta \to X$ is on one hand a closed immersion (by Corollary~\ref{C:closed immersion})
and an open immersion (by Lemma~\ref{L:open immersion}(1) and the fact that $\eta \to X^{\an}_{L,\mathrm{fin}}$ is an open immersion). Since
$X$ is connected and $\Gamma_M \times_{X^{\an}_{L,\mathrm{fin}}} \eta$
is nonempty, $\Gamma_M \times_{X^{\an}_{L,\mathrm{fin}}} \eta \to X$ must be an isomorphism. In particular, the map $\Gamma_M \to X$, which is an open immersion by Lemma~\ref{L:open immersion}, must be an isomorphism and hence quasicompact. This completes the proof.
\end{proof}

We thank Michael Temkin for suggesting the following example.
\begin{example} \label{E:open immersion}
Put $X = \rmA^1[0,\infty]$ and $U = \rmA^1(0,\infty] \cup \rmA^1[0,0]$; that is, $X$ is a closed unit disc and $U$ is the union of the boundary circle and its complement. Then
$U$ is a quasi-Stein space and the natural map $U \to X$ is an open immersion which is a bijection on rigid analytic points but not an isomorphism. The corresponding map of Berkovich spaces is also a bijection, but the corresponding map of Huber spaces is not (it misses the type 5 point pointing into the open unit disc).

This example shows that Lemma~\ref{L:open immersion}(1) does not suffice for the proof of Theorem~\ref{T:full rank 1 classification},
thus necessitating the subsequent lemmas. It may be possible to give an alternate completion of the proof by proving directly that $\Gamma_M \to X$ is surjective on Huber points, but even the definition of the ring $\calR_A(\pi_K)$ for $A$ a field carrying a valuation of arbitrary rank is a bit delicate.
\end{example}

\subsection{Global triangulation}

The notion of a \emph{trianguline} representation of $G_{\Qp}$, one for which the associated $(\varphi, \Gamma)$-module is a successive extension of rank $1$ $(\varphi, \Gamma)$-modules, was introduced by Colmez \cite{colmez}.
This notion was later generalized by Nakamura \cite{nakamura} to representations of $G_K$ for $K$ a finite extension of $\Qp$, using Berger's language of $\BB$-pairs.
In this subsection, we redevelop the theory of trianguline representations of $G_K$ using $(\varphi, \Gamma)$-modules.  Although the two languages are equivalent, the approach using $(\varphi, \Gamma)$-modules is better suited for arithmetic families.

In this subsection, we use Construction~\ref{Con: rank one varphi, Gamma_K modules} to define the notion of a trianguline $(\varphi, \Gamma_K)$-module.
We then show that densely pointwise trianguline
$(\varphi,\Gamma)$-modules admit triangulations on Zariski-dense open
subsets.

\begin{defn}
\label{D:strictly triangular}
Let $\delta_1,\ldots,\delta_d: K^\times \to \Gamma(X, \calO_X)^\times$
be continuous characters.  A $(\varphi, \Gamma_K)$-module $M$ of rank
$d$ over $\calR_X(\pi_K)$ is \emph{trianguline with ordered parameters
  $\delta_1, \dots, \delta_d$} if, after perhaps enlarging $L$, there
exists an increasing filtration $(\fil_i M)_{i=0, \dots, d}$ of $M$
given by $(\varphi, \Gamma_K)$-submodules and line bundles
$\calL_1,\ldots,\calL_d$ on $X$ such that each $\gr_iM \simeq
\calR_X(\pi_K)(\delta_i) \otimes_{\calO_X} \calL_i$.  Such a
filtration is called a \emph{triangulation (with ordered parameters
  $\delta_1, \dots, \delta_d$)} of $M$.  Since every rank one object
is of the form $\calR_X(\pi_K)(\delta) \otimes_{\calO_X} \calL$ for a
uniquely determined $\delta$ and a uniquely determined (up to
isomorphism) $\calL$, it follows that $M$ is trianguline if and only
if after perhaps enlarging $L$ it is a successive extension of rank
one objects, and that given the filtration by saturated
$(\varphi,\Gamma_K)$-stable submodules with rank one graded pieces,
the ordered parameters and required line bundles are well-defined (up
to isomorphism).

In the case $X = \Max(L)$, we say that $M$ is \emph{strictly
  trianguline with ordered parameters $\delta_1, \dots, \delta_d$} if,
for each $i$, the the submodule $\fil_{i+1}M$ is the unique way of
enlarging $\fil_iM$ to a submodule of $M$ with quotient isomorphic to
$\calR_X(\pi_K)(\delta_{i+1})$ (and in particular the trianguline
filtration is the unique one with these ordered parameters).  When $M$
is already known to be trianguline with these ordered parameters, this
is equivalent to $H^0_{\varphi,
  \gamma_K}((M/\fil_iM)(\delta_{i+1}^{-1}))$ being one-dimensional for
all $i$, or alternatively $H^0_{\varphi,
  \gamma_K}((\fil_iM)^\vee(\delta_i))$ being one-dimensional for any
$i$.
\end{defn}

\begin{defn}
A subset $Z$ of closed points of $X$ is called \emph{Zariski dense} if there exists an admissible affinoid cover $\{\Max(A_i)\}_{i \in I}$ of $X$ such that $Z \cap \Max(A_i)$ is Zariski dense in $\Max(A_i)$ for each $i$.
Note that this does not imply that $Z$ is Zariski dense in all affinoid subdomains of $X$.

A $(\varphi, \Gamma_K)$-module $M$ over $\calR_X(\pi_K)$ is called \emph{densely pointwise strictly trianguline}  if there exist continuous characters $\delta_1, \dots, \delta_d: K^\times \to \Gamma(X, \calO_X)^\times$ and  a Zariski dense subset $X_\alg \subseteq X$ such that for each $z \in X_\alg$, $M_z$ is strictly trianguline with ordered parameters $\delta_{1,z}, \dots, \delta_{d,z}$.
\end{defn}

\begin{cor}
\label{C:finite cohomology sheaf}
Let $M$ be a $(\varphi, \Gamma_K)$-module over $\calR_X(\pi_K)$.  Then $H^*_{\varphi, \gamma_K}(M)$ and $H^*_{\varphi, \gamma_K}(M/t_\sigma)$ for any $\sigma \in \Sigma$ are coherent sheaves over $X$.  Moreover, locally on $X$ they are the cohomology of complexes of locally free sheaves concentrated in degrees $[0,2]$.
\end{cor}
\begin{proof}
The claims for $\rmC^\bullet_{\varphi,\gamma_K}(M)$ follow from
Theorem~\ref{T:finite cohomology}, so we show how to deduce the result
for $\rmC^\bullet_{\varphi,\gamma_K}(M/t_\sigma)$ from this.  It
suffices to assume that $X = \Max(A)$ is affinoid.  Since one has a
short exact sequence of complexes
\[
0 \to \rmC^\bullet_{\varphi,\gamma_K}(t_\sigma M) \to
\rmC^\bullet_{\varphi,\gamma_K}(M) \to
\rmC^\bullet_{\varphi,\gamma_K}(M/t_\sigma) \to 0,
\]
it is clear that $\rmC^\bullet_{\varphi,\gamma_K}(M/t_\sigma) \in
\bbD_\perf^{[-1,2]}(A)$.  But Lemma~\ref{C:R_A flat over A} shows that
$\rmC^\bullet_{\varphi,\gamma_K}(M/t_\sigma)$ consists of flat
$A$-modules, so Lemma~\ref{L:flat bottom cokernel} allows us to
conclude.
\end{proof}

Before moving on to the main theorem, we briefly discuss a flattening technique using Fitting ideals.  (This is inspired by the work of Raynaud and Gruson \cite{raynaud}.)

\begin{defn}
Given a finitely generated module $M$ over a noetherian ring $R$, write it as the cokernel of an $R$-linear map $\phi: R^m \to R^n$; let $\Phi \in \mathrm{M}_{n\times m}(R)$ be the matrix of $\phi$.
For $r \in \ZZ_{\geq 0}$, the \emph{$r$th Fitting ideal} of $M$ is defined to be the ideal $\Fitt_r(M)$ of $R$ generated by the $(n-r)$-minors of $\Phi$.  This is independent of the choice of the presentation \cite[Corollary-Definition~20.4]{eisenbud}.
Moreover, the construction of Fitting ideals commutes with arbitrary base change \cite[Corollary~20.5]{eisenbud}.
\end{defn}

\begin{lemma}
\label{L:linear algebra}
Let $R$ be a noetherian ring and let $M$ be a finitely generated module given as the cokernel of an $R$-linear homomorphism
 $\phi: R^m \to R^n$.

(1)
If $M \cong N \oplus F$ is a direct sum of $R$-modules with $F$ locally free of rank $d$, then $\Fitt_{r+d}(M) = \Fitt_r(N)$.

(2)
For any $r \in \ZZ_{\geq 0}$, $\Fitt_r(M)$ is a nil-ideal if and only if the rank of $M$ at each generic point of $\Spec R$ is at least $r+1$.

(3)
Assume that there exists $r \in \NN$ such that $\Fitt_{r-1}(M) = (0)$
and $\Fitt_r(M)$ is generated by $a \in R$ that is not a zero-divisor.
Then the image of $\phi$ is flat over $R$ of constant rank $n-r$.  In particular, $M$ has constant generic rank $r$ and it has projective-dimension $\leq 1$.

\end{lemma}
\begin{proof}
(1) We may work Zariski locally on $\Spec R$, reducing to the case where $F$ is free.  In this case, the claim is straightforward from the definitions.

(2) The ideal $\Fitt_r(M)$ consists of nilpotent elements if and only if $\Fitt_r(M)$ is contained in all minimal primes $\gothp$, if and only if, for each minimal prime $\gothp$, $\Fitt_r(M \otimes_R \Frac(R/\gothp)) = (0)$.  This reduces us to the case where $R$ is a field, in which case the claim follows from elementary linear algebra.

(3)
This follows from a variant of the proof of \cite[Chapter~4,~Proposition~1~and~Lemma~1]{raynaud}.
Let $\Phi = (\phi_{ij})_{1\leq i\leq n, 1\leq j \leq m}$ denote the matrix of $\phi$ for the standard bases $\bbe_1, \dots, \bbe_m$ for $R^m$ and $\bbf_1, \dots,\bbf_n$ for $R^n$.
Zariski locally on $\Spec R$, and after suitable permutation of the bases, and perhaps multiplying $a$ by a unit in $R$, we may assume that $\det(\phi_{ij})_{1 \leq i,j \leq n-r}  = a$.
Write $\Phi$ as a block matrix $\big(\begin{smallmatrix}
A & B \\ C & D
\end{smallmatrix}\big)$, where $A$ is an $(n-r)\times(n-r)$-matrix with determinant $a$.  Let $\Adj(A) $ denote the adjugate matrix of $A$.

We claim that we have the following identity
\begin{equation}
\label{E:linear algebra}
\begin{pmatrix}
A\\ C
\end{pmatrix}
\Adj (A) B = a
\begin{pmatrix}
B\\ D
\end{pmatrix}.
\end{equation}
It suffices to check that $C \Adj(A)B = a D$, and for this it suffices to replace $C$ by a single one of its rows, $B$ by a single one of its columns, and $D$ by the corresponding one of its entries (which is just an element of $R$).
In this case, it is a simple linear algebra exercise to check that
$
a(aD - C\Adj(A)B) = a \cdot \det
\big(\begin{smallmatrix}
A&B\\C& D
\end{smallmatrix}\big);
$
the right hand side vanishes because it is an $(n-r+1)$-minor of $\Phi$.
Since $a$ is not a zero-divisor, we have $aD = C \Adj(A)B$.

Note that each entry of $\Adj(A)B$ is an $(n-r) $-minor of $\Phi$, and hence divisible by $a$.  Since $a$ is not a zero-divisor, there is a unique matrix $E$ with coefficients in $R$ such that 
$\big(\begin{smallmatrix}
A\\C
\end{smallmatrix}\big)E = \big(
\begin{smallmatrix}
B\\ D
\end{smallmatrix}\big)$. This means that $\phi(\bbe_1), \dots, \phi(\bbe_{n-r})$ generate the image of $\phi$.  Since $\det A$ is not a zero-divisor, the image of $\phi$ is freely generated by $\phi(\bbe_1), \dots, \phi(\bbe_{n-r})$.  The lemma is proved.
\end{proof}

\begin{cor}
\label{C:flattening technique}
Let $X$ be a reduced rigid $L$-analytic space.

(1) Let $\phi: C \to D$ be a homomorphism of locally free coherent sheaves on $X$.
Assume that the generic rank $r$ of the cokernel of $\phi$ is constant on $X$.
Let $f: Y \to X$ be the blowup of $X$ along the $r$th Fitting ideal of the cokernel of $\phi$.
Then the homomorphism $f^*\phi: f^*C \to f^*D$ has flat image.
The formation of $(Y,f)$ commutes with dominant base change in $X$.
Moreover, for any morphism $g: Z \to Y$, the composite $f \circ g: Z \to X$ has the same property: $(f \circ g)^*\phi$ has flat image.

(2) Let $(C^\bullet, d^\bullet)$ be a bounded above complex of locally
free coherent sheaves on $X$.  Assume that the generic ranks of the
$H^i(C^\bullet)$ and the $\Coker(d^i)$ are constant on $X$.  Construct
the sequence of morphisms
\[
\cdots \to Y^i \xrightarrow{f_i} Y^{i+1} \xrightarrow{f_{i+1}} \cdots \to X
\]
as follows.  Let $Y^i = X$ for $i \gg 0$.  Given $f_i,f_{i+1},\ldots$, put
\[
g_i = \cdots \circ f_{i+1} \circ f_i: Y^i \longrightarrow X
\]
and apply (1) to $g_i^*d^{i-1}: g_i^*C^{i-1} \to g_i^*C^i$ to obtain $f_{i-1}: Y^{i-1} \to Y^i$.  Then the $(Y_i,f_i)$ depend only on the quasi-isomorphism class of $(C^\bullet,d^\bullet)$, and their formation commutes with dominant base change in $X$.
\end{cor}
\begin{proof}
We may assume that $X$ is affinoid throughout.  Also, we remark that
generic ranks do not change under dominant base change.

Then the first claim of (1) is clear from Lemma~\ref{L:linear
  algebra}.  The second claim follows because the formations of
Fitting ideals and blowups commute with base change.  The third claim
is obvious.

For (2), we prove that for each $i$ the construction of $f_{i-1}$
depends only on the quasi-isomorphism class of
$(C^\bullet,d^\bullet)$.  Note that our hypothesis on the generic
ranks of the $H^i(C^\bullet)$ and the $\Coker(d^i)$ implies that the
$\Image(d^i)$ and the $C^i$ also have generic ranks that are
constant on $X$.  We have the short exact sequence
\[
0 \to H^i(g_i^*C^\bullet) \to \Coker(g_i^*C^{i-1} \to g_i^*C^i) \to \Image(g_i^*C^i \to g_i^*C^{i+1}) \to 0,
\]
in which the last term is a flat module by the construction of $g_i$,
hence it is projective of constant rank $s$.  In particular, the short
exact sequence splits.  Let $r$ be the generic rank of
$H^i(g_i^*C^\bullet)$, so that $\Coker(g_i^*C^{i-1} \to g_i^*C^i)$ has
generic rank $r+s$.  Then $f_{i-1}$ is none other than the blowup of
the ideal
\[
\Fitt_{r+s}(\Coker(g_i^*C^{i-1} \to g_i^*C^i)) = \Fitt_r(H^i(g_i^*C^\bullet))
\]
on $Y^i$.  As the complex $(C^\bullet,d^\bullet)$ consists of locally
free coherent sheaves, it is clear that $H^i(g_i^*C^\bullet)$, and
therefore also the above ideal, depends only on its quasi-isomorphism
class.  The commutation with base change follows from the
corresponding property in (1).
\end{proof}

\begin{lemma}
\label{L:tor given by Fitting}
Let $R$ be a noetherian ring.
Let $Q$ be the cokernel of an injective homomorphism $\phi: M \to N$ of finitely generated flat $R$-modules of constant ranks $m$ and $n$, respectively.  Let $Z$ be the closed subscheme of $\Spec R$ defined by the $(n-m)$th Fitting ideal of $Q$.  Then for any $z \in \Spec R$, one has $\Tor_1^R(Q, \kappa_z) \neq 0$ if and only if $z \in Z$.
\end{lemma}
\begin{proof}
Tensoring $\phi$ with $\kappa_z$ gives $0 \to \Tor_1^R(Q, \kappa_z) \to M_z \xrightarrow{\phi_z} N_z \to Q_z \to 0$.  So $\Tor_1^R(Q, \kappa_z) \neq 0$ if and only if $\phi_z$ is not injective, if and only if the rank of $\phi_z$ is not $m$, if and only if the $(n-m)$th Fitting ideal vanishes at $z$, if and only if $z \in Z$.
\end{proof}

\begin{remark}
In the following theorem and its corollary, we assume that $X$ is a
reduced rigid analytic space over $L$.  The assumption on the
reducedness of $X$ is essential because we invoke arguments at the
residue fields, which do not see non-reduced structure.
\end{remark}

\begin{theorem}
\label{T:global triangulation step 1}
Let $X$ be a reduced rigid $L$-analytic space.
Let $M$ be a $(\varphi,\Gamma_K)$-module over $\calR_X(\pi_K)$ of rank $d$ and let $\delta: K^\times \to \Gamma(X, \calO_X)^\times$ be a continuous character.  Suppose that there exists a Zariski dense subset $X_\alg$ of closed points of $X$ such that for every $z \in X_\alg$, $H^0_{\varphi, \gamma_K}(M^\dual_z(\delta_z))$ is one dimensional, and the image of $\calR_{\kappa_z}(\pi_K)$ under any basis of this space is saturated in $M_z^\dual(\delta_z)$.  Then there exist canonical data of
\begin{itemize}
\item[(a)] a proper birational morphism $f: X' \to X$ of reduced rigid analytic spaces,
\item[(b)] a unique (up to $\calO_{X'}^\times$) homomorphism $\lambda: f^*M \to \calR_{X'}(\pi_K)(\delta) \otimes_{X'} \calL$ of $(\varphi, \Gamma_K)$-modules over $\calR_{X'}(\pi_K)$, where $\calL$ is a line bundle over $X'$ with trivial $(\varphi, \Gamma_K)$-actions,
\end{itemize}
such that the following conditions are satisfied.
\begin{itemize}
\item[(1)] The set $Z$ of closed points $z \in X'$ failing to have the following property is Zariski closed and disjoint from $f^{-1}(X_\alg)$ (hence its complement is Zariski open and dense): the induced homomorphism $\lambda_z: M_z \to \calR_{\kappa_z}(\pi_K)(\delta_z)$ is surjective and the corresponding element spans $H^0_{\varphi, \gamma_K}(M^\vee_z(\delta_z))$ (hence the latter is one-dimensional).
\item[(2)] Locally on $X'$, the cokernel of $\lambda$ is killed by some power of $t$, and is supported over $Z$ in the sense that for any analytic function $g$ vanishes along $Z$, some power of $g$ kills the cokernel of $\lambda$ too.
\item[(3)] The kernel of $\lambda$ is a $(\varphi, \Gamma_K)$-module over $\calR_{X'}(\pi_K)$ of rank $d-1$.
\end{itemize}

Moreover, when $X$ is a smooth rigid analytic curve, we can take $X'=X$.
\end{theorem}
\begin{proof}
If we prove the theorem for the normalization $p: \widetilde{X} \to X$
(with $p^*\delta$ and $p^{-1}(X_\alg)$) to obtain $\widetilde{f}:
\widetilde{X}' \to \widetilde{X}$, then the composite $f = p \circ
\widetilde{f}: X' = \widetilde{X}' \to X$ satisfies the claims of the
theorem, so we replace $X$ by $\widetilde{X}$.  We remark here that in
the case where $X$ is a smooth rigid analytic curve, there does not
exist a nonisomorphic proper birational morphism into $X$, so we have
$\widetilde{X} = X$ (and, later in the proof, $X'_0 = X$ and $X' = X$)
in this case.  It also suffices to treat connected components
separately.  So we assume henceforth that $X$ is normal and connected.
Note that every connected affinoid domain $\Max(A)$ in $X$ has
$\Spec(A)$ irreducible: using \cite[Lemma~2.1.1]{conrad} and
excellence to pass between completions, we see that $\Max(A)$ is
normal, and since it is also connected it is irreducible by
definition, so that \cite[Lemma~2.1.4]{conrad} shows that $\Spec(A)$
is irreducible.  It follows that any coherent sheaf on $X$, or its
pullback under any dominant morphism, has constant generic rank.

By Corollary~\ref{C:finite cohomology sheaf}, locally on $X$, say over
the affinoid subdomain $W \subseteq X$, the complex
$\rmC_{\varphi,\gamma_K}^\bullet(M^\vee(\delta))$
(resp. $\rmC_{\varphi,\gamma_K}^\bullet(M^\vee(\delta)/t_\sigma)$) is
quasi-isomorphic to some complex $A^\bullet$ of locally free sheaves
concentrated in degrees $[0,2]$ (resp. $[-1,2]$), and by
Theorem~\ref{T:base change} for any morphism $j: W' \to W$ one has
that $\rmC_{\varphi,\gamma_K}^\bullet(j^*M^\vee(\delta))$
(resp.\ $\rmC_{\varphi,\gamma_K}^\bullet(j^*M^\vee(\delta)/t_\sigma)$)
is represented by the na\"ive pullback $j^*A^\bullet$.  With this in
mind, by Corollary~\ref{C:flattening technique}(2), locally on $X$
there exists a proper birational morphism $f_0: X'_0 \to X$ such that
$N_0 = f_0^*(M^\vee(\delta))$ satisfies the conditions
\begin{itemize}
\item[(i)] $H^0_{\varphi, \gamma_K}(N_0)$ is flat and $H^i_{\varphi, \gamma_K}(N_0)$ has $\Tor$-dimension less than or equal to $1$ for each $i =1, 2$, and
\item[(ii)] for each $\sigma \in \Sigma$, $H^0_{\varphi, \gamma_K}(N_0/t_\sigma)$ is flat and $H^i_{\varphi, \gamma_K}(N_0/ t_\sigma)$ has $\Tor$-dimension less than or equal to $1$ for each $i=1,2$.
\end{itemize}
By Theorem~\ref{T:base change}(2) and the invariance under base change in Corollary~\ref{C:flattening technique}(2), the locally constructed morphisms glue to a proper birational morphism $f_0: X'_0 \to X$ locally satisfying the same conditions.
Note that $X'_0$ is connected; we may and will always take $X'_0$ to
be reduced and normal.

For a closed point $z \in X'_0$, the base change spectral sequence
$E_2^{i,j} = \Tor_{-i}^{X'_0}(H^j_{\varphi,\gamma_K}(N_0), \kappa_z)
\Rightarrow H^{i+j}_{\varphi, \gamma_K}(N_{0,z})$ gives a short exact
sequence (using that $H^1_{\varphi,\gamma_K}(N_0)$ has
$\Tor$-dimension at most $1$)
\begin{equation}
\label{E:local global comparison}
0 \to H^0_{\varphi, \gamma_K}(N_0) \otimes \kappa_z \to H^0_{\varphi, \gamma_K}(N_{0,z}) \to \Tor_1^{X'_0} (H^1_{\varphi, \gamma_K}(N_0), \kappa_z )\to 0.
\end{equation}
The condition (i) allows us to invoke Lemma~\ref{L:tor given by
  Fitting}, so that the set $Z'_0$ of $z \in X'_0$ for which the last
term of \eqref{E:local global comparison} does not vanish is Zariski
closed, and $X'_0-Z'_0$ is open and dense.  (Note that the formation
of $Z'_0$ is compatible with pullback because it is given by the zero
locus of a Fitting ideal.)  Thus for any point $z$ in the Zariski
dense subset $f_0^{-1}(X_\alg) \backslash Z'_0$ of $X'_0$,
$H^0_{\varphi, \gamma_K}(N_0) \otimes \kappa_z \cong H^0_{\varphi,
  \gamma_K}(N_{0,z})$ is one-dimensional.  Therefore, in view of
condition (i), $H^0_{\varphi, \gamma_K}(N_0)$ is a locally free sheaf
over $X'_0$ of rank one.  Let $\calL_0$ denote the dual line bundle of
$H^0_{\varphi, \gamma_K}(N_0)$.  Dualizing the natural homomorphism
$\calR_{X'_0}(\pi_K) \otimes_{X'_0} \calL_0^\vee\to N_0$ gives a
homomorphism $\lambda_0: f_0^*M \to \calR_{X'_0}(\pi_K) (\delta)
\otimes_{X'_0} \calL_0$ of $(\varphi, \Gamma_K)$-modules.

The uniqueness of $\lambda_0$ follows from the construction, and that $H_{\varphi,\gamma_K}^0(N_0)$ is a line bundle.  We next verify properties (1)--(2) for $X'_0$.  Afterwards we will construct morphism $f: X' \to X'_0 \to X$ by blowing up in $X'_0$, which will preserve the properties (1)--(2) shown over $X'_0$, and check (3) in this setting.

(1) Consider the locus $Z_0$ consisting of closed points $z \in X'_0$
that fail the condition that $\lambda_z: M_z \to
\calR_{\kappa_z}(\pi_K)(\delta_z)$ is surjective and the corresponding
element spans $H_{\varphi,\gamma_K}^0(M_z^\vee(\delta_z))$.  For each
$\sigma \in \Sigma$ we have a natural map $\chi_\sigma: H^0_{\varphi,
  \gamma_K}(N_0) \to H^0_{\varphi, \gamma_K}(N_0/t_\sigma)$.  By the
discussion above, \eqref{E:local global comparison} implies that
$H^0_{\varphi, \gamma_K}(N_{0,z})$ is one-dimensional if and only if
$z \notin Z'_0$.  In this case, Lemma~\ref{L:phi Gamma module
  saturated} implies that $\calR_{\kappa_z}(\pi_K) \to
M^\vee_z(\delta_z)$ is saturated, or equivalently $M_z \to
\calR_{\kappa_z}(\pi_K)(\delta_z)$ is surjective, if and only if for
each $\sigma \in \Sigma$ the map
\[
H^0_{\varphi, \gamma_K}(M^\vee_z (\delta_z)) \cong H^0_{\varphi, \gamma_K}(N_0) \otimes_{X'_0} \kappa_z \to H^0_{\varphi, \gamma_K}(N_0/t_\sigma) \otimes_{X'_0} \kappa_z \hookrightarrow H^0_{\varphi, \gamma_K}(M^\vee_z (\delta_z) / t_\sigma)
\]
is nontrivial.  Here, the injectivity of the last homomorphism above
follows from the base change spectral sequence $E_2^{i,j} =
\Tor_{-i}^{X'}(H^j_{\varphi, \gamma_K}(N_0 / t_\sigma), \kappa_z)
\Rightarrow H^{i+j}_{\varphi, \gamma_K}(N_{0,z}/t_\sigma)$ and the
fact that $H^1_{\varphi,\gamma_K}(N_0/t_\sigma)$ has $\Tor$-dimension
at most $1$.  It follows that $Z_0$ is the union of $Z'_0$ and the
subspaces where the respective $\chi_\sigma$ vanish, hence is Zariski
closed.  Finally, $f_0^*X_\alg$ is disjoint from $Z_0$ because our
hypotheses on each point of $X_\alg$ are simply a restatement of the
negation of the condition on membership in $Z_0$.

(2) Let $Q_0$ denote the cokernel of $\lambda_0$.  By the discussion in (1), $Q_{0,z} = 0$ if $z \notin Z_0$. By Lemma~\ref{L:local saturation=>global saturation}, this implies that $Q_0$ is supported on $Z_0$, in the sense that, locally on $X'_0$, for any analytic function $g$ vanishing along $Z_0$, some power of $g$ kills $Q_0$.

By \eqref{E:local global comparison}, $H^0_{\varphi, \gamma_K}(N_0) \otimes \kappa_z$ injects into $H^0_{\varphi, \gamma_K}(N_{0,z})$ and hence $\lambda_z$ is nontrivial.
Thus, the image of $\lambda_z$ is a $(\varphi, \Gamma_K)$-submodule of $\calR_{\kappa_z}(\pi_K)(\delta_z)$.  By Corollary~\ref{C:submodule of rank one module}, $Q_{0,z}$ is killed by some power of $t$.

We claim that, locally on $X'_0$, $Q_0$ itself is killed by some power of $t$.  So assume that $X'_0$ is affinoid.  We first observe that $\lambda_0$ is the base change of a homomorphism $\lambda_0^r: f_0^*M^r \to \calR_{X'_0}^r(\pi_K)(\delta) \otimes_{X'_0} \calL_0$ for some $r \in (0, C(\pi_K)]$.  Let $Q_0^r$ denote the cokernel of $\lambda_0^r$.  We first note that $Q_0^{[r/p,r]}$ is killed by some power of $t$, say $t^n$.  Indeed, $Q_0^{[r/p,r]}$ is supported on the zero locus of $t$ in $\Max(\calR_{X'_0}^{[r/p,r]}(\pi_K))$, which is an affinoid rigid analytic space, finite over $X_0'$.  This implies that $Q_0^{[r/p^m, r/p^{m-1}]} = (\varphi^m)^*Q_0^{[r/p,r]}$ is also killed by $t^n$ for every $m \in \NN$.  Hence $Q_0^r$ itself is killed by $t^n$, proving that $Q_0$ is killed by $t^n$.

To obtain the morphism $f: X' \to X$, affinoid-locally in $X'_0$ we apply Corollary~\ref{C:flattening technique}(1) to any finite presentation of $Q_0^{[r/p,r]}$ (which is a finite module by the preceding paragraph), and we glue these local constructions globally over $X'_0$ using the invariance under dominant base change provided by the corollary.  Write $\calL$, $\lambda$, $Q$, $Q^r$, and $Q^{[r,s]}$ to denote the respective pullbacks along $X' \to X'_0$ of $\calL_0$, $\lambda_0$, $Q_0$, $Q_0^r$, and $Q_0^{[r,s]}$.  By our construction of $X'$ the module $Q^{[r/p,r]}$ has $\Tor$-dimension at most $1$; pulling back by $\varphi$ shows that $Q^{[r/p^m,r/p^{m-1}]}$ also has $\Tor$-dimension at most $1$ for all $m \in \NN$; gluing shows that $Q^r$ has $\Tor$-dimension at most $1$.

(3) Recall that $Q$ is the cokernel of $\lambda$, and let $P$ denote the kernel of $\lambda$.  For any closed point $z \in X'$,  the $\Tor$ spectral sequence computing the cohomology of the complex $\big[M \xrightarrow\lambda \calR_{X'}(\pi_K)(\delta) \otimes_{X'} \calL\big] \stackrel{\bbL}\otimes_{X'} \kappa_z$ gives rise to the exact sequence
\[
0 \to \Tor_2^{X'}(Q,\kappa_z) \to P_z \to \Ker\big( \lambda_z: M_z \to \calR_{\kappa_z}(\pi_K)(\delta_z)\big) \to \Tor_1^{X'}(Q, \kappa_z) \to 0.
\]
By our construction of $X'$ the first term vanishes for all $z$, and
(2) implies that the last term in this sequence is killed by a power
of $t$.  In particular, $P_z$ is a $(\varphi,\Gamma_K)$-module over
$\calR_{\kappa_z}(\pi_K)$ of rank $d-1$.  For each affinoid subdomain
$\Max(A) \subseteq X$, choose a model $M^r$ for $M$ over
$\calR_A^r(\pi_K)$ for some $r$; by Lemma~\ref{L:extending isom and
  surj}, $\lambda$ arises from a map $\lambda^r: M^r \to
\calR_A^r(\pi_K)(\delta)$, giving rise to a model $P^r =
\Ker(\lambda^r)$ for $P$ over $\calR_A^r(\pi_K)$.  Each $P^r_z$ is a
coadmissible model for $P_z$ over $\calR_{\kappa_z}^r(\pi_K)$, and is
finite projective of rank $d-1$ by $\varphi$-equivariance.  Invoking
Lemma~\ref{L:constant dimension function implies flat}(2), we see that
$P^r$ is the global sections of a $\varphi$-bundle over
$\calR_A^r(\pi_K)$, hence by Proposition~\ref{P:phi module v.s. phi
  bundle} is finite projective of rank $d-1$.
\end{proof}

\begin{cor}
\label{C:global triangulation}
Let $X$ be a reduced rigid analytic space over $L$, all of whose connected components are irreducible.
Let $M$ be a densely pointwise strictly trianguline $(\varphi,\Gamma_K)$-module over $\calR_X(\pi_K)$ of rank $d$, with respect to the ordered parameters $\delta_1, \dots, \delta_d$ and the Zariski dense subset $X_\alg$.  Then there exist canonical data of
\begin{itemize}
\item[(a)] a proper birational morphism $f: X' \to X$ of reduced rigid analytic spaces,
\item[(b)] a unique increasing filtration $(\fil_i(f^*M))_{i=1, \dots, d}$ on the pullback $(\varphi, \Gamma_K)$-module $f^*M$ over $\calR_{X'}(\pi_K)$ via $(\varphi,\Gamma_K)$-stable coherent $\calR_{X'}(\pi_K)$-submodules,
\end{itemize}
such that the following conditions are satisfied.
\begin{itemize}
\item[(1)] The set $Z$ of closed points $z \in X'$ at which $(\fil_\bullet(f^*M))_z$ fails to be a strictly trianguline filtration on $M_z$ with ordered parameters $\delta_{1, z}, \dots, \delta_{d,z}$ is Zariski closed in $X'$ and disjoint from $f^{-1}(X_\alg)$ (hence the complement of $Z$ is Zariski open and dense).
\item[(2)] Each $\gr_i(f^*M)$ embeds $(\varphi, \Gamma_K)$-equivariantly into $\calR_{X'}(\pi_K)(\delta_i) \otimes_{X'} \calL_i$ for some line bundle $\calL_i$ over $X'$, and the cokernel of the embedding is, locally on $X'$, killed by some power of $t$ and supported on $Z$.
\item[(2')] The first graded piece $\gr_1(f^*M)$ is isomorphic to $\calR_{X'}(\pi_K)(\delta_1) \otimes_{X'} \calL_1$.
\end{itemize} 

Moreover, when $X$ is a smooth rigid analytic curve, we can take $X'$ equal to $X$.
\end{cor}
\begin{proof}
The existence of data satisfying all properties (1) and (2) follows from Theorem~\ref{T:global triangulation step 1} by induction, noting for (1) the equivalent description of a strict triangulation in Definition~\ref{D:strictly triangular}.

We now prove (2').  For this, we go into the proof of
Theorem~\ref{T:global triangulation step 1}.  The claim is local on
$X'$, so we replace $X'$ by a member of an admissible affinoid
covering such that $\calL_1$ is trivial and we have a natural
injective homomorphism $\lambda^r: \gr_1(f^*M^r) \to
\calR_{X'}^r(\pi_K)(\delta_1)$ of $(\varphi, \Gamma_K)$-modules over
$\calR_{X'}^r(\pi_K)$ for some $r \in (0, C(\pi_K)]$.  Let
  $\lambda^{[s,r]}$ denote its base change to $\calR_{X'}^{[s,r]}$ for
  any $s \in (0,r]$ and let $Q^{[s,r]}$ denote the cokernel of the
    latter; we must prove that the latter is zero.  Since $Q^{[s,r]}$ is killed by some power of $t$, it is a
    finite $\calO_{X'}$-module.  Tensoring $\lambda^{[s,r]}$ with
    $\kappa_z$ for $z \in X'$ gives an exact sequence
\begin{equation}\label{E:sequence for Q}
0 \to \Tor_1^{X'}(Q^{[s,r]}, \kappa_z)
\to \gr_1(f^*M^{[s,r]})_z \xrightarrow{\lambda_z} \calR^r_{\kappa_z}(\pi_K)(\delta_{1,z}) \to Q^{[s,r]}_z \to 0,
\end{equation}
where the first zero is by Corollary~\ref{C:R_A flat over A}.  Since
$\lambda_z$ is a homomorphism of finitely generated, torsion-free
modules over a finite product of B\'ezout domains and is nontrivial
generically, it must be injective, so that $\Tor_1^{X'}(Q^{[s,r]},
\kappa_z) = 0$.  On the other hand, since $Q^{[s,r]}$ is finitely
generated over $\calO_{X'}$, \cite[Exercise~6.2]{eisenbud} shows that
$Q^{[s,r]}$ is flat over $X'$.  However, this is only possible if
$Q^{[s,r]}$ is zero, because we know that $Q_z$ and hence
$Q_z^{[s,r]}$ is trivial at $z \notin Z$.
\end{proof}

\begin{remark}
Although the proper birational morphism $X' \to X$ appearing in Theorem~\ref{T:global triangulation step 1} is canonically constructed,
and its formation commutes with dominant, flat base change,
it is not clear to us if there is a \emph{universal} such $X'$ satisfying the conditions of 
Corollary~\ref{C:global triangulation}.

The ``bad locus'' $Z$ is necessary in the theorem; we will see this in
the example of the Coleman-Mazur eigencurve
(Proposition~\ref{P:eigencurve}).
\end{remark}

\begin{remark}
If $M$ is the base change of a $(\varphi, \Gamma)$-module over $\calR_X^{r_0}$, then the filtrations obtained in Corollary~\ref{C:global triangulation} can be taken to be base-changed from a filtration over $\calR_{X'}^{r_0}$, at least locally on $X'$.  This is because we may start by working with $M \otimes_{\calR_{X'}^{r_0}} \calR_{X'}$ and invoking Lemma~\ref{L:extending isom and surj}.
\end{remark}

Despite the fact that the global triangulation only behaves well away from the bad locus $Z$, we can show that actually at each closed point $z \in X$, the $(\varphi, \Gamma_K)$-module is trianguline (with slightly different parameters).  We thank Matthew Emerton for suggesting this application.

\begin{theorem}
\label{T:pointwise triangulation}
Let $X$ be a rigid analytic space over $L$.
Let $M$ be a densely pointwise strictly trianguline $(\varphi,\Gamma_K)$-module over $\calR_X(\pi_K)$ of rank $d$, with ordered parameters $\delta_1, \dots, \delta_d$ and the Zariski dense subset $X_\alg$.  Then for every $z \in X$, the specialization $M_z$ is trianguline with parameters $\delta'_{1, z}, \dots, \delta'_{d,z}$, where $\delta'_{i, z} = \delta_{i,z}\prod_{\sigma \in \Sigma} x_\sigma^{n_{i,z,\sigma}}$ for some $n_{i,z,\sigma} \in \ZZ$.
\end{theorem}
\begin{proof}
Fix $z \in X$.  After replacing $X$ by the reduced subspace of a
connected component of the normalization, we may assume that $X$ is
reduced and irreducible.  Applying Corollary~\ref{C:global
  triangulation} to $M$ gives a proper birational morphism $f: X' \to
X$ such that $f^*M$ admits a filtration of $(\varphi,
\Gamma_K)$-modules over $\calR_{X'}(\pi_K)$ as stated therein.  For
any point $z' \in f^{-1}(z)$ one has $(f^*M)_{z'} \cong M_z
\otimes_{\kappa_z} \kappa_{z'}$, so it suffices to replace $(X,M,z)$
by $(X',f^*M,z')$ in what follows.

Since the cokernel of the inclusion $\gr_i(M) \to \calR_X(\delta_i)$ is killed by some power of $t$, the cohomology groups of the complex
\[
[0 \to \fil_{i-1}(M)_z \to \fil_{i}(M)_z \to \calR_{\kappa_z}(\pi_K)(\delta_i) \to 0]
\]
are killed by some power of $t$.  The claim now follows from Corollary~\ref{C:submodule of rank one module}.
\end{proof}

\begin{example}
\label{Ex:main theorem rank 2}
We work out the argument of Theorem~\ref{T:pointwise triangulation} in detail for $\fil_2(M)$ when $X$ is a smooth (reduced with all connected components irreducible) rigid analytic curve.  This method can be extended to any $\fil_i(M)$, although the argument is more complicated.
In this case, Corollary~\ref{C:global triangulation} implies that we have the following exact sequence
\begin{equation}
\label{E:triangulation rank 2}
0 \to \calR_X(\pi_K)(\delta_1) \otimes_X \calL_1 \xrightarrow{\lambda} \fil_2(M) \xrightarrow{\mu} \calR_X(\pi_K)(\delta_2) \otimes_X \calL_2 \to Q \to 0,
\end{equation}
where the cokernel $Q$ is killed by some power of $t$ and is supported on a subset $Z \subset X$ disjoint from $X_\alg$ meeting each affinoid at finitely many points.  After shrinking $X$, we assume $Q$ is supported at a unique point $z$ of $Z$.

Tensoring \eqref{E:triangulation rank 2} with $\kappa_z$, we obtain
\begin{gather}
\label{E:triangulation rank 2-1}
\fil_2(M)_z \xrightarrow{\mu_z} \calR_{\kappa_z}(\pi_K)(\delta_{2,z}) \to Q_z \to 0, \textrm{ and}
\\
\label{E:triangulation rank 2-2}
0 \to \calR_{\kappa_z}(\pi_K)(\delta_{1,z}) \xrightarrow{\lambda_z} \Ker \big(\mu_z: \fil_2(M)_z \to \calR_{\kappa_z}(\pi_K)(\delta_{2,z})\big)
 \to \Tor_1^X(Q, \kappa_z) \to 0.
\end{gather}
Since the kernel of $\mu_z$ is torsion-free and has generic rank one, it is a $(\varphi, \Gamma_K)$-module over $\calR_{\kappa_z}(\pi_K)$.  Since it also contains $\calR_{\kappa_z}(\pi_K)(\delta_{1,z})$ as a subobject, Corollary~\ref{C:submodule of rank one module} implies that it is isomorphic to $\calR_{\kappa_z}(\pi_K)(\delta'_{1,z})$, where $\delta'_{1,z} = \delta_{1,z} \prod_{\sigma \in \Sigma} x_\sigma^{-k_{z,\sigma}}$ for some $k_{z,\sigma} \in \ZZ_{\geq 0}$.
In particular, this implies that  $\Tor_1^X(Q, \kappa_z)$ is isomorphic to $\calR_{\kappa_z}(\pi_K)(\delta'_{1,z}) /( \prod_{\sigma \in \Sigma}t_\sigma^{k_{z,\sigma}})$.  After shrinking $X$, we may assume that it is affinoid and its global sections are a PID, which shows that $Q_z$ must also be isomorphic to $\calR_{\kappa_z}(\pi_K)(\delta'_{1,z})/(\prod_{\sigma\in\Sigma}t_\sigma^{k_{z,\sigma}})$, and the gradeds of $Q$ for the $\gothm_z$-adic filtration (where $\gothm_z$ is the maximal ideal at $z$) must similarly have the form
\[
\gothm_z^nQ/\gothm_z^{n+1}Q \cong \calR_{\kappa_z}(\pi_K)(\delta'_{1,z}) \big/\big( \prod_{\sigma \in \Sigma}t_{\sigma}^{k_{z,\sigma,n}} \big)
\]
where for each fixed $\sigma \in \Sigma$ the $k_{z,\sigma,n}$ form a nonincreasing sequence of nonnegative integers that eventually vanish.

In view of \eqref{E:triangulation rank 2-1}, $Q_z$ is also isomorphic
to $\calR_{\kappa_z}(\pi_K)(\delta_{2,z}) / \big( \prod_{\sigma \in
  \Sigma}t_\sigma^{k_{z,\sigma}} \big)$.  In particular, for any
$\sigma \in \Sigma$ such that $k_{z,\sigma} >0$, we have a $(\varphi,
\Gamma)$-equivariant isomorphism
\[
\calR_{\kappa_z}(\pi_K)(\delta'_{1,z}) / (t_\sigma) \cong \calR_{\kappa_z}(\pi_K) (\delta_{2,z})/(t_\sigma).
\]
Note that these two objects are not genuine $(\varphi, \Gamma_K)$-modules, so in particular, the isomorphism does not provide us with any information about the $\varphi$-action.  Still, it gives an isomorphism
\[
\bbD_{\Sen, m}(\delta'_{1, z}) \otimes_{K \otimes_{\Qp} \kappa_z, \sigma} \kappa_z \cong \bbD_{\Sen, m}(\delta_{2, z}) \otimes_{K \otimes_{\Qp} \kappa_z, \sigma} \kappa_z,
\]
for some $m$ sufficiently large.  Looking at the action of
$\Theta_\Sen$ on both sides of the isomorphism,
Corollary~\ref{L:weight of rank one module} implies that \[
\mathrm{wt}_\sigma(\delta_{2,z}) = \mathrm{wt}_\sigma(\delta'_{1,z}) =
\mathrm{wt}_\sigma(\delta_{1,z}) - k_{z,\sigma}.
\]

In conclusion, for any $\sigma \in \Sigma$, either $k_{z,\sigma}$ is
zero or it is strictly positive and equal to
$\mathrm{wt}_\sigma(\delta_{1,z}) - \mathrm{wt}_\sigma(\delta_{2,z})$;
the local triangulation at $z$ is given by
\begin{equation}\label{E:conclusion}
0 \to \calR_{\kappa_z}(\pi_K)(\delta_{1,z}\prod_{\sigma\in \Sigma}
x_\sigma^{-k_{z,\sigma}}) \to \fil_2(M)_z \to
\calR_{\kappa_z}(\pi_K)(\delta_{2,z}\prod_{\sigma\in \Sigma}
x_\sigma^{k_{z,\sigma}}) \to 0.
\end{equation}
\end{example}

\subsection{Triangulation over eigenvarieties}

We apply the triangulation result Corollary~\ref{C:global triangulation} (resp.\ Theorem~\ref{T:pointwise triangulation}) to the case of eigenvarieties and obtain global (resp.\ pointwise) triangulation immediately.  For simplicity, we restrict ourselves to the case when $K = \Qp$, where $\Sigma = \{\mathrm{id}\}$, the $\mathrm{id}$-Hodge-Tate weight is just the usual Hodge-Tate weight, and we can take $t_\mathrm{id} = t$.

\begin{defn}
\label{D:refined representation}
Let $V$ be a $p$-adic representation of $G_{\Qp}$ over a finite extension $L$ of $\Qp$, and let $\delta_1, \ldots, \delta_d: \Qp^\times \to L^\times$ be potentially semistable characters.  We write $F_i = \delta_i(p)$, and let $\kappa_i \in \ZZ$ be the Hodge-Tate weight of $\delta_i$, so that $(\delta_i \cdot x^{\kappa_i})|_{\Zp^\times}$ has finite order, thus giving a character $\psi_i: G_{\Qp} \twoheadrightarrow \Gal(\Qp(\mu_{p^n})/\Qp) \cong (\ZZ/p^n)^\times \to L^\times$ for some $n.$  We say that $V$ is \emph{refined trianguline} with ordered parameters $\delta_1, \dots, \delta_d: \Qp^\times \to L^\times$ if
\begin{itemize}
\item[(a)] $V$ becomes semistable over some $\Qp(\mu_{p^n})$,
\item[(b)] $V$ has Hodge-Tate weights $\kappa_1 \leq \dots \leq \kappa_d$ and \emph{distinct} Frobenius eigenvalues $p^{\kappa_1}F_1, \dots, p^{\kappa_d}F_d$,
\item[(c)] the action of $G_{\Qp}$ on $\bbD_\pst(V)$ stabilizes each $p^{\kappa_i}F_i$-eigenspace and, on this line, is via $\psi_i$, and
\item[(d)] the complete flag on $\bbD_\pst(V)$ whose $i$th subobject
  is spanned by the eigenspaces for
  $p^{\kappa_1}F_1,\ldots,p^{\kappa_i}F_i$ is in general position with
  respect to the Hodge filtration, with weights
  $\kappa_1,\ldots,\kappa_i$, and each step is stable under the
  monodromy operator.
\end{itemize}
\end{defn}

\begin{lemma}
\label{L:refined family is strict trianguline}
If $V$ is refined trianguline with ordered parameters $\delta_1, \dots, \delta_d$, then $\bbD_\rig(V)$ is strictly trianguline with ordered parameters $\delta_1, \dots, \delta_d$.
\end{lemma}
\begin{proof}
By Berger's correspondence between filtered $(\varphi, N, G_{\Qp})$-modules and potentially semistable $(\varphi, \Gamma)$-modules described in \cite[Th\'eor\`eme~A]{berger-dictionary}, the complete flag described in Definition~\ref{D:refined representation}(d) gives a triangulation of $M = \bbD_\rig(V)$ with ordered parameters $\delta_1, \dots, \delta_d$.  Note that $M/\fil_iM(\delta_{i+1}^{-1})$ corresponds via Berger's correspondence to a filtered $(\varphi, N, G_{\Qp})$-module with a unique $\varphi$-eigenvalue equal to 1.  This shows that $H^0_{\varphi, \gamma_{\Qp}}(M/\fil_iM(\delta_{i+1}^{-1}))$ is at most one-dimensional, and there fore one-dimensional, so that $\bbD_\rig(V)$ is strictly trianguline.
\end{proof}

\begin{example}\label{E:refined family}
Let $X$ be a rigid analytic space over $\Qp$,
let $V$ be a family of $G_{\Qp}$-representations over $X$,
take $M = \bbD_\rig(V)$, and assume given continuous characters
$\delta_1,\ldots,\delta_d: \Qp^\times \to \Gamma(X, \calO_X)^\times$ and a
Zariski-dense subset $X_\alg \subset X$ satisfying the following
condition:
\begin{itemize}
\item For each $z \in X_\alg$ the fiber $V_z$ is refined trianguline
  with parameters $\delta_{1,z}, \dots, \delta_{d,z}$.
\end{itemize}
Putting $F_i = \delta_i(p)$ and $\kappa_i = -\mathrm{wt}(\delta_i)$,
so that $\kappa_1,\ldots,\kappa_d,F_1,\ldots,F_d \in
\Gamma(X,\calO_X)$, such data are almost the same as those employed by
Bella\"iche \cite[3.2.1]{R:Bellaiche} (or Bella\"iche-Chenevier
\cite[Chapter~4]{R:BellaicheChenevier}, where a family of
pseudocharacters is treated), except these authors have a stronger
denseness condition, and for simplicity they require the $V_z$ to be
crystalline.  By Lemma~\ref{L:refined family is strict trianguline},
the fiber $V_z$ for any $z \in X_\alg$ is strictly trianguline with
ordered parameters $\delta_{1,z},\ldots,\delta_{d,z}$.  Then
Corollary~\ref{C:global triangulation} (if $X$ is reduced and its
connected components are irreducible) and Theorem~\ref{T:pointwise
  triangulation} apply to this situation.
\end{example}

\begin{example} \label{exa:eigencurve}
We conclude this subsection by specializing to the case where
\begin{itemize}
\item $X$ is any disjoint union of irreducible subspaces of the
  Coleman-Mazur eigencurve and $V$ its family of
  $G_{\Qp}$-representations,
\item $X_\alg$ is the set of classical points of weight $k \geq 2$ not
  in the image of the $\theta^{k-1}$-map,
\item $\delta_1|_{\Zp^\times}$ is trivial and $\delta_2|_{\Zp^\times}$
  is the inverse of the ``weight-nebentypus character'' map,
\item $\kappa_1$ is the constant function $0$ and $\kappa_2$ measures
  one minus the weight, and
\item $F_1$ is the $U_p$-map and $F_2$ is the inverse of the $U_p$-map,
\end{itemize}
and we obtain the following more precise result.  If necessary, we may
enlarge each irreducible subspace of $X$ within the eigencurve so that
it encompases some point of classical weight.  Then Coleman's
classicality theorem implies that all nearby points of sufficiently
large classical weight are classical, and $X_\alg$ is indeed dense.
\end{example}

\begin{prop}
\label{P:eigencurve}
With the preceding notations, there exists a triangulation over the
entire desingularization $X'$ of $X$.  For any point $z \in X'$, the
global triangulation restricts to a triangulation of $\bbD_\rig(V_z)$
if and only if the overconvergent modular form $f$ corresponding to
$z$ is not in the image of the $\theta^{k-1}$-map for some integer $k
\geq 2$.  If $z$ corresponds to an overconvergent modular form of
integer weight $k \geq 2$ in the image of the $\theta^{k-1}$-map, the
triangulation extends to a submodule whose fiber at $z$ has index
$t^{k-1}$ in its saturation.
\end{prop}

\begin{proof}
The existence of global triangulation follows immediately from
Corollary \ref{C:global triangulation}.  We discuss the rest of the
statement case by case.

If $z \in X'$ corresponds to an overconvergent modular form with a
non-integer weight or with weight less than or equal to 1, then
$\mathrm{wt}_\mathrm{id}(\delta_{1,z}) -
\mathrm{wt}_\mathrm{id}(\delta_{2,z})$ is not a positive integer.  By
the calculation of Example~\ref{Ex:main theorem rank 2}, the global
triangulation is saturated at the point $z$.

If $z \in X'$ corresponds to an overconvergent modular form with
integer weight $k \geq 2$ that is not in the image of $\theta^{k-1}$,
\cite[Proposition~5.4.3]{breuil-emerton} implies that $z$ corresponds
to a classical modular form $f$ of weight $k$ and $U_p$-eigenvalue
$\alpha_z$, and \cite[Th\'eor\`eme 1.1.3]{breuil-emerton} implies that
the representation $V_z$ of $G_{\Qp}$ attached to $f$ is nonsplit.  In
the computations of Example~\ref{Ex:main theorem rank 2} the integer
$k_{z,\mathrm{id}}$ is either $0$ or $k-1$, and it suffices to show
that $k_{z,\mathrm{id}}=0$.  Weak admissibility of
$\bbD_\mathrm{pst}(V_z)$ implies that if $k_{z,\mathrm{id}} = k-1$,
then $\ord_p \alpha_z = k-1$, and it is easy to see that this
forces $\bbD_\cris(V_z)$ to be a direct sum of two weakly admissible
submodules, contradicting the fact that $V_z$ is not split.  It
follows that one must have $k_{z,\mathrm{id}}=0$.

Now suppose that $z \in X'$ corresponds to an overconvergent modular
form $f$ of weight $k \geq 2$ that is in the image of $\theta^{k-1}$,
i.e.\ there exists an overconvergent modular form $f'$ of weight $2-k$
such that $(q\frac d{dq})^{k-1}(f'(q)) =f(q)$.  If necessary, add to
$X$ an irreducible region of the eigencurve so that there exists $z'
\in X'$ corresponding to $f'$.  Let $V_z$ (resp. $V_{z'}$) be the
representation of $G_\QQ$ attached to $f$ (resp. $f'$); by looking at
the eigenvalues of Hecke operators away from $p$, we know that $V_z =
V_{z'} \otimes \chi^{1-k}$.  It is also clear that $\delta_{1,z'} =
\delta_{1,z} |x|^{k-1}$ and $\delta_{2,z'} = \delta_{2,z}
x^{2k-2}|x|^{k-1}$.  The computation in Example~\ref{Ex:main theorem
  rank 2}, gives us two integers $k_{z,\mathrm{id}}$ and
$k_{z',\mathrm{id}}$.  Because the weight at $z'$ is $2-k < 1$, the
first case we treated above shows that the global triangulation is
saturated at $z'$, i.e.\ $k_{z',\mathrm{id}} = 0$.  By the conclusion
of Example~\ref{Ex:main theorem rank 2}, $k_{z,\mathrm{id}}$ is equal
to $0$ or $k-1$.  We claim that $k_{z,\mathrm{id}} = k-1$, which would
finish the proof of this proposition.  Suppose not,
i.e.\ $k_{z,\mathrm{id}}=0$.  Let $\kappa$ be a finite extension of
$\Qp$ containing both $\kappa_z$ and $\kappa_{z'}$.  After base
changing to $\kappa$, the exact sequence \eqref{E:conclusion} for $z$
is
\begin{equation}
\label{E:eigencurve exact sequence3}
0 \to \calR_\kappa(\delta_{1,z}) \xrightarrow{\lambda_z} \bbD_\rig(V_z) \xrightarrow{\mu_z} \calR_{\kappa}(\delta_{2,z})  \to 0,
\end{equation}
whereas the sequence \eqref{E:conclusion} for $z'$, twisted by
$\chi^{1-k}$, is
\begin{equation}
\label{E:eigencurve exact sequence4}
0 \to t^{1-k} \calR_\kappa(\delta_{1,z}) \xrightarrow{\lambda_{z'}} \bbD_\rig(V_z) \xrightarrow{\mu_{z'}} t^{k-1} \calR_\kappa(\delta_{2,z}) \to 0.
\end{equation}
Note that $\mu_z$ in \eqref{E:eigencurve exact sequence3} does not
factor through $\mu_{z'}$ in \eqref{E:eigencurve exact sequence4}.
Thus, the composition $t^{1-k} \calR_\kappa(\delta_{1,z})
\xrightarrow{\lambda_{z'}} \bbD_\rig(V_z) \xrightarrow{\mu_z}
\calR_\kappa(\delta_{2,z})$ is nonzero.  By comparing weights, we
conclude that $x^{1-k} \delta_{1,z} = \delta_{2,z}$.  This implies
that $\alpha_z^2 = p^{k-1}$, where $\alpha_z$ is the $U_p$-eigenvalue
of $f$.  This is impossible because, writing $\alpha_z'$ for the
$U_p$-eigenvalue of $f'$, one has $\alpha_z = p^{k-1}\alpha_z'$ and so
$\ord_p \alpha_z \geq k-1$.
\end{proof}


\begin{thebibliography}{99}

\bibitem{ega3}
A.~Grothendieck, \textit{El\'ements de g\'eom\'etrie alg\'ebrique, III. \'Etude cohomologique des faisceaux coh\'erents, I.} \textit{Inst. Hautes \'Etudes Sci. Publ. Math.} No. \textbf{11}, 1961, 167 pp.

\bibitem{R:Bellaiche} J.~Bella\"iche, Ranks of Selmer groups in an
  analytic family, {\it Trans.~A.M.S.} {\bf 364} (2012), no. 9,
  4735--4761.

\bibitem{R:BellaicheChenevier} J.~Bella\"iche and G.~Chenevier,
  Families of Galois representations and Selmer groups.  {\it
    Ast\'erisque} {\bf 324} (2009).

\bibitem{bellovin} R.~Bellovin, $p$-adic Hodge theory in rigid
  analytic families, preprint, \texttt{arXiv:1306.5685v1}.

\bibitem{benois} D.~Benois, A generalization of Greenberg's
  $\mathcal{L}$-invariant, \textit{Amer. J. Math.} {\bf 133} (2011),
  no. 6, 1573--1632.

\bibitem{berger-intro} L. Berger, An introduction to the theory of $p$-adic representations, \textit{Geometric aspects of Dwork theory}, Vol. I, 255--292, de Gruyter, Berlin, 2004.

\bibitem{berger-dictionary}
L.~Berger, \'Equations diff\'erentielles $p$-adiques et $(\varphi,
N)$-modules filtr\'es. \cite{R:Ast}, 13--38.

\bibitem{R:Ast} L.~Berger, C.~Breuil, and P.Colmez (eds.),
  Repr\'esentations $p$-adiques de groupes $p$-adiques I:
  R\'epresentations galoisiennes et $(\varphi,\Gamma)$-modules.  {\it
    Ast\'erisque} {\bf 319} (2008).

\bibitem{R:Ast2} L.~Berger, C.~Breuil, and P.Colmez (eds.),
  Repr\'esentations $p$-adiques de groupes $p$-adiques I:
  R\'epresentations de $\mathrm{GL}_2(\mathbf{Q}_p)$ et
  $(\varphi,\Gamma)$-modules.  {\it Ast\'erisque} {\bf 330} (2010).

\bibitem{berger-colmez} L.~Berger and P.~Colmez, Familles de
  repr\'esentations de de~Rham et monodromie $p$-adique.
  \cite{R:Ast}, 303--337.

\bibitem{BGR} S.~Bosch, U.~G\"untzer, and R.~Remmert, {\it
  Non-Archimedean analysis: A systematic approach to rigid analytic
  geometry.}  Grundlehren der Mathematischen Wissenschaften {\bf 261}.

\bibitem{bourbaki}
N. Bourbaki, \textit{Espaces Vectoriels Topologiques}, reprint of the 1981 original, Springer,
Berlin, 2007.

\bibitem{breuil-emerton}
C.~Breuil and M.~Emerton,
Repr\'esentations $p$-adiques ordinares de $\mathrm{GL}_2(\Qp)$ et compatibilit\'e local-global.
  \textit{Ast\'erisque} {\bf 331} (2010), 235--295.

\bibitem{buzzard}
K.~Buzzard,
Eigenvarieties, in {\it $L$-functions and Galois representations}, 59--120, {\it London Math. Soc. Lecture Note Ser.}, {\bf 320}, Cambridge Univ. Press, Cambridge, 2007.

\bibitem{chenevier} G.~Chenevier, Sur la densit\'e des
  repr\'esentations cristallines du groupe de Galois absolu de $\Qp$,
  \textit{Math. Ann.} {\bf 355} (2013), no. 4, 1469--1525.

\bibitem{cherbonnier-colmez}
F.~Cherbonnier and P.~Colmez,
Th\'eorie d'Iwasawa des repr\'esentations p-adiques d'un corps local, {\it J. A.M.S.} {\bf 12} (1999) 241--268.

\bibitem{coleman}
R.~Coleman, Classical and overconvergent modular forms, {\it Invent. Math.} {\bf 124} (1996), no. 1--3, 215--241.

\bibitem{colmez-Iw} P.~Colmez, Th\'eorie d'Iwasawa de r\'epresentations
  de de~Rham d'un corps local, \textit{Ann. of Math.} (2) {\bf 148}
  (1998), no. 2, 485--571.

\bibitem{colmez} P.~Colmez, Repr\'esentations triangulines de
  dimension $2$.  \cite{R:Ast}, 213--258.

\bibitem{colmez-kirillov} P.~Colmez,
Repr\'esentations de $\mathrm{GL}_2(\Qp)$ et $(\varphi,\Gamma)$-modules.
\cite{R:Ast2}, 281--509.

\bibitem{colmez-mirabolique} P.~Colmez, $(\varphi,\Gamma)$-modules et
  repr\'esentations du miraboliquie de $\mathrm{GL})_2(\Qp)$.
  \cite{R:Ast2}, 61--153.

\bibitem{conrad}
B.~Conrad,
Irreducible components of rigid spaces, {\it Ann. Inst. Fourier (Grenoble)} {\bf 49} (1999), no. 2, 473--541.

\bibitem{crew}
R.~Crew, Finiteness theorems for the cohomology of an overconvergent isocrystal on a curve, \textit{Ann. Sci. \'Ecole Norm. Sup.} (4) {\bf 31} (1998), no. 6, 717--763. 

\bibitem{eisenbud}
D.~Eisenbud, \textit{Commutative algebra. With a view toward algebraic geometry}, \textit{Graduate Texts in Mathematics}, \textbf{150}, Springer-Verlag, New York, 1995.

\bibitem{fontaine}
J.-M.~Fontaine, Repr\'esentations $p$-adiques des corps locaux,
\textit{The Grothendieck Festschrift}, {\bf 2}, Birkh\"auser, Boston,
1991, 249--309.

\bibitem{fontaine-ouyang}
J.-M.~Fontaine and Y.~Ouyang, \textit{Theory of $p$-adic Galois
  representations}, book to appear, downloaded from
\url{http://staff.ustc.edu.cn/~yiouyang/galoisrep.pdf} on 6 August 2013.

\bibitem{greenberg} R.~Greenberg, Iwasawa theory and $p$-adic
  deformations of motives.  \textit{Motives (Seattle, WA, 1991)},
  \textit{Proc. Symp. Pure Math.}, {\bf 55}, Part 2, Amer. Math. Soc.,
  Providence, 1994, 193--223.

\bibitem{hartshorne}
R.~Hartshorne, \textit{Algebraic geometry}. \textit{Graduate Texts in Mathematics}, No. \textbf{52}. Springer-Verlag, New York-Heidelberg, 1977. xvi+496 pp.

\bibitem{hellmann} E.~Hellmann, Families of trianguline representations and finite slope spaces, preprint, \texttt{arXiv:1202.4408v1}.

\bibitem{herr}
L.~Herr, Sur la cohomologie galoisienne des corps $p$-adiques, \textit{Bull. Soc. Math. France} {\bf 126}
(1998), no. 4, 563--600.

\bibitem{herr2}
L.~Herr, Une approche nouvelle de la dualit\'e locale de Tate,
\textit{Math. Ann.} {\bf 320} (2001), no. 2, 307--337.

\bibitem{kedlaya-finite}
K.S.~Kedlaya, Finiteness of rigid cohomology with coefficients, \textit{Duke Mathematical Journal} \textbf{134} (2006), 15--97.

\bibitem{kedlaya-relative} 
K.S.~Kedlaya, Slope filtrations for relative Frobenius.  \cite{R:Ast}, 259--301.

\bibitem{kiran-liu} K.S.~Kedlaya and R.~Liu, On families of
  $(\varphi,\Gamma)$-modules.  {\it Algebra and Number Theory} {\bf 4}
  (2010), no.~7, 943--967.

\bibitem{R:Kisin} M.~Kisin, Overconvergent modular forms and the
  Fontaine--Mazur conjecture.  {\it Invent.~Math.} {\bf 153} (2003),
  no.~2, 373--454.

\bibitem{liu} R.~Liu, Cohomology and duality for
  $(\varphi,\Gamma)$-modules over the Robba ring, {\it
  Int. Math. Res. Not.} {\bf 2008}, no.~3.

\bibitem{R:Liu2} R.~Liu, Triangulation of refined families, preprint,
\texttt{arXiv:1202.2188v3}.

\bibitem{matsumura}
H.~Matsumura, \textit{Commutative ring theory}. Translated from the Japanese by M. Reid. Second edition. \textit{Cambridge Studies in Advanced Mathematics}, \textbf{8}. Cambridge University Press, Cambridge, 1989. xiv+320 pp.


\bibitem{nakamura}
K.~Nakamura,
Classification of two-dimensional split trianguline representations of $p$-adic fields, \textit{Compos. Math.} {\bf 145} (2009), no. 4, 865--914.

\bibitem{nakamura2}
K.~Nakamura,
Deformations of trianguline $\BB$-pairs and Zariski density of two dimensional crystalline representations.  Preprint.  \texttt{arXiv:1006.4891v4}.

\bibitem{nekovar}
J.~Nekov\'a\v{r}, Selmer complexes. {\it Ast\'erisque} {\bf 310} (2006).

\bibitem{pottharst1} J.~Pottharst, Analytic families of finite-slope
  Selmer groups.  \textit{Algebra and Number Theory} {\bf 7} (2013),
  no. 7, 1571--1612.

\bibitem{pottharst2} J.~Pottharst, Cyclotomic Iwasawa theory of
  motives.  Preprint.

\bibitem{raynaud}
M.~Raynaud,
Flat modules in algebraic geometry. 
{\it Compositio Math.} {\bf 24} (1972), 11–-31.

\bibitem{schneider}
P. Schneider, \textit{Nonarchimedean Functional Analysis}, Springer-Verlag, Berlin, 2002.

\bibitem{schneider-teitelbaum} P.~Schneider and J.~Teitelbaum, Algebras of $p$-adic
  distributions and admissible representation.  {\it Invent.\ Math.}
  {\bf 153} (2003), no.~1, 145--196.


\end{thebibliography}
\end{document}